\documentclass{amsart}
\usepackage{amssymb}
\usepackage{hyperref}
\usepackage{mathrsfs}

\usepackage{tikz}
\usetikzlibrary{cd}

\newcommand{\C}{\mathbb{C}}
\newcommand{\Z}{\mathbb{Z}}
\newcommand{\Q}{\mathbb{Q}}
\newcommand{\F}{\mathbb{F}}
\newcommand{\bk}{\Bbbk}
\newcommand{\Ga}{\mathbb{G}_{\mathrm{a}}}
\newcommand{\id}{\mathrm{id}}

\newcommand{\tFl}{\widetilde{\mathrm{Fl}}}
\newcommand{\Fl}{\mathrm{Fl}}
\newcommand{\Gr}{\mathrm{Gr}}
\newcommand{\tfS}{\widetilde{\mathfrak{S}}}
\newcommand{\fS}{\mathfrak{S}}
\newcommand{\rS}{\mathrm{S}}

\newcommand{\si}{{\frac{\infty}{2}}}
\newcommand{\IC}{\mathrm{IC}}
\newcommand{\ICFl}{\mathrm{IC}}
\newcommand{\ICGr}{\overline{\mathrm{IC}}}
\newcommand{\Waki}{\mathcal{W}}
\newcommand{\WFl}{\mathcal{W}}
\newcommand{\WGr}{\overline{\mathcal{W}}}
\newcommand{\Ras}{\mathrm{Ras}}
\newcommand{\Rasu}{\mathrm{Ras}_{\mathrm{u}}}

\newcommand{\cF}{\mathcal{F}}
\newcommand{\cG}{\mathcal{G}}
\newcommand{\cI}{\mathcal{I}}
\newcommand{\cJ}{\mathcal{J}}
\newcommand{\cK}{\mathcal{K}}
\newcommand{\cM}{\mathcal{M}}

\newcommand{\cP}{\mathcal{P}}

\newcommand{\pH}{{}^{\mathrm{p}} \hspace{-0.5pt} \mathcal{H}}
\newcommand{\Gaits}{\mathrm{Ga}^\si}

\newcommand{\Lie}{\mathrm{Lie}}

\newcommand{\fin}{{\mathrm{f}}}
\newcommand{\aff}{{\mathrm{aff}}}
\newcommand{\ext}{{\mathrm{ext}}}
\newcommand{\bX}{\mathbf{X}}
\newcommand{\bY}{\mathbf{Y}}
\newcommand{\fR}{\mathfrak{R}}
\newcommand{\fRs}{\mathfrak{R}_{\mathrm{s}}}

\newcommand{\oblv}{\mathrm{oblv}}
\newcommand{\Av}{\mathrm{Av}}

\newcommand{\bi}{\mathbf{i}}
\newcommand{\bj}{\mathbf{j}}

\DeclareMathOperator{\psdim}{psdim}

\DeclareMathOperator{\rank}{rank}
\DeclareMathOperator*{\colim}{\mathrm{colim}}

\newcommand{\simto}{\xrightarrow{\sim}}
\newcommand{\la}{\langle}
\newcommand{\ra}{\rangle}
\newcommand{\Hom}{\mathrm{Hom}}
\newcommand{\Ext}{\mathrm{Ext}}

\newcommand{\chr}{\mathrm{char}}

\newcommand{\Iw}{\mathrm{I}}
\newcommand{\Iwu}{\mathrm{I}_{\mathrm{u}}}
\newcommand{\siIw}{\mathrm{I}^{\si}}
\newcommand{\siIwu}{\mathrm{I}_{\mathrm{u}}^{\si}}
\newcommand{\Loop}{\mathcal{L}}
\newcommand{\Shv}{\mathsf{Shv}}
\newcommand{\Shvc}{\mathsf{Shv}_{\mathrm{c}}}
\newcommand{\Shvfg}{\mathsf{Shv}_{\mathrm{fg}}}
\newcommand{\Vect}{\mathsf{Vect}}
\newcommand{\Rep}{\mathsf{Rep}}
\newcommand{\Spec}{\mathrm{Spec}}
\newcommand{\Ind}{\mathrm{Ind}}
\newcommand{\Sat}{\mathrm{Sat}}

\newcommand{\DFl}{\Delta}
\newcommand{\NFl}{\nabla}
\newcommand{\DGr}{\overline{\Delta}}
\newcommand{\NGr}{\overline{\nabla}}

\newcommand{\pt}{\mathrm{pt}}

\newcommand{\scO}{\mathscr{O}}

\newcommand{\dualG}{\mathbf{G}^\vee}
\newcommand{\dualB}{\mathbf{B}^\vee}
\newcommand{\dualT}{\mathbf{T}^\vee}
\newcommand{\dualg}{\mathbf{g}^\vee}
\newcommand{\dualb}{\mathbf{b}^\vee}
\newcommand{\dualu}{\mathbf{u}^\vee}

\mathchardef\mhyphen="2D

\numberwithin{equation}{section}
\newtheorem{thm}{Theorem}[section]
\newtheorem{lem}[thm]{Lemma}
\newtheorem{prop}[thm]{Proposition}
\newtheorem{cor}[thm]{Corollary}

\theoremstyle{definition}

\theoremstyle{remark}
\newtheorem{rmk}[thm]{Remark}
\newtheorem{ex}[thm]{Example}

\title{Semiinfinite sheaves on affine flag varieties}

\author{Pramod N. Achar}
\address{Department of Mathematics\\
  Louisiana State University\\
  Baton Rouge, LA 70803\\
  U.S.A.}
\email{pramod@math.lsu.edu}

\author{Gurbir Dhillon}
\address{UCLA Mathematics Department, Los Angeles, CA 90095-1555, USA.}
\email{gsd@math.ucla.edu}

\author{Simon Riche}
\address{Universit\'e Clermont Auvergne, CNRS, LMBP, F-63000 Clermont-Ferrand, France.}
\email{simon.riche@uca.fr}

\thanks{P.A. was supported by NSF Grant No.~DMS-2202012. G.D. was supported by an NSF Postdoctoral Fellowship under grant No.~DMS-2103387. This project has received
funding from the European Research Council (ERC) under the European Union's Horizon 2020
research and innovation programme (S.R., grant agreement No.~101002592).
}

\begin{document}

\begin{abstract}
We study a category of semiinfinite sheaves on the affine flag variety of a connected reductive algebraic group $G$, with coefficients in a field $\bk$ (of arbitrary characteristic different from that of the base field), generalizing some results of Gaitsgory and showing that this category behaves like familiar categories of sheaves on flag varieties in many respects. Our interest is motivated by expected relations with representation theory of the Lie algebra of the Langlands dual reductive group over $\bk$.
\end{abstract}

\maketitle


\section{Introduction}

This paper is the first in a series dedicated to the study of categories of semiinfinite constructible sheaves on the affine flag variety of a connected reductive algebraic group with coefficients in a field of positive characteristic, and of their applications to various questions related to the representation theory of the Langlands dual Lie algebra.

\subsection{Sheaves on semiinfinite affine flag varieties vs.~semiinfinite sheaves on affine flag varieties}
\label{ss:intro-si-sheaves}

Constructing geometric models for categories of representations of reductive algebraic groups or related objects has been a powerful tool to attack problems in Representation Theory. Classically this strategy has been mainly used for problems involving base fields of characteristic $0$, but it has also been successfully used more recently in some settings involving fields of positive characteristic. Here the ``geometric model'' sometimes involves categories of coherent sheaves, but in the present paper we are more specifically interested in models involving categories of \emph{constructible} sheaves. For the case of categories of representations of Frobenius kernels of reductive groups (or, in other words, restricted modules for their Lie algebras), or representations of the associated small quantum groups, it has been long expected that the adequate geometry should be that of the \emph{semiinfinite affine flag variety} (as defined by Fe{\u\i}gin--Frenkel~\cite{ff}) of the Langlands dual reductive group; see in particular~\cite[\S 11]{lusztig-ICM} and~\cite[\S\S 1.3--1.5]{fm}. 

A difficulty in making this idea precise is that the semiinfinite affine flag variety is easy to define as a set, but equipping it with some algebro-geometric structure suitable for studying sheaves is a highly nontrivial task, because this object is necessarily infinite-dimensional in a strong sense. Various approaches to this problem have been proposed over the years, either by following the original definition~\cite{kato-ICM} or by working with ``finite-dimensional approximations'' of this space, see e.g.~\cite{fm, abbgm}, or~\cite{bouthier} in a more general setting. But all of them are quite technical, and definitely not as straightforward as a (derived) category of sheaves on an algebraic variety.

A different approach to this problem has been suggested (for different purposes) by Gaitsgory in~\cite{gaitsgory}: instead of considering a category of Iwahori-equivariant constructible sheaves on the semiinfinite affine flag variety, i.e.~the quotient of the loop group by a ``semiinfinite Iwahori subgroup,'' he proposes to study instead a category of sheaves on the affine flag variety (i.e.~the quotient of the loop group by the Iwahori subgroup) which are equivariant for the semiinfinite Iwahori subgroup. It turns out that using an appropriate $\infty$-categorical framework, the definition of the latter category is not particularly complicated, and quite similar to that of usual categories of equivariant sheaves on a space. (Objects of such categories are what we call ``semiinfinite sheaves.'')

\begin{rmk}
 A different constructible model for some categories of representations of Lie algebras of reductive groups in terms of sheaves on the affine Grassmannian of the Langlands dual group (inspired by the first part of~\cite{abbgm}) has been studied by the first and third authors in~\cite{ar-model}. This model does not involve the semiinfinite affine flag variety, but it has not led to much progress on the understanding of the combinatorics of representations yet.
\end{rmk}

\subsection{Structure of \texorpdfstring{$\infty$}{infinity}-categories of semiinfinite sheaves}
\label{ss:structure-oo-cat}

As explained above, $\infty$-categories of semiinfinite sheaves on affine Grassmannians of reductive algebraic groups are introduced and studied, in the setting of characteristic-$0$ coefficients, by Gaitsgory in~\cite{gaitsgory}. (Gaitsgory treats in parallel the cases of \'etale sheaves, sheaves for the analytic topology, and $\mathcal{D}$-modules. As we are interested in more general coefficients, on our side only the first two settings will be considered.) The definition in fact makes sense for arbitrary partial affine flag varieties, and arbitrary fields of coefficients, and the main goal of this paper is to set out this theory in this more general framework. (More specifically, only the cases of affine flag varieties and affine Grassmannians will be treated in detail, but coefficients can be general; it should be clear that other partial affine flag varieties can be studied by similar methods.)

Consider a connected reductive algebraic group $G$ over an algebraically closed base field $\F$, with a choice of Borel subgroup $B \subset G$ and maximal torus $T \subset B$. We assume either that $\F=\C$, in which case we will consider sheaves for the analytic topology, with coefficients in an arbitrary field $\bk$, or otherwise we work with \'etale sheaves; in this case we choose a prime number $\ell$ which is invertible in $\F$, and a field of coefficients $\bk$ which is either a finite extension or an algebraic closure of $\F_\ell$ or $\Q_\ell$. We then have the loop group $\Loop G$ of $G$ (a group ind-scheme over $\F$), the arc group $\Loop^+ G$ (a group scheme over $\F$), and the standard Iwahori subgroup $\Iw$ given by the preimage of $B$ in $\Loop^+ G$, from which we obtain the affine Grassmannian
\[
\Gr=\Loop G / \Loop^+ G
\]
and the affine flag variety
\[
\Fl=\Loop G / \Iw.
\]
(Both are ind-schemes of ind-finite type over $\F$.) The \emph{semiinfinite Iwahori subgroup} is the group ind-scheme
\[
 \siIw = \Loop^+ T \ltimes \Loop U,
\]
where $U$ is the unipotent radical of $B$. 

Following Gaitsgory we consider the $\infty$-categories $\Shv(\siIw \backslash \Fl)$ and $\Shv(\siIw \backslash \Gr)$ of $\siIw$-equivariant $\bk$-sheaves on $\Fl$ and $\Gr$ respectively; these are full subcategories in the $\infty$-categories $\Shv(T \backslash \Fl)$ and $\Shv(T \backslash \Gr)$ respectively of $T$-equivariant complexes. (See~\S\ref{ss:si-sheaves-Fl} for details.) In $\Shv(\siIw \backslash \Fl)$ we have standard and costandard objects parametrized by the extended affine Weyl group $W_\ext = W_\fin \ltimes X_*(T)$ (where $W_\fin$ is the Weyl group of $(G,T)$ and $X_*(T)$ is the cocharacter lattice of $T$), and in $\Shv(\siIw \backslash \Gr)$ we have similar collections of objects parametrized by $X_*(T)$. Still following Gaitsgory, we explain in~\S\ref{ss:perv-semiinf-def} how to define ``perverse t-structures'' on these $\infty$-categories,\footnote{Since the $\siIw$-orbits on $\Fl$ and $\Gr$ are infinite-dimensional, this definition is not a direct application of the standard theory of perverse sheaves. It is more ad-hoc, but seems to be the right theory to relate these constructions to others appearing in Geometry or Representation Theory. See also~\cite{bkv} for a general study of perverse t-structures in such contexts.} whose hearts\footnote{More generally, below, a superscript ``$\heartsuit$'' always means the heart of a t-structure (which should be clear from the context). Recall that the heart of a t-structure is equivalent to the nerve of an abelian ($1$-)category.} are denoted $\Shv(\siIw \backslash \Fl)^\heartsuit$ and $\Shv(\siIw \backslash \Gr)^\heartsuit$. We also show that the standard objects belong to the heart of the perverse t-structure (both on $\Fl$ and $\Gr$, see Lemma~\ref{lem:dsi-heart}), that on $\Fl$ the costandard objects also belong to $\Shv(\siIw \backslash \Fl)^\heartsuit$ (see Proposition~\ref{prop:nsi-heart}), and that the simple objects in $\Shv(\siIw \backslash \Fl)^\heartsuit$ and $\Shv(\siIw \backslash \Gr)^\heartsuit$ are parametrized by $W_\ext$ in the case of $\Fl$ (see Corollary~\ref{cor:classif-simples-Fl}), and by $X_*(T)$ in the case of $\Gr$ (see Proposition~\ref{prop:si-simple}). In fact, in the latter case the simple objects have an explicit description: they coincide with the standard objects.

Finally, 
in the last two sections of the paper we study in more detail some important objects of $\Shv(\siIw \backslash \Gr)$ and $\Shv(\siIw \backslash \Fl)$. First, in Section~\ref{sec:dualizing} we show that the dualizing complexes $\omega_{\Gr}$ and $\omega_{\Fl}$ are concentrated in degree $-\infty$ for the perverse t-structures.
Then in Section~\ref{sec:Gaits-sheaf}
we define and study the counterpart in our setting of an object in $\Shv(\siIw \backslash \Gr)$ which is the central player in~\cite{gaitsgory}. Gaitsgory calls this object the ``semiinfinite intersection cohomology sheaf,'' but we rather call it the ``Gaitsgory sheaf,'' and denote it $\Gaits$. This object is \emph{not} a simple object, but it is interesting and useful nevertheless. We compute in Theorem~\ref{thm:properties-Ga} the dimensions of its stalks and costalks, which are expressed in terms of Kostant's partition function. (For the computation of costalks, our proof requires the characteristic of $\bk$ to be good for $G$.) In particular, these computations show that this object is ``independent of $\bk$'' in an appropriate sense. They play a crucial role in the companion paper~\cite{adr}.

All these constructions and statements have analogues in the setting where the group $\Iw$ is replaced by its pro-unipotent radical $\Iwu$ (i.e.~the preimage of $U$ in $\Loop^+ G$) and/or $\siIw$ is replaced by the subgroup $\siIwu$ (defined as the kernel of the composition of natural morphisms $\siIw \to \Loop^+ T \to T$), which we also study in parallel.

 
Many of the results presented above resemble standard properties of perverse sheaves on (possibly affine) flag varieties. There are however important differences between the setting we consider here and the more standard setting of constructible sheaves on affine flag varieties, the most serious of which is that there is no Grothendieck--Verdier duality\footnote{See Remark~\ref{rmk:simples-Gr}\eqref{it:Verdier-duality} for a precise statement justifying this assertion.} on the $\infty$-categories $\Shv(\siIwu \backslash \Gr)$ or $\Shv(\siIwu \backslash \Fl)$. As a consequence, stalks and costalks on the one hand, and standard and costandard objects on the other hand, play rather different roles. In particular, standard objects are compact (in the $\siIwu$-equivariant setting) while costandard objects are not, and the fact that simple objects in the heart of the perverse t-structure on $\Shv(\siIw \backslash \Gr)$ are the standard objects does \emph{not} imply that they also coincide with costandard objects.

 
\subsection{Raskin's equivalences}
\label{ss:intro-Raskin-equiv}

Most of the results stated in~\S\ref{ss:structure-oo-cat} are generalizations of statements proved by Gaitsgory in~\cite{gaitsgory}.
However, for some of them the proofs require different arguments. Indeed, Gaitsgory's approach relies on some relations between the semiinfinite categories of sheaves on $\Gr$ and other categories which have been studied independently, namely constructible sheaves on Drinfeld's compactifications and coherent sheaves on appropriate versions of the Springer resolution. These tools have no counterparts
for positive-characteristic coefficients as of now (at least, no counterpart that can be used in our study), and our methods are intrinsic to constructible sheaves on affine flag varieties.

In fact we use in a more systematic way some properties of objects arising in the construction of the Gaitsgory sheaf (see~\S\ref{ss:some-perv-sheaves}), and a general construction due to Raskin (see~\S\ref{ss:raskin-equiv}) which provides equivalences relating our $\infty$-categories of semiinfinite sheaves with more familiar $\infty$-categories of $\Iw$-equivariant (or $\Iwu$-equivariant) complexes on $\Fl$ (or $\Gr$, or variants).
As shown in Proposition~\ref{prop:ras-waki}, these equivalences match standard objects in semiinfinite sheaves with the familiar \emph{Wakimoto sheaves} in the Iwahori-equivariant categories, which can be used to deduce a number of their properties from known properties of the latter objects. Here also a difference between standard and costandard objects appears: whereas the objects corresponding to standard objects under Raskin's equivalences are classical objects, those corresponding to costandard objects (which we also describe in Corollary~\ref{cor:ras-waki-dual}, copying another construction of Gaitsgory, this time from~\cite{gaitsgory-ext}) do not have finite-dimensional support, hence cannot be seen in the usual constructible setting.

This construction also allows us to relate our perverse t-structure with another t-structure that has been considered in Geometric Representation Theory, namely
the ``new'' t-structure on the constructible derived category of $\Iwu$-equivariant sheaves on $\Fl$ studied by Bezrukavnikov--Lin in~\cite{bl}; see~\S\ref{ss:comparison-new-tstr} for details. (More specifically, Bezrukavnikov--Lin work in the setting of $\ell$-adic sheaves. Our construction is independent of the choice of coefficients, and therefore provides an extension of their constructions to arbitrary characteristics.)

Let us note that in Section~\ref{sec:dualizing} we prove a property of the Raskin equivalences (for sheaves on $\Gr$ and $\Fl$) that might be of independent interest, namely that they have bounded cohomological amplitude; in particular, this implies that the perverse t-structure and the ``new'' t-structure only differ by some finite amount. A similar result, in the setting of (critically twisted) $D$-modules of $\Fl$ (and without imposing $\Iw$-equivariance) appears in~\cite[\S 2.3]{fg}, but our proof is different.

\subsection{Relation with representation theory in positive characteristic}
\label{ss:relation-gB}

Our first application of the constructions described above is the description of the cohomology of stalks of intersection cohomology sheaves on Drinfeld's compactifications and Zastava schemes in case $\bk$ has good characteristic, which is explained in the companion paper~\cite{adr}. (The main ingredient used for that is the description of costalks of the Gaitsgory sheaf.)
In this subsection and the next one we explain some other expected applications of these constructions, which formed our main initial motivation to dive into this subject, and will be the subject of future work.

We assume that $\mathrm{char}(\bk)>0$, and
let $\dualG$ be the split connected reductive algebraic group over $\bk$ whose Frobenius twist $(\dualG)^{(1)}$ is the Langlands dual group $G^\vee_\bk$ of $G$ over $\bk$. Let also $\dualg$ be the Lie algebra of $\dualG$, and $\dualB$, $\dualT$ be the Borel subgroup and maximal torus whose Frobenius twists are the canonical Borel subgroup $B^\vee_\bk$ and maximal torus $T^\vee_\bk$ in $G^\vee_\bk$ produced from our data. Consider the $\infty$-category $(\dualg,\dualB)\mhyphen\mathsf{Mod}$ of $(\dualg,\dualB)$-modules, i.e.~strongly $\dualB$-equivariant $\mathcal{U}\dualg$-modules, or in other words $\dualB$-equivariant $\mathcal{U}\dualg$-modules on which the differential of the $\dualB$-action coincides with the restriction of the $\mathcal{U}\dualg$-action to $\mathcal{U}\dualb$, where $\dualb$ is the Lie algebra of $\dualB$. (The ordinary abelian category of finitely generated modules in this category is the ``modular category $\mathcal{O}$'' considered e.g.~in~\cite{losev,losev2,rsi}. We will not discuss the details of the $\infty$-categorical structure that needs to be considered here.) This $\infty$-category in particular contains the Verma modules $\mathsf{M}(\lambda) = \mathcal{U}\dualg \otimes_{\mathcal{U}\dualb} \bk_\lambda$ for $\lambda \in X^*(\dualT)$.

Assuming that $\bk$ is an algebraic closure of $\F_\ell$ with $\ell \geq h$ (where $h$ is the Coxeter number of $\dualG$, or equivalently of $G$) one can consider the ``extended block'' $(\dualg,\dualB)\mhyphen\mathsf{Mod}_0$ of $(\dualg,\dualB)\mhyphen\mathrm{Mod}$ containing the trivial module $\bk$; it also contains all the Verma modules $\mathsf{M}(w \bullet 0)$ for $w \in W_\ext$, where ``$\bullet$'' is the standard ``dot-action'' of $W_\ext$ on $X^*(\dualT)$, which involves the identification $X_*(T)=X^*(T^\vee_\bk)$ and the morphism $X^*(T^\vee_\bk) \to X^*(\dualT)$ given by pullback along the Frobenius morphism.

In future work, building on the results of~\cite{bezr2}, we will show that there exists a t-exact equivalence of $\infty$-categories
\begin{equation}
\label{eqn:equiv-gBmod}
 (\dualg,\dualB)\mhyphen\mathsf{Mod}_0 \simto \Shv(\siIw \backslash \Loop G / \Iwu)
\end{equation}
which sends the Verma module $\mathsf{M}(w \bullet 0)$ to the standard perverse sheaf associated with $w_\circ w$, for any $w \in W_\ext$. (Here $w_\circ$ is the longest element in $W_\fin$.)
%
%
Note that this result is related to, but not quite in line with, the expectations 
from~\cite{fm} or the results from~\cite{abbgm}, which involve the principal block $\Rep_0(\dualG_1 \dualT)$ of the category of $\dualG_1 \dualT$-modules (where $\dualG_1$ is the Frobenius kernel of $\dualG$) rather than $(\dualg,\dualB)$-modules. One possible explanation for this discrepancy is the fact that, contrary to the constructions in~\cite{fm,abbgm}, our category $\Shv(\siIw \backslash \Loop G / \Iwu)$ does not involve any factorization condition. We expect that restoring this factorization condition in the appropriate way should allow us to reconstruct the category $\Rep_0(\dualG_1 \dualT)$. (A more precise description of this will be discussed in future work.) In any case, at the combinatorial level this difference is (to some extent) irrelevant, because there exist natural t-exact functors
\[
 (\dualg,\dualB)\mhyphen\mathsf{Mod} \leftarrow \Rep(\dualG_1 \dualB) \to \Rep(\dualG_1 \dualT)
\]
which identify simple objects in the hearts of the natural t-structures in the three $\infty$-categories.

The category $\Shv(\siIw \backslash \Loop G / \Loop^+ G)$ should also have a natural interpretation in this picture. Namely, we expect that there exists an equivalence of abelian categories
\begin{equation}
\label{eqn:equiv-Gr-Bmod}
\Shv(\siIw \backslash \Loop G / \Loop^+ G)^\heartsuit \cong \Rep((\dualB)^{(1)})^\heartsuit,
\end{equation}
where the right-hand side is the abelian category of representations of the Frobenius twist $(\dualB)^{(1)}$ of $\dualB$,
such that the normalized pullback functor
\[
\Shv(\siIw \backslash \Loop G / \Loop^+ G)^\heartsuit \to \Shv(\siIw \backslash \Loop G / \Iwu)^\heartsuit 
\]
corresponds to the functor
\[
\Rep((\dualB)^{(1)})^\heartsuit \to \bigl( (\dualg,\dualB)\mhyphen\mathsf{Mod} \bigr)^\heartsuit
\]
sending a $(\dualB)^{(1)}$-module $M$ to the $(\dualg,\dualB)$-module given by $M$, with trivial action of $\mathcal{U}\dualg$, and $\dualB$-action obtained from the original $(\dualB)^{(1)}$-action via pullback along the Frobenius morphism.

We expect that these constructions will shed light on the combinatorics of representations of $\dualG_1$, following e.g.~\cite{rw}.


\begin{rmk}
\begin{enumerate}
 \item 
 We emphasize that~\eqref{eqn:equiv-Gr-Bmod} should only be an equivalence of \emph{abelian} categories, and is not expected to be true at the derived level. In fact the $\infty$-category $\Shv(\siIw \backslash \Loop G / \Loop^+ G)$ should rather have an interpretation in terms of ind-coherent sheaves on the quotient $((\dualu)^{(1)} \times_{(\dualg)^{(1)}} \{0\})/(\dualB)^{(1)}$ where $\dualu$ is the Lie algebra of the unipotent radical of $\dualB$ and the fiber product is taken in the derived sense; see~\cite[Theorem~4.2.3]{gaitsgory} for a version of this claim in characteristic $0$.
 \item
 The image under~\eqref{eqn:equiv-Gr-Bmod} of the object $\Gaits$ 
should be the $(\dualB)^{(1)}$-module $\scO((\dualB)^{(1)}/(\dualT)^{(1)})$. Note that in Remark~\ref{rmk:Gaits-injective-siIw} we explain that the object $\Gaits \in \Shv(\siIw \backslash \Gr)^{\heartsuit}$ is injective, which is consistent with this expectation.
\end{enumerate}
\end{rmk}

\subsection{A geometric realization of periodic Kazhdan--Lusztig polynomials}


Consider the Hecke algebra $\mathcal{H}_\ext$ associated with $W_\ext$, and the associated periodic Kazhdan--Lusztig polynomials $(p_{A,B} : A,B \in \mathcal{A})$ where $\mathcal{A}$ is the set of alcoves in $\mathbb{R} \otimes_{\Z} X_*(T)$, see~\cite{lusztig-generic}; here we follow the conventions of~\cite[\S 4]{soergel}.
Denote by $A_0 \in \mathcal{A}$ the fundamental alcove,
and consider the simple objects $(\IC^\si_w : w \in W_\ext)$ in the heart of the perverse t-structure on $\Shv(\siIwu \backslash \Fl)$.\footnote{Here, compared to the notation used in the body of the paper, we omit the ``forgetful'' functor $\oblv^{\siIw | \siIwu}$.} For $y \in W_\ext$ we denote by $\bi_y : \fS_y \to \Fl$ the embedding of the $\siIw$-orbit parametrized by $y$. As for $\Fl$ one can consider the category $\Shv(\siIwu \backslash \fS_y)$ of $\siIwu$-equivariant sheaves on $\fS_y$, and there exists an equivalence of $\infty$-categories
\[
 \Shv(\siIwu \backslash \fS_y) \cong \Vect_\bk
\]
where the right-hand side is the $\infty$-category of $\bk$-vector spaces. (See~\S\ref{ss:si-orbits} for the details of our normalization of this equivalence.) Given a complex $\cF$ in $\Shv(\siIwu \backslash \fS_y)$ and $n \in \Z$, we denote by $\rank \pH^n(\cF)$ the dimension of the degree-$n$ cohomology of the corresponding complex of vector spaces.

With this notation, in later work we will prove that, assuming that $\mathrm{char}(\bk)$ is either $0$ or a large\footnote{How ``large'' this number should be is related to the validity of Lusztig's conjecture on characters of simple $\dualG$-modules, where we use the notation of~\S\ref{ss:relation-gB}. We do not expect to be able to say anything concrete on this question beyond what is known from other works.} prime number, for any $w,y \in W_\ext$ we have
\begin{equation}
\label{eqn:stalks-IC-per}
 \sum_{n \in \Z} \rank \bigl( \pH^n((\bi_y)^* \IC^{\si}_w) \bigr) \cdot v^{-n} = p_{yw_\circ(A_0),ww_\circ(A_0)}(v).
\end{equation}

\begin{ex}
 As an example, in case $w=t_\lambda w_\circ$ for some $\lambda \in X_*(T)$, by Lemma~\ref{lem:pullback-simples} we have $\IC^{\si}_{t_\lambda w_\circ} = p^{\dag}(\DGr^{\si}_\lambda)$, where $p^\dag : \Shv(\siIwu \backslash \Gr) \to \Shv(\siIwu \backslash \Fl)$ is the normalized pullback functor and $\DGr^{\si}_\lambda$ is the standard object in $\Shv(\siIwu \backslash \Gr)$ associated with $\lambda$.
 This implies that
\[
 \rank \pH^n(\bi_y^* \IC^{\si}_w) = \begin{cases}
                        1 & \text{if $y=t_\lambda x$ for some $x \in W_\fin$ and $n=-\ell(x w_\circ)$;} \\
                        0 & \text{otherwise.}
                       \end{cases}
\]
On the other hand we have
\[
 p_{B,\lambda + A_0} = \begin{cases}
              v^{\ell(x)} & \text{if $B=\lambda + x(A_0)$ with $x \in W_\fin$;} \\
              0 & \text{otherwise,}
             \end{cases}
\]
which verifies~\eqref{eqn:stalks-IC-per} in this case.
\end{ex}

\begin{rmk}
If we combine the conjectural equality~\eqref{eqn:stalks-IC-per} with the equivalence~\eqref{eqn:equiv-gBmod}, one obtains a formula for the dimension of some morphism spaces in the derived category of $(\dualg,\dualB)$-modules. This formula is proved directly and using different tools by Q.~Situ and the third author in~\cite{rsi}.
\end{rmk}

\subsection{Contents}

In Section~\ref{sec:sheaves-Fl} we explain the definition of the various categories of sheaves we will work with, and the construction of an important equivalence of $\infty$-categories due to Raskin. In Section~\ref{sec:Waff} we recall the combinatorics of various families of orbits on $\Fl$ and $\Gr$. In Section~\ref{sec:complexes-I-orbits} we recall the definition and main properties of some families of complexes of sheaves attached to orbits of the Iwahori subgroup on $\Fl$ and $\Gr$. In Section~\ref{sec:st-cost-semiinf} we define and study the standard and costandard objects in the categories of semiinfinite sheaves. In Section~\ref{sec:perv-t-str-si} we define the perverse t-structure on categories of semiinfinite sheaves, and prove its main properties. 
In Section~\ref{sec:dualizing} we show that the dualizing complexes of $\Gr$ and $\Fl$ are in degree $-\infty$ for the perverse t-structure.
Finally, in Section~\ref{sec:Gaits-sheaf} we define the Gaitsgory sheaf, and compute the rank of its restriction and corestriction to semiinfinite orbits.

\subsection{Acknowledgements}

As should be clear from the description above, this paper owes much to work of Gaitsgory.
We thank Alexis Bouthier for very helpful explanations on various topics related to this work, Ivan Losev for indicating a reference, and Gabriel Dospinescu, Timo Richarz and Quan Situ for useful discussions and comments.

\section{Sheaves on affine flag varieties}
\label{sec:sheaves-Fl}

\subsection{Sheaves}
\label{ss:sheaves}

We will work with usual $\infty$-categories of constructible sheaves on prestacks, as discussed in various contexts and various degrees of precision in~\cite{gaitsgory,lz, rs,agkrrv, scholze,bkv,zhu,dlyz},
with a field of coefficients denoted $\bk$ (which is not assumed to be of characteristic $0$). More specifically, we fix an algebraically closed field $\F$, and consider one of the following settings:
\begin{enumerate}
 \item 
 $\F$ is arbitrary, $\ell$ is a prime number invertible in $\F$, $\bk$ is either a finite extension or an algebraic closure of $\F_\ell$ or $\Q_\ell$, and we consider \'etale $\bk$-sheaves on separated schemes of finite type over $\F$;
 \item
 $\F=\C$, $\bk$ is arbitrary, and we work with sheaves for the analytic topology on separated schemes of finite type over $\F$.
\end{enumerate}
The first setting is e.g.~discussed in detail in~\cite[\S 2.2]{dlyz}, building on the more complete treatment in the setting of perfect prestacks in~\cite[\S 10]{zhu}. The second setting is unfortunately not yet discussed in the literature. We will not consider the adaptation of the constructions in this setting, as we are not capable of doing so; instead we will just assume that a sheaf theory as discussed in~\cite[\S 2.2]{dlyz} exists also in this setting, and possesses similar properties.

So, for any algebraic space $X$ of finite presentation over $\F$ we have a $\bk$-linear small stable $\infty$-category of bounded complexes of constructible sheaves on $X$, denoted $\Shvc(X)$, and the $\bk$-linear stable presentable $\infty$-category
\[
 \Shv(X) := \Ind(\Shvc(X))
\]
(where $\Ind$ denotes the ind-completion).
We have push/pull functors associated with morphisms of algebraic spaces of finite presentation for categories $\Shvc(-)$, and we deduce similar functors for the categories $\Shv(-)$ by $\Ind$-extension. In particular, all of our push/pull functors between such $\infty$-categories will be continuous (i.e.~preserve colimits) by design. 



Next, one extends the definition of $\Shvc(X)$ and $\Shv(X)$ by left Kan extension to all quasi-compact and quasi-separated (qcqs) algebraic spaces. Concretely, if a (qcqs) algebraic space $X$ is written as a cofiltered limit $\lim_i X_i$ with each $X_i$ of finite presentation, and with transition maps being affine morphisms, we have
\[
\Shvc(X) = \colim_i \Shvc(X_i), \quad \Shv(X) = \colim_i \Shv(X_i)
\]
with transition functors giving by $!$-pullbacks. Then, these constructions are extended further to prestacks by right Kan extension. Concretely, for a prestack $X$, we have
\begin{equation}
\label{eqn:Shvc-Shv-limits}
\Shvc(X) = \lim_S \Shvc(S), \quad \Shv(X) = \lim_S \Shv(S)
\end{equation}
where $S$ runs over the qcqs algebraic spaces mapping to $X$, with transition functors given by $!$-pullbacks.

\begin{rmk}
In practice, all the $\infty$-categories of sheaves we will consider below can be described in terms of such $\infty$-categories attached to prestacks which are locally of finite presentation over $\F$. For such a prestack, in the limits in~\eqref{eqn:Shvc-Shv-limits} one can restrict to algebraic spaces which are of finite presentation over $\F$, which allows us to bypass the left Kan extension to all qcqs algebraic spaces.
\end{rmk}

By construction,
for any prestack $Y$ 
we have a canonical fully faithful functor
\[
\Shvc(Y) \to \Shv(Y).
\]
This functor induces a canonical functor
\begin{equation*}
\Ind(\Shvc(Y)) \to \Shv(Y)
\end{equation*}
commuting with colimits.
If $Y$ is an algebraic space of finite presentation over $\F$,
this functor is an equivalence, but for more general prestacks this property is not true (see Remark~\ref{rmk:Shv-BT} below). In case $Y$ is an algebraic stack of finite type over $\F$ then the homotopy category of $\Shvc(Y)$ is the familiar bounded derived category of constructible complexes on $Y$.

\subsection{Functorialities}
\label{ss:functorialities}


As explained in the references discussed in~\S\ref{ss:sheaves}, the $\infty$-categories $\Shv(X)$ are equipped with highly structured functorialities, which make it a sheaf theory (on an appropriate $\infty$-category of correspondences) in the sense of~\cite[\S 8.2]{zhu}. In this subsection we spell out the consequences we will use below.
First, for any morphism of prestacks $f : X \to Y$ we have a continuous $!$-pullback functor
\[
f^! : \Shv(Y) \to \Shv(X).
\]
In particular this allows, for any prestack $X$, to define the dualizing complex $\omega_X$ as the $!$-pullback of $\bk$, considered as a complex of sheaves on $\pt:=\Spec(\F)$, under the structure morphism $X \to \pt$. 


Next, for any morphism of prestacks $f : X \to Y$ which is indfp in the sense of~\cite[\S 2.2]{dlyz}, we have a continuous $*$-pushforward functor
\[
f_* : \Shv(X) \to \Shv(Y).
\]
The functors $(-)^!$ and $(-)_*$ satisfy appropriate forms of the projection formula and base change theorem, as discussed (in the setting of general sheaf theories) in~\cite[\S 8.2.1]{zhu}. Considering adjoints to these functors (which exist, and are well behaved, in some special cases) one also obtains $*$-pullback and $!$-pushforward functors, which satisfy projection formulas and base change theorems obtained by adjunction.

In practice, the prestacks we will consider will arise from ind-schemes of ind-finite type over $\F$. (For the basics of the general theory of ind-schemes, see e.g.~\cite{richarz}.) If $Y$ is such an object, which we write as
\[
 Y = \colim_{i \in I} Y_i
\]
where $I$ is a filtered poset, each $Y_i$ is a scheme of finite type over $\F$, and the transition morphisms $Y_i \to Y_j$ are closed immersions, then we have
\[
 \Shv(Y) = \lim_{i \in I} \Shv(Y_i)
\]
where for $i \leq j$ the transition functor $\Shv(Y_j) \to \Shv(Y_i)$ is the $!$-pullback functor. Now each of these functors admits a left adjoint, namely the (fully faithful) pushforward functor under the closed immersion $Y_i \to Y_j$. 
By the usual result of Lurie allowing to transform a limit into a colimit (see~\cite[Part~I, Chap.~1, Proposition~2.5.7]{gr}), 
this implies that we have
\begin{equation}
\label{eqn:shv-indsch-colim}
\Shv(Y) = \colim_{i \in I} \Shv(Y_i),
\end{equation}
where now for $i \leq j$ the transition functor $\Shv(Y_i) \to \Shv(Y_j)$ is the pushforward functor, and the colimit is computed in the $\infty$-category of presentable $\infty$-categories. 

If $X,Y$ are ind-schemes of ind-finite type over $\F$, with respective presentations
\[
 X = \colim_{i \in I} X_i, \quad Y=\colim_{j \in J} Y_j
\]
as above,
and if $f : X \to Y$ is a morphism, then we have a continuous functor $f_*$ by the comments above. But as in~\cite[Proposition~2.3.3]{rs} we also have a (continuous) functor $f_!$ obtained as the left adjoint of $f^!$. Concretely, there exists an increasing function $a : I \to J$ and a system of morphisms $(f_i)_{i \in I}$ compatible with the structure morphisms, where $f_i$ is a morphism from $X_i$ to $Y_{a(i)}$, such that $f=\colim_{i \in I} f_i$. Then the continuous pushforward functors
\[
 f_*, f_! : \Shv(X) \to \Shv(Y)
\]
are obtained
by passage to colimits from the case of schemes of finite type. 

Let us now assume that the morphism $f$ as above is representable in schemes of finite type, in the sense that for any $\F$-scheme of finite type $S$ and for any morphism $S \to Y$ the fiber product $S \times_Y X$ is a scheme of finite type. In this case, as in~\cite[Theorem~2.4.2]{rs} the functor $f_*$ admits a (continuous) left adjoint
\[
f^* : \Shv(Y) \to \Shv(X).
\]
In fact, 
one can choose the presentations above such that $J=I$ and $X_i=Y_i \times_Y X$ for any $i$, and choose $a=\mathrm{id}_I$ and $f_i$ to be the morphism induced by $f$. Then $f^*$ is defined on the subcategory $\Shv(Y_i)$ by the usual theory for schemes of finite type, and we deduce that it is defined everywhere.


The property considered above is satisfied in particular when $f$ is represented by a closed immersion or by an open immersion; in the former case we have $f_*=f_!$, and in the latter case we have $f^* = f^!$. If $Z,X$ are ind-schemes of ind-finite type over $\F$ and if $i : Z \to X$ is a morphism which is representable by a closed immersion, with complementary open immersion $j : U \to X$, then the categories
\[
 \Shv(Z), \quad \Shv(X), \quad \Shv(U)
\]
and the functors $i_*$, $i^*$, $i^!$, $j_*$, $j^!$, $j^*$ satisfy the ``recollement'' formalism of~\cite[\S 1.4]{bbdg}, or more specifically its immediate translation in the world of stable $\infty$-categories studied e.g.~in~\cite[\S 5.2]{loregian}.
In particular,
for any $\mathcal{F} \in \Shv(X)$ we have fiber sequences
\begin{equation}
\label{eqn:fiber-seq-open-closed}
i_! i^! \mathcal{F} \to \mathcal{F} \to j_* j^* \mathcal{F}, \qquad j_! j^! \mathcal{F} \to \mathcal{F} \to i_*i^* \mathcal{F}.
\end{equation}
In this setting we will say that a complex $\cF$ is \emph{supported on $Z$} if it belongs to the essential image of the fully-faithful functor $i_*$, or equivalently if $j^* \cF = 0$.

Below we will use the following fact. Consider an ind-scheme of ind-finite type $X = \colim_i X_i$ over $\F$, and a collection $(U_n : n \in \Z_{\geq 0})$ of open sub-ind-schemes such that $U_n \subset U_{n+1}$ for any $n$, and $X = \bigcup_n U_n$. Denote, for $n \in \Z_{\geq 0}$, by $j_n : U_n \to X$ the open embedding. For any $\cF$ in $\Shv(X)$, we have a canonical functor $\Z_{\geq 0} \to \Shv(X)$ sending $n$ to $(j_n)_! (j_n)^* \cF$, and a canonical morphism
\begin{equation}
\label{eqn:colim-opens}
 \colim_n \, (j_n)_! (j_n)^* \cF \to \cF.
\end{equation}
We claim that this morphism is an isomorphism. In fact, by~\eqref{eqn:shv-indsch-colim} we can assume that $\cF$ is obtained from a complex on $X_i$, and then the claim follows from the fact that $U_n \cap X_i = X_i$ for $n \gg 0$ since $X_i$ is noetherian.

%
%

In case $Y$ is an ind-scheme of ind-finite type over $\F$ as above and $H$ is an affine $\F$-group scheme of finite type acting on $Y$ and preserving each $Y_i$, we also have
\begin{equation}
\label{eqn:shv-indsch-colim-equiv}
\Shv(H \backslash Y) = \lim_{i \in I} \Shv(H \backslash Y_i) = \colim_{i \in I} \Shv(H \backslash Y_i)
\end{equation}
where in the first case the limit is taken with respect to $!$-pullbacks, and in the second case the colimit is taken with respect to pushforward functors (and in the $\infty$-category of presentable $\infty$-categories).
We also have the functors $f_*, f_!, f^*, f^!$ as above for morphisms induced by equivariant morphisms of ind-schemes satisfying the appropriate assumptions.

We also have a natural ``forgetful'' functor
\[
\oblv^H : \Shv(H \backslash Y) \to \Shv(Y),
\]
formally defined as the $!$-pullback under the projection map $Y \to H \backslash Y$.
To avoid cluttered notation, the dualizing complex $\omega_{H \backslash Y}$ 
will also be denoted $\omega_Y$; its image under $\oblv^H$ is $\omega_Y$. Note that $\oblv^H$ is conservative, since $\Shv(H \backslash Y)$ can be described in terms of the usual ``\v{C}ech nerve'' of the projection $Y \to H \backslash Y$.


\begin{rmk}
\label{rmk:Shv-BT}
An important example to keep in mind is when $Y=\mathrm{pt}/T$ for $T$ an $\F$-torus. In this case $\Shv(\mathrm{pt}/T)$ is compactly generated, but the objects in $\Shvc(\mathrm{pt}/T)$ are \emph{not} compact in general. In fact, recall that $\Shv(\mathrm{pt})$ identifies with the derived category of $\bk$-vector spaces. With this identification:
\begin{itemize}
\item
$\Shv(\mathrm{pt}/T)$ consists of objects whose image under $\oblv^T : \Shv(\mathrm{pt}/T) \to \Shv(\mathrm{pt})$ have finite-dimensional cohomology;
\item
the subcategory of compact objects in $\Shv(\mathrm{pt}/T)$ is the (small) idempotent complete stable subcategory generated by pushforward of finite-dimensional vector spaces under the projection $\mathrm{pt} \to \mathrm{pt}/T$.
\end{itemize}
In particular, the constant sheaf belongs to $\Shvc(\mathrm{pt}/T)$, but is not compact.
\end{rmk}

For $X$ an $\F$-ind-scheme of ind-finite type with a presentation $X=\colim_{i \in I} X_i$ as above, it turns out that the $\infty$-category $\Shvc(X)$ does not model what one often calls ``constructible sheaves'' in this context; namely, there exists a fully faithful functor
\[
\colim_{i \in I} \Shvc(X_i) \to \Shvc(X)
\]
(where the left-hand side is the colimit in small stable $\infty$-categories, or equivalently in all $\infty$-categories, and with respect to pushforward functors), but this functor is not an equivalence in general (see~\cite[Remark~5.3.4]{bkv} or the discussion at the beginning of~\cite[\S 10.6.2]{zhu}). To remedy this, Zhu develops in~\cite[\S 10]{zhu} a formalism $\Shvfg$ of ``finitely generated sheaves'' on appropriate prestacks (taking values in small stable $\infty$-categories). When $X$ is an ind-scheme as above, then it immediately follows from the definition that
\begin{equation}
\label{eqn:Shvfg-colim}
\Shvfg(X) = \colim_{i \in I} \Shvc(X_i),
\end{equation}
see~\cite[Lemma~10.161]{zhu}, and we have
\[
\Shv(X) = \Ind(\Shvfg(X)),
\]
see~\cite[Remark~10.162]{zhu}. In fact an equivalence similar to~\eqref{eqn:Shvfg-colim} holds for what Zhu calls ``ind-placid stacks,'' which includes the Hecke stack. Details will be discussed below in the special cases we want to consider.

\subsection{Affine flag varieties}
\label{ss:affine-flag-var}

Let $\F$ be as in~\S\ref{ss:sheaves}, and let $G$ be a connected reductive algebraic group over $\F$. Fix a Borel subgroup $B \subset G$ and a maximal torus $T \subset B$, and denote by $U$ the unipotent radical of $B$. We will denote by
\[
\bY = X_*(T), \quad \text{resp.} \quad \bX=X^*(T),
\]
the cocharacter, resp.~character, lattice of $T$.
Let $\fR \subset \bX$ be the set of roots for $(G,T)$, and let $\fR_+ \subset \fR$ be the set of positive roots determined by $B$, i.e.~the set of nonzero weights of $T$ on $\Lie(B)$. We will denote by $\rho$ the halfsum of the positive roots. We will also denote by $W_\fin$ the Weyl group of $(G,T)$, which we endow with the Coxeter structure determined by our choice of $B$, and by $w_\circ$ its longest element.

For any affine scheme $X$ over $\F$ we will denote by
\[
\Loop X, \quad \text{resp.} \quad \Loop^+ X,
\]
the ind-affine ind-scheme,
resp.~affine scheme, representing the functor sending an $\F$-algebra $R$ to $X(R( \hspace{-1pt} (z) \hspace{-1pt} ))$, resp.~$X(R[ \hspace{-1pt} [z] \hspace{-1pt} ])$, where $z$ is an indeterminate. There exist canonical morphisms $\Loop^+ X \to \Loop X$ and $\Loop^+ X \to X$.
In case $X$ is a group scheme, $\Loop X$, resp.~$\Loop^+ X$, admits a natural structure of group ind-scheme, resp.~group scheme, and the morphisms $\Loop^+ X \to \Loop X$ and $\Loop^+ X \to X$ are morphisms of group (ind-)schemes. In this case the kernel of the morphism $\Loop^+ X \to X$ will be denoted $\Loop^{++} X$.

In particular we will consider the groups $\Loop G$, $\Loop^+ G$, $\Loop^+ T$ and $\Loop^{++} T$. Each $\lambda \in \bY$ defines an $\F$-point $z^\lambda \in (\Loop T)(\F)$; we will denote similarly its image in $(\Loop G)(\F)$.

Let $\Iw \subset \Loop^+G$ be the Iwahori subgroup attached to $B$, i.e.~the preimage of $B$ under the canonical morphism $\Loop^+G \to G$, and let $\Iwu$ be the pro-unipotent radical of $\Iw$, i.e.~the preimage of $U$ under this morphism. We will consider in addition the ``semiinfinite Iwahori subgroup''
\[
 \siIw := \Loop^+T \ltimes \Loop U
\]
and its ``pro-unipotent radical''
\[
 \siIwu := \Loop^{++} T \ltimes \Loop U.
\]
Here $\Iw$ and $\Iwu$ are group schemes over $\F$, but $\siIw$ and $\siIwu$ are only group ind-schemes over $\F$. Since $U$ is unipotent, we can write
\[
 \Loop U = \colim_{n \geq 0} \Loop_{(n)} U
\]
where each $\Loop_{(n)} U$ is a group scheme over $\F$ and each transition morphism $\Loop_{(n)} U \to \Loop_{(n+1)} U$ is a closed immersion of group schemes. These group schemes can be chosen to be stable under the conjugation action of $\Loop^+ T$, so that we obtain presentations
\begin{equation}
\label{eqn:presentation-siIw}
 \siIw = \colim_{n \geq 0} \siIw_{(n)}, \quad \siIwu = \colim_{n \geq 0} \siIw_{\mathrm{u},(n)}
\end{equation}
where
\[
 \siIw_{(n)} := \Loop^+T \ltimes \Loop_{(n)} U, \quad \siIw_{\mathrm{u},(n)} = \Loop^{++} T \ltimes \Loop_{(n)} U
\]
are group schemes.

\begin{rmk}
\label{rmk:choice-LnU}
Explicitly, after choosing some $\lambda \in \bY$ which is strongly dominant (i.e.~which satisfies $\langle \lambda, \alpha \rangle >0$ for any $\alpha \in \fR_+$), one can e.g.~choose $\Loop_{(n)} U = z^{-n\lambda} \Loop^+ U z^{n\lambda}$.
\end{rmk}

We will consider the affine flag variety
\[
\Fl = \Loop G/\Iw,
\]
the ``extended affine'' flag variety
\[
 \tFl = \Loop G/\Iwu,
\]
and the affine Grassmannian
\[
\Gr = \Loop G/\Loop^+G.
\]
In the three cases the quotient can be taken either for the \'etale topology or for the fppf topology (giving rise to the same object), and we obtain separated ind-schemes of ind-finite type over $\F$, endowed with actions of the group scheme $\Loop^+ G$. Moreover, there exist $\Loop^+ G$-equivariant morphisms
\begin{equation*}
 \tFl \xrightarrow{q} \Fl \xrightarrow{p} \Gr.
\end{equation*}

\begin{rmk}
\begin{enumerate}
 \item 
 It can also sometimes be convenient to work with the left quotients
 \begin{equation}
 \label{eqn:left-quotients}
  \Fl' = \Iw \backslash \Loop G, \quad \tFl' = \Iwu \backslash \Loop G, \quad \Gr' = \Loop^+ G \backslash \Loop G,
 \end{equation}
with the actions of $\Loop^+G$ induced by multiplication on the right on $\Loop G$. Of course these variants are isomorphic to the ind-schemes considered above via the map induced by $g \mapsto g^{-1}$ on $\Loop G$. 
 \item
 All the constructions that we consider below have obvious analogues for sheaves on any partial affine flag variety. We will only spell out explicitly these constructions for $\Fl$ and $\Gr$, since these are the only cases for which we have concrete applications in mind.
\end{enumerate}
\end{rmk}

\subsection{Sheaves on affine flag varieties}
\label{ss:sheaves-Fl}

Following the constructions in~\S\S\ref{ss:sheaves}--\ref{ss:functorialities}
we can consider the stable presentable $\infty$-categories
\[
\Shv(\tFl), \quad 
\Shv(\Fl), \quad \Shv(\Gr), \quad
\Shv(T \backslash \tFl), \quad 
\Shv(T \backslash \Fl), \quad \Shv(T \backslash \Gr),
\]
and 
we have continuous functors
\begin{gather*}
p^!, p^* : \Shv(\Gr) \to \Shv(\Fl), \quad p_!, p_* : \Shv(\Fl) \to \Shv(\Gr) \\
q^!, q^* : \Shv(\Fl) \to \Shv(\tFl), \quad q_!, q_* : \Shv(\tFl) \to \Shv(\Fl)
\end{gather*}
and similarly for the equivariant versions. These functors satisfy
\[
 q^* \cong q^![-2\dim(T)], \quad p^* \cong p^![-2\ell(w_\circ)], \quad p_* \cong p_!,
\]
and we have adjoint pairs $(q^*, q_*)$, $(q_!, q^!)$, $(p^*, p_*)$, $(p_!, p^!)$.
We will also set
\[
p^\dag = p^![-\ell(w_\circ)] \cong p^*[\ell(w_\circ)].
\]


\subsection{Iwahori-equivariant sheaves on affine flag varieties}
\label{ss:Iw-eq-sheaves}

Let $X$ be either $\Gr$, $\Fl$, or $\tFl$,
and let $J$ be either $\Iw$ or $\Iwu$. Then we have the corresponding ``Hecke stack'' $J \backslash X$, i.e.~the \'etale stackification of the naive quotient stack $J \backslash X$. The associated $\infty$-category of $\bk$-sheaves $\Shv(J \backslash X)$ admits the following 
concrete description in terms of $\infty$-categories of sheaves on algebraic stacks of finite type.

There exists a presentation
\begin{equation}
\label{eqn:presentation-X}
X = \colim_{n \geq 0} X_n
\end{equation}
of $X$ as an ind-scheme where each $X_n$ is separated of finite type over $\F$ and stable under the action of $J$, and an identification
\begin{equation}
\label{eqn:J-lim}
 J = \lim_{m \geq 0} J_{m}
\end{equation}
where each $J_{m}$ is an affine group scheme of finite type over $\F$ and the transition morphisms $J_{m} \to J_{m-1}$ are quotient morphisms with smooth connected unipotent kernels, such that for any $n \geq 0$ there exists $M_n \in \Z_{\geq 0}$ such that the action of $J$ on $X_n$ factors through an action of $J_{M_n}$. Then for $n \geq 0$ the $\infty$-category
\[
\Shv(J_{m} \backslash X_n)
\]
does not depend on $m \geq M_n$ up to canonical equivalence; in fact the $!$-pullback functor $\Shv(J_{m} \backslash X_n) \to \Shv(J \backslash X_n)$ is an equivalence, see~\cite[Proposition~2.2.11]{rs}. We have natural fully faithful functors
\[
\Shv(\Iwu \backslash X_n) \to \Shv(X_n), \quad \Shv(\Iw \backslash X_n) \to \Shv(T \backslash X_n)
\]
(again, defined via $!$-pullback).

For any $n \geq 1$ we have a canonical functor $\Shv(J \backslash X_n) \to \Shv(J \backslash X_{n-1})$ corresponding under the identifications above to the $!$-pullback functor $\Shv(J_m \backslash X_n) \to \Shv(J_m \backslash X_{n-1})$ for $m \gg 0$,
and since $J \backslash X = \colim_n J \backslash X_n$ and $\Shv$ transforms colimits into limits we have
\[
 \Shv(J \backslash X) = \lim_{n \geq 0} \Shv(J \backslash X_n).
\]
Here again we have fully faithful functors
\[
\Shv(\Iwu \backslash X) \to \Shv(X), \quad \Shv(\Iw \backslash X) \to \Shv(T \backslash X).
\]
As in~\eqref{eqn:shv-indsch-colim} and~\eqref{eqn:shv-indsch-colim-equiv}, we also have
\begin{equation}
\label{eqn:shv-I-Fl-colim}
\Shv(J \backslash X) = \colim_{n \geq 0} \Shv(J \backslash X_n)
\end{equation}
where the colimit is taken with respect to pushforward functors.

We have canonical exact and continuous ``forgetful'' functors
\begin{equation}
\label{eqn:sh-Iw-Iwu-oblv}
 \Shv(\Iw \backslash X) \xrightarrow{\oblv^{\Iw | \Iwu}} \Shv(\Iwu \backslash X) \xrightarrow{\oblv^{\Iwu}} \Shv(X),
\end{equation}
the second of which is fully faithful. (Explicitly, these functors are $!$-pullback functors associated with the natural maps $\Iwu \backslash X \to \Iw \backslash X$ and $X \to \Iwu \backslash X$ respectively.) The functors $p^!$, $p^*$, $p^\dag$, $q^*$, $q^!$ of~\S\ref{ss:sheaves-Fl} restrict to functors
\[
 p^!, p^*, p^\dag : \Shv(J \backslash \Gr) \to \Shv(J \backslash \Fl), \quad 
 q^*, q^! : \Shv(J \backslash \Fl) \to \Shv(J \backslash \tFl)
\]
which are compatible in the obvious way (via the functors $\oblv^{\Iw | \Iwu}$ and $\oblv^{\Iwu}$) with their non-equivariant counterparts.
We also have pushforward functors
\[
p_*, p_! : \Shv(J \backslash \Fl) \to \Shv(J \backslash \Gr), \quad
q_*, q_! : \Shv(J \backslash \tFl) \to \Shv(J \backslash \Fl),
\]
which are compatible in the obvious way with their non-equivariant counterparts, and the appropriate adjunctions. 

It is a standard fact that the $\infty$-category $\Shv(\Iw \backslash \Fl)$ admits a canonical structure of algebra object in the monoidal $\infty$-category of $\bk$-linear stable presentable $\infty$-categories. For discussions of this construction in various sheaf-theoretic contexts, see e.g.~\cite[\S 4.2]{alwy},~\cite[\S 2.2.8]{cvdhs},~\cite[\S 4.1.1]{dlyz}. The associated bifunctor on $\Shv(\Iw \backslash \Fl)$ is the usual convolution product $\star^\Iw$, and the unit object is the skyscraper sheaf $\delta_{\Fl}$ at the base point of $\Fl$. The $\infty$-category $\Shv(\Iw \backslash \Gr)$ is naturally a left module over $\Shv(\Iw \backslash \Fl)$, and the $\infty$-categories $\Shv(\Fl)$ and $\Shv(T \backslash \Fl)$ are naturally right modules over $\Shv(\Iw \backslash \Fl)$, with the full subcategories $\Shv(\Iwu \backslash \Fl) \subset \Shv(\Fl)$, $\Shv(\Iw \backslash \Fl) \subset \Shv(T \backslash \Fl)$ being stable under this action.

In case $X=\Gr$ or $\Fl$, we have the usual perverse t-structure (for the middle perversity) on the $\infty$-category $\Shv(X)$, and there exist unique t-structures on the $\infty$-categories $\Shv(\Iw \backslash X)$ and $\Shv(\Iwu \backslash X)$ such that the functors in~\eqref{eqn:sh-Iw-Iwu-oblv} are t-exact. (For an extensive study of these constructions, see e.g.~\cite[\S 6]{bkv}.) All these t-structures will be called ``perverse.''
Their hearts are the Ind-completions of the usual categories of perverse sheaves; they will be denoted $\Shv(\Iw \backslash X)^\heartsuit$, $\Shv(\Iwu \backslash X)^\heartsuit$ and $\Shv(X)^\heartsuit$ respectively. It is a standard fact also that the forgetful functor
\[
\oblv^{\Iw | \Iwu} : \Shv(\Iw \backslash X)^\heartsuit \to \Shv(\Iwu \backslash X)^\heartsuit
\]
is fully faithful, and that its essential image is stable under subquotients.

For $X=\tFl$, we will only consider the category $\Shv(\Iw \backslash \tFl)$. In this case, we will shift the usual perverse t-structure, in such a way that the functor
\[
 q^! : \Shv(\Iw \backslash \Fl) \to \Shv(\Iw \backslash \tFl)
\]
becomes t-exact. Again this t-structure will be called the perverse t-structure, and its heart will be denoted $\Shv(\Iw \backslash \tFl)^\heartsuit$. (The main point of this normalization is that, applying similar conventions for the $\infty$-categories considered in Remark~\ref{rmk:Iw-eq-variants}\eqref{it:Ieq-cat-leftquot} below, the equivalences in~\eqref{eqn:equiv-left-right} are t-exact.)


\begin{rmk}
\phantomsection
\label{rmk:Iw-eq-variants}
\begin{enumerate}
 \item 
 \label{it:sph}
  Below we will also consider the $\infty$-category $\Shv(\Loop^+ G \backslash \Gr)$ of $\Loop^+ G$-equivariant sheaves on $\Gr$. This category also admits a canonical structure of algebra object in the monoidal $\infty$-category of $\bk$-linear stable presentable $\infty$-categories, with the associated bifunctor given by convolution $\star^{\Loop^+ G}$, and the $\infty$-categories $\Shv(\Iw \backslash \Gr)$ and $\Shv(\Iwu \backslash \Gr)$ are right modules for this algebra object. We have a natural exact functor
 \[
  \oblv^{\Loop^+G | \Iw} : \Shv(\Loop^+ G \backslash \Gr) \to \Shv(\Iw \backslash \Gr)
 \]
 which is compatible with these actions.
 \item
 \label{it:Ieq-cat-leftquot}
 One can also consider the obvious versions of the categories above for the ind-schemes in~\eqref{eqn:left-quotients}; by standard arguments the $\infty$-categories obtained in this way satisfy
 \begin{equation}
 \label{eqn:equiv-left-right}
  \Shv(\Iw \backslash \Fl) \cong \Shv(\Fl' / \Iw), \ \Shv(\Iwu \backslash \Fl) \cong \Shv(\tFl' / \Iw), \ \Shv(\Iw \backslash \tFl) \cong \Shv(\Fl' / \Iwu),
 \end{equation}
 in such a way that the $*$-pullback functor $\Shv(\Fl' / \Iw) \to \Shv(\tFl' / \Iw)$ associated with the natural quotient map $\tFl' \to \Fl'$ corresponds to the forgetful functor $\oblv^{\Iw | \Iwu} : \Shv(\Iw \backslash \Fl) \to \Shv(\Iwu \backslash \Fl)$.
To emphasize this fact, it is sometimes convenient to denote these categories as
\[
 \Shv(\Iw \backslash \Loop G / \Iw), \quad \Shv(\Iwu \backslash \Loop G / \Iw), \quad \Shv(\Iw \backslash \Loop G / \Iwu).
\]
\end{enumerate}
\end{rmk}

In some statements below we also consider finitely generated sheaves in the sense discussed in~\S\ref{ss:functorialities}. With the notation used above we have
\[
\Shvfg(J \backslash X) = \colim_{n \geq 0} \Shvc(J\backslash X_n).
\]
(In fact this follows from~\cite[Lemma~10.161]{zhu}, since $J \backslash X$ is an ind-very placid stack, see~\cite[Example~10.117]{zhu}.) We also have
\begin{gather*}
\Shv(\Iwu \backslash \Fl) = \Ind \bigl( \Shvfg(\Iwu \backslash \Fl) \bigr), \quad 
\Shv(\Iwu \backslash \Gr) = \Ind \bigl( \Shvfg(\Iwu \backslash \Gr) \bigr), \\
\Shv(\Iw \backslash \tFl) = \Ind \bigl( \Shvfg(\Iw \backslash \tFl) \bigr)
\end{gather*}
by~\cite[Remark~10.162]{zhu}, but the corresponding statement for $\Iw \backslash \Fl$ does not hold.

\subsection{Semiinfinite Iwahori-equivariant sheaves on affine flag varieties}
\label{ss:si-sheaves-Fl}

We continue to let $X$ be either $\Gr$, $\Fl$, or $\tFl$, but let now $J$ be either $\siIw$ or $\siIwu$. Then one can consider the stack quotient $J \backslash X$. The corresponding $\infty$-category of sheaves admits the following explicit description, following~\cite[\S 1.2]{gaitsgory}.

Write $J=\colim_k J_{(k)}$ for the presentation in~\eqref{eqn:presentation-siIw}. For any $k$, there exists a presentation~\eqref{eqn:presentation-X} such that each $X_n$ is stable under the action of $J_{(k)}$, and an identification
\begin{equation}
\label{eqn:Jk-lim}
 J_{(k)} = \lim_{m \geq 0} J_{(k),m}
\end{equation}
where each $J_{(k),m}$ is an affine group scheme of finite type over $\F$, each transition morphism $J_{(k),m} \to J_{(k),m-1}$ is a quotient morphism with smooth connected unipotent kernel and such that for any $n \geq 0$ there exists $M_n$ such that 
the action of $J_{(k)}$ on $X_n$ factors through an action of $J_{(k),M_n}$. Then by the same considerations as in~\S\ref{ss:Iw-eq-sheaves} the $\infty$-category $\Shv(J_{(k),m} \backslash X_n)$ does not depend on $m \geq M_n$ up to canonical equivalence, and coincides with
$\Shv(J_{(k)} \backslash X_n)$.
For any $n \geq 1$ we have a $!$-pullback functor $\Shv(J_{(k)} \backslash X_n) \to \Shv(J_{(k)} \backslash X_{n-1})$ corresponding to the $!$-pullback functor $\Shv(J_{(k),m} \backslash X_n) \to \Shv(J_{(k),m} \backslash X_{n-1})$ for $m \gg 0$,
and we have
\[
 \Shv(J_{(k)} \backslash X) = \lim_{n \geq 0} \Shv(J_{(k)} \backslash X_n).
\]
In case $J=\siIwu$ we have a fully faithful ``forgetful'' functor
\[
 \Shv(J_{(k)} \backslash X) \to \Shv(X)
\]
and in case $J=\siIw$ we have a fully faithful ``forgetful'' functor
\[
 \Shv(J_{(k)} \backslash X) \to \Shv(T \backslash X).
\]

For any $k$ we have a fully faithful ``forgetful'' functor
\[
 \Shv(J_{(k+1)} \backslash X) \to \Shv(J_{(k)} \backslash X),
\]
and we have
\[
 \Shv(J \backslash X) = \lim_{k \geq 0} \Shv(J_{(k)} \backslash X).
\]
(To justify this, one notes that $J \backslash X = \colim_k J_{(k)} \backslash X$ since quotient stacks are defined as a colimit over an appropriate \v{C}ech nerve and colimits commute with colimits, and then one uses the fact that $\Shv$ transforms colimits into limits.)
We also have fully faithful ``forgetful'' functors
\[
 \Shv(\siIwu \backslash X) \to \Shv(X), \quad
 \Shv(\siIw \backslash X) \to \Shv(T \backslash X).
\]

We have canonical exact and continuous ``forgetful'' functors
\[
 \Shv(\siIw \backslash X) \xrightarrow{\oblv^{\siIw | \siIwu}} \Shv(\siIwu \backslash X) \xrightarrow{\oblv^{\siIwu}} \Shv(X)
\]
defined as $!$-pullback under the maps $X \to \siIwu \backslash X \to \siIw \backslash X$,
the second of which is fully faithful. The functor $\oblv^{\siIw | \siIwu}$ is the restriction of the forgetful functor
\[
 \oblv^T : \Shv(T \backslash X) \to \Shv(X).
\]
The latter functor has a left adjoint $\Av^T_!$ and a right adjoint $\Av^T_*$. These functors restrict to functors
\[
 \Av^{\siIw | \siIwu}_!, \Av^{\siIw | \siIwu}_* : \Shv(\siIwu \backslash X) \to \Shv(\siIw \backslash X)
\]
which are left and right adjoint to $\oblv^{\siIw | \siIwu}$ respectively.

The functors $p^!$, $p^*$, $p^\dag$ of~\S\ref{ss:sheaves-Fl} restrict to functors
\[
 p^!, p^*, p^\dag : \Shv(\siIw \backslash \Gr) \to \Shv(\siIw \backslash \Fl), \quad p^!,p^*,p^\dag : \Shv(\siIwu \backslash \Gr) \to \Shv(\siIwu \backslash \Fl)
\]
which are compatible in the obvious way (via the functors $\oblv^{\siIw | \siIwu}$ and $\oblv^{\siIwu}$) with their non-equivariant counterparts.
We also have (isomorphic) pushforward functors
\[
p_*, p_! : \Shv(\siIw \backslash \Fl) \to \Shv(\siIw \backslash \Gr), \quad
p_*, p_! : \Shv(\siIwu \backslash \Fl) \to \Shv(\siIwu \backslash \Gr)
\]
which are compatible in the obvious way with their non-equivariant counterparts, and the usual adjunctions. Similar comments apply to the morphism $q$.

It is clear that the full subcategories
\[
\Shv(\siIwu \backslash \Fl) \subset \Shv(\Fl), \quad \Shv(\siIw \backslash \Fl) \subset \Shv(T \backslash \Fl)
\]
are stable under the right actions of $\Shv(\Iw \backslash \Fl)$ considered in~\S\ref{ss:Iw-eq-sheaves}, hence are naturally right modules for this algebra object. Similarly, the $\infty$-categories $\Shv(\siIwu \backslash \Gr)$ and $\Shv(\siIw \backslash \Gr)$ are canonically right modules over the algebra object $\Shv(\Loop^+ G \backslash \Gr)$ of Remark~\ref{rmk:Iw-eq-variants}\eqref{it:sph}.

If $X$ is as above,
for any $\lambda \in \bY$ we have an automorphism of $X$ given by the left action of $z^\lambda \in (\Loop G)(\F)$, see~\S\ref{ss:affine-flag-var}. The composition of pushforward under this map with coholomogical shift $[-\langle \lambda, 2\rho \rangle]$ provides an autoequivalence of $\Shv(X)$, which will be denoted
\begin{equation}
\label{eqn:shift-equiv}
 \cF \mapsto \cF \langle \lambda \rangle,
\end{equation}
and which restricts to an autoequivalence of $\Shv(\siIwu \backslash X)$. Since $z^\lambda$ commutes with $T$, we similarly have an autoequivalence of $\Shv(T \backslash X)$, which restricts to an autoequivalence of $\Shv(\siIw \backslash X)$, and will be denoted by the same symbol.
(In case $X=\Gr$, this construction appears in~\cite[\S 5.1.2]{gaitsgory}.)

\subsection{Raskin's equivalences}
\label{ss:raskin-equiv}

Now we recall a construction due to Raskin~\cite[\S 6]{raskin}, as reformulated by Gaitsgory in~\cite[\S 5.2]{gaitsgory}, which will be a fundamental tool for our constructions. (Gaitsgory only considers in detail a special case, corresponding to $X=\Gr$. He notes however that this construction applies, and has similar properties, for any $\infty$-category acted on by $\Loop G$ in an appropriate sense.) For more on the context of this construction, see also~\cite{raskin-ws}.

If $X$ is either $\Gr$, $\Fl$, or $\tFl$, then one can consider the $\bk$-linear stable presentable $\infty$-category $\Shv(\Loop^+ T \backslash X)$, which can be described along similar lines as those of~\S\ref{ss:Iw-eq-sheaves}, and we have exact and fully-faithful ``forgetful'' functors
\[
 \Shv(\Iw \backslash X) \xrightarrow{\oblv^{\Iw | \Loop^+ T}} \Shv(\Loop^+ T \backslash X) \xleftarrow{\oblv^{\siIw | \Loop^+ T}} \Shv(\siIw \backslash X).
\]
The functor $\oblv^{\Iw | \Loop^+ T}$ admits a continuous right adjoint
\[
 \mathrm{Av}^{\Iw | \Loop^+T}_* : \Shv(\Loop^+ T \backslash X) \to \Shv(\Iw \backslash X).
\]
(In fact, to justify this claim, one can replace $X$ by an $\Iw$-stable subscheme on which the action factors through the quotient by a pro-unipotent normal subgroup of finite codimension. Then any object is equivariant for the subgroup generated by $\Loop^+T$ and this normal subgroup, and $\mathrm{Av}^{\Iw | \Loop^+T}_*$ is realized by the usual $*$-averaging functor from the latter subgroup to $\Iw$.) 
As explained in~\cite[Proposition~5.2.2]{gaitsgory}, the composition
\[
 \Shv(\siIw \backslash X) \xrightarrow{\oblv^{\siIw | \Loop^+ T}} \Shv(\Loop^+ T \backslash X) \xrightarrow{\mathrm{Av}^{\Iw | \Loop^+T}_*} \Shv(\Iw \backslash X)
\]
is an equivalence of categories, which we will denote by
\[
 \Ras : \Shv(\siIw \backslash X) \simto \Shv(\Iw \backslash X);
\]
see~\S\ref{ss:raskin-equiv-proof} below for details.

The inverse equivalence $\Ras^{-1}$ can be described as follows. As explained in~\S\ref{ss:raskin-equiv-proof} below, the functor $\oblv^{\siIw | \Loop^+ T}$ admits a (continuous) left adjoint
\[
 \Av_!^{\siIw | \Loop^+ T} : \Shv(\Loop^+ T \backslash X) \to \Shv(\siIw \backslash X).
\]
Then it is explained in~\cite[\S 5.2.3]{gaitsgory} that the composition
\[
 \Av_!^{\siIw | \Loop^+ T} \circ \oblv^{\Iw | \Loop^+ T} : \Shv(\Iw \backslash X) \to \Shv(\siIw \backslash X)
\]
is the inverse to $\Ras$.

It is clear that these equivalences intertwine the functors $p^*$, $p^!$, etc.~in the obvious way.
Similar considerations replacing $\Loop^+ T$ by $\Loop^{++} T$ lead, when $X$ is either $\Gr$ or $\Fl$, to equivalences of categories
\[
 \Rasu : \Shv(\siIwu \backslash X) \simto \Shv(\Iwu \backslash X).
\]
satisfying the same properties, and such that
\[
\Rasu \circ \oblv^{\siIw | \siIwu} \cong \oblv^{\Iw | \Iwu} \circ \Ras. 
\]

It is clear that in case $X=\Fl$ the equivalences $\Ras$ and $\Rasu$ are equivalences of right modules over the algebra object $\Shv(\Iw \backslash \Fl)$, and that in case $X=\Gr$ they are equivalences of right modules over the algebra object $\Shv(\Loop^+ G \backslash \Gr)$.


\subsection{Some details on the proof of Raskin's equivalences}
\label{ss:raskin-equiv-proof}

In this subsection we briefly outline the proof (following~\cite[\S 5.2]{gaitsgory}, see also~\cite[\S 3.1.9]{dl}) of the fact that the functor $\Ras$ of~\S\ref{ss:raskin-equiv} is an equivalence, mainly to convince the reader that this statement applies for any kinds of coefficients. (No detail of this proof will be required below.) These considerations will use some of the constructions explained in Sections~\ref{sec:Waff}--\ref{sec:complexes-I-orbits} below; in fact, the essential ingredient of the proof is the fact that the standard and costandard objects in $\Shv(\Iw \backslash \Fl)$ (whose definition is recalled in~\S\ref{ss:standard-costandard-Fl} below) are invertible for the convolution product. The proof that $\Rasu$ is an equivalence is similar; in this case one needs to work with the ``free monodromic'' versions of the standard and costandard objects, which will not be explicitly discussed here (see e.g.~\cite{bezr2, cd}).

As in Remark~\ref{rmk:choice-LnU}, we choose a strongly dominant $\lambda \in \bY$, and take $\Loop_{(n)} U = z^{-n\lambda} \Loop^+ U z^{n\lambda}$ in the notation of~\S\ref{ss:affine-flag-var}. For any $n \geq 0$ we have a forgetful functor
\[
\Shv(\siIw_{(n)} \backslash X) \to \Shv(\Loop^+ T \backslash X),
\]
which admits a continuous left adjoint
\[
 \Av_!^{\siIw_{(n)} | \Loop^+ T} : \Shv(\Loop^+ T \backslash X) \to \Shv(\siIw_{(n)} \backslash X).
\]
(The justification is similar to that for the functor $\mathrm{Av}^{\Iw | \Loop^+T}_*$ in~\S\ref{ss:raskin-equiv}.) These functors form an inductive system parametrized by $\Z_{\geq 0}$, and the left adjoint $\Av_!^{\siIw | \Loop^+ T}$ to $\oblv^{\siIw | \Loop^+ T}$ is given by
\[
\Av_!^{\siIw | \Loop^+ T} ( \cF) = \colim_{n \geq 0} \Av_!^{\siIw_{(n)} | \Loop^+ T} ( \cF)
\]
for $\cF$ in $\Shv(\Loop^+ T \backslash X)$.

These functors have more explicit descriptions when applied to objects in the image of $\oblv^{\Iw | \Loop^+ T}$: in fact, using Lem\-ma~\ref{lem:orbits-Iw-Fl} below one sees that for any $\cF$ in $\Shv(\Iw \backslash X)$ we have
\[
\Av_!^{\siIw_{(n)} | \Loop^+ T} ( \oblv^{\Iw | \Loop^+ T}(\cF)) \cong z^{-n \lambda} \cdot \DFl_{t_{n\lambda}} \star^\Iw \cF [n N]
\]
where $N$ is the pairing of $\lambda$ with the sum of the positive roots, the notation $t_{n\lambda}$ is explained in~\S\ref{ss:Weyl}, the corresponding standard object $\DFl_{t_{n\lambda}}$ is defined in~\S\ref{ss:standard-costandard-Fl}, and we denote by $z^{-n \lambda} \cdot (-)$ the pushforward functor under the automorphism of $X$ given by the left action of $z^{-n \lambda} \in \Loop G(\F)$. Hence, for any such $\cF$ we have
\begin{equation}
\label{eqn:Av-formula}
\Av_!^{\siIw | \Loop^+ T} ( \oblv^{\Iw | \Loop^+ T}(\cF)) 
\cong \colim_{n \geq 0} z^{-n \lambda} \cdot \DFl_{t_{n\lambda}} \star^\Iw \cF [nN].
\end{equation}

Next, one considers the adjunction morphism
\begin{equation}
\label{eqn:Raskin-adjunction}
\cF \to \mathrm{Av}^{\Iw | \Loop^+T}_* \circ \Av_!^{\siIw | \Loop^+ T}(\cF)
\end{equation}
for $\cF$ in $\Shv(\Iw \backslash X)$, where forgetful functors are omitted. Then using~\eqref{eqn:Av-formula} one obtains that
\begin{multline*}
\mathrm{Av}^{\Iw | \Loop^+T}_* \circ \Av_!^{\siIw | \Loop^+ T}(\cF) \cong \colim_{n \geq 0} \mathrm{Av}^{\Iw | \Loop^+T}_*(z^{-n \lambda} \cdot \DFl_{t_{n\lambda}} \star^\Iw \cF)[nN] \\
\cong \colim_{n \geq 0} \NFl_{t_{-n \lambda}} \star^\Iw \DFl_{t_{n\lambda}} \star^\Iw \cF
\end{multline*}
where the object $\NFl_{t_{-n \lambda}}$ is defined in~\S\ref{ss:standard-costandard-Fl}. By Lemma~\ref{lem:standards-costandards}\eqref{it:standards-costandards-2} we have a canonical isomorphism $\NFl_{t_{-n \lambda}} \star^\Iw \DFl_{t_{n\lambda}} \cong \delta_{\Fl}$; it follows that we have an identification
\[
\mathrm{Av}^{\Iw | \Loop^+T}_* \circ \Av_!^{\siIw | \Loop^+ T}(\cF) \cong \cF,
\]
under which the adjunction morphism~\eqref{eqn:Raskin-adjunction} becomes the identity.

At this point, to conclude it suffices to show that the functor $\Ras$ is conservative, or equivalently that it does not kill any object. Let $U^-$ be the unipotent radical of the Borel subgroup of $G$ opposite to $B$ with respect to $T$; then the multiplication morphism
\[
 \Loop^{++} (U^-) \times \Loop^+ T \times \Loop^+ U \to \Iw
\]
is an isomorphism,
hence if $\cG$ is in $\Shv(\siIw \backslash X)$ we have
\begin{equation}
\label{eqn:Ras-U-}
\Ras(\cG) = \mathrm{Av}^{\Iw | \Loop^+T}_*(\cG) \cong \mathrm{Av}^{\Loop^{++}(U^-)}_*(\cG).
\end{equation}
The basic observation here is that there exists a nonzero morphism $\cG' \to \cG$ where $\cG' \in \Shv(X)$ is an object which is equivariant with respect to the action of the group $z^{-n\lambda} \Loop^{++} (U^-) z^{n\lambda}$ for some $n \gg 0$, so that $\mathrm{Av}_*^{z^{-n\lambda} \Loop^{++} (U^-) z^{n\lambda}}(\cG) \neq 0$. The latter object naturally belongs to
\[
\Shv((z^{-n\lambda} \Iw z^{n\lambda}) \backslash X),
\]
hence $z^{n\lambda} \cdot \mathrm{Av}_*^{z^{-n\lambda} \Loop^{++} (U^-) z^{n\lambda}}(\cG)$ is a nonzero object in $\Shv(\Iw \backslash X)$, and one checks as above that
\begin{multline*}
\mathrm{Av}^{\Loop^{++}(U^-)}_*(\cG) = \mathrm{Av}^{\Loop^{++}(U^-)}_* \circ \mathrm{Av}_*^{z^{-n\lambda} \Loop^{++} (U^-) z^{n\lambda}}(\cG) \\
\cong \nabla_{t_{-n \lambda}} \star^{\Iw} (z^{n\lambda} \cdot \mathrm{Av}_*^{z^{-n\lambda} \Loop^{++} (U^-) z^{n\lambda}}(\cG)).
\end{multline*}
Using once again the invertibility of $\nabla_{t_{-n \lambda}}$, this object is nonzero. This concludes the proof, in view of~\eqref{eqn:Ras-U-}.

\section{Affine Weyl group and orbits}
\label{sec:Waff}

\subsection{Weyl groups}
\label{ss:Weyl}

Recall the notation of~\S\ref{ss:affine-flag-var}.
Let $\bY_+ \subset \bY$ be the subset of dominant coweights determined by $\fR_+$, and let $\bY_{++}$ be the subset of strongly dominant coweights, i.e., coweights $\lambda$ such that $\la \lambda,\alpha\ra > 0$ for all $\alpha \in \fR^+$. Let also $\fR^\vee \subset \bY$ be the system of coroots, let $\fR^\vee_+ \subset \fR^\vee$ be the subset of positive coroots, and let $\fRs^\vee \subset \fR^\vee_+$ be the associated basis. We will denote by $\preceq$ the order on $\bY$ such that $\lambda \preceq \mu$ iff $\mu - \lambda$ is a sum of positive coroots; then $\bY_{\succeq 0} = \{\mu \in \bY \mid \mu \succeq 0\}$ is the free commutative monoid generated by $\fRs^\vee$.

Let 
\[
W_\aff = W_\fin \ltimes \Z\fR^\vee \quad \text{and} \quad W_\ext = W_\fin \ltimes \bY
\]
denote the affine and extended affine Weyl groups, respectively. We will need to distinguish the elements of $\bY$ from their images in $W_\ext$; for that, for $\lambda \in \bY$ we will write $t_\lambda$ for the corresponding element of $W_\ext$. It is a standard fact that there exists a unique subset $S_\aff \subset W_\aff$ such that $(W_\aff, S_\aff)$ is a Coxeter system with associated length function 
given by
\[
\ell(t_\lambda w) = \sum_{\substack{\alpha \in \fR_+ \\ w^{-1}(\alpha) \in \fR_+}} |\la \lambda, \alpha \ra| + \sum_{\substack{\alpha \in \fR_+ \\ w^{-1}(\alpha) \in -\fR_+}} |1 + \la \lambda,\alpha\ra|
\qquad\text{for $w \in W_\fin$, $\lambda \in \Z\fR^\vee$.}
\]

This formula makes sense more generally for $\lambda \in \bY$, which allows us to extend $\ell$ to $W_\ext$. (This extension still satisfies $\ell(w)=\ell(w^{-1})$ for any $w \in W_\ext$.) Note that for $\lambda \in \bY$ and $w \in W_\fin$ we have
\begin{equation}
\label{eqn:formulas-length}
 \ell(t_\lambda w) = 
 \begin{cases}
  \ell(t_\lambda) + \ell(w) & \text{if $\lambda \in \bY_+$;} \\
  \ell(t_\lambda) - \ell(w) & \text{if $\lambda \in -\bY_{++}$.}
 \end{cases}
\end{equation}

The subgroup $W_\aff \subset W_\ext$ is normal and,
setting $\Omega = \{w \in W_\ext \mid \ell(w)=0\}$, multiplication induces an isomorphism
\[
 \Omega \ltimes W_\aff \simto W_\ext,
\]
where $\Omega$ acts on $W_\aff$ by Coxeter group automorphisms. We extend the Bruhat order on $W_\aff$ to $W_\ext$ by declaring that, for $\omega,\omega' \in \Omega$ and $w,w' \in W_\aff$,
\[
 \omega w \leq \omega' w' \quad \Leftrightarrow \quad \omega=\omega' \text{ and } w \leq w'.
\]

\begin{rmk}
 The conventions we use here are opposite to those used in some references like~\cite{central} or~\cite{mr1}. In terms of alcove geometry, our ``fundamental alcove'' is $\{v \in \mathbb{R} \otimes_{\Z} \bY \mid \forall \alpha \in \fR_+, \, -1 \leq \langle v, \alpha \rangle \leq 0\}$.
 
 For instance, when $G=\mathrm{PGL}_2$ and with $T$, resp.~$B$, consisting of diagonal, resp.~upper triangular, matrices, we have a canonical identification $\bY=\Z$ such that $\bY_+=\Z_{\geq 0}$. Then the unique nontrivial element of length $0$ in $W_\ext$ is $st_1$ (where $s \in W$ is the unique nontrivial element), and we have $S_\aff = \{s, st_2\}$.
\end{rmk}


For $\lambda \in \bY$ we will denote by
$w_\lambda$ the unique element of minimal length in $t_\lambda W_\fin$.
If $v_\lambda$ is the unique element of $W_\fin$ of minimal length such that $v_\lambda(\lambda) \in \bY_+$, then we have
\[
w_\lambda = t_\lambda v_\lambda^{-1} 
\qquad\text{and}\qquad
\ell(w_\lambda) = \ell(t_\lambda) - \ell(v_\lambda),
\]
see e.g.~\cite[Lemma~2.4]{mr1}.
In particular, we have
\begin{equation}
\label{eqn:wmin-dom-antidom}
w_\lambda =
\begin{cases}
t_\lambda & \text{if $\lambda \in \bY_+$,} \\
t_\lambda w_\circ & \text{if $\lambda \in - \bY_{++}$}
\end{cases}
\end{equation}
(which of course is consistent with~\eqref{eqn:formulas-length}.)

Set $W_\ext^{+} := \{w \in W_\ext \mid \forall \mu \in \bY_+, \, \ell(t_\mu w) = \ell(t_\mu) + \ell(w)\}$. (Note that, in view of~\eqref{eqn:formulas-length}, all the elements of the form $t_\lambda w_\fin$ with $\lambda \in \bY_+$ and $w_\fin \in W_\fin$ belong to $W_\ext^+$.)
If $w,y \in W_\ext$, then as in~\cite[Lemma~2.4]{ar-model} the following conditions are equivalent:
\begin{enumerate}
\item
\label{it:si-order-1}
there exists $\mu \in \bY$ such that $t_\mu w$ and $t_\mu y$ belong to $W_\ext^+$ and $t_\mu w \leq t_\mu y$;
\item
\label{it:si-order-2}
for any $\mu \in \bY$ such that $t_\mu w$ and $t_\mu y$ belong to $W_\ext^+$, we have $t_\mu w \leq t_\mu y$.
\end{enumerate}
Since our present setting and claim are slightly different from those in~\cite{ar-model}, we briefly recall the argument. It is clear that the second condition implies the first one, since for any $w,y \in W_\ext$ there $\mu \in \bY$ such that $t_\mu w$ and $t_\mu y$ belong to $W_\ext^+$. Conversely, assume that the first condition is satisfied, and fix $\mu \in \bY$ which satisfies this condition. If $\mu' \in \bY$ is any element such that $t_\mu w$ and $t_\mu y$ belong to $W_\ext^+$, and if $\nu \in \bY_+$ is such that $(\eta + \mu)-\mu' \in \bY_+$, then using~\cite[Lemma~2.1]{ar-steinberg} and the definition of $W_\ext^+$ we see that $t_{\eta+\mu} w \leq t_{\eta+\mu} y$, and then that $t_{\mu'} w \leq t_{\mu'} y$.

We will write
\[
w \leq^\si y
\]
if the conditions~\eqref{it:si-order-1}--\eqref{it:si-order-2} above are satisfied. It is clear that this defines an order on $W_\ext$, which is called the \emph{semiinfinite order}.

\subsection{Iwahori orbits on \texorpdfstring{$\Fl$}{Fl} and \texorpdfstring{$\tFl$}{tFl}}


The orbits of $\Iw$ on $\Fl$ are naturally labelled by $W_\ext$, in the following way. If $w \in W_\ext$ and $w=t_\mu w_\fin$ for some $w_\fin \in W_\fin$, then we can consider the $\F$-point $w_\fin \Iw \in \Fl$, its image $z^\mu w_\fin \Iw$ under multiplication by $z^\mu$, and the $\Iw$-orbit $\Fl_w$ of this point. Then we have
\[
\Fl = \bigsqcup_{w \in W_\ext} \Fl_w.
\]
Moreover, it is a standard fact that each $\Fl_w$ is also an $\Iwu$-orbit, isomorphic to an affine space of dimension $\ell(w)$, and that for any $w \in W_\ext$ we have
\[
\overline{\Fl_w} = \bigsqcup_{y \leq w} \Fl_y.
\]
We will denote by
\[
j_w : \Fl_w \to \overline{\Fl_w}, \quad \overline{i}_w : \overline{\Fl_w} \to \Fl
\]
the embeddings, and set $i_w = \overline{i}_w \circ j_w$.


For $\alpha \in \fR$ and $n \in \Z$, we will denote by $U_{\alpha + n \hbar}$ the $1$-parameter subgroup in $\Loop G$ which identifies with the image of the morphism $x \mapsto u_\alpha(x z^n)$ for any choice of an isomorphism $u_\alpha$ between the additive group $\Ga$ and the root subgroup of $G$ associated with $\alpha$. We have $U_{\alpha + n\hbar} \subset \Iw$ iff either $\alpha \in \fR_+$ and $n \geq 0$, or $\alpha \in -\fR_+$ and $n>0$. We fix an arbitrary order on $\fR_+$. The following statement is standard.

\begin{lem}
\label{lem:orbits-Iw-Fl}
 Let $\mu \in \bY_{+}$ and $w_\fin \in W_\fin$, and set $w=t_\mu w_\fin$. Then multiplication and action on $z^\mu w_\fin \Iw$ induce an isomorphism of $\F$-schemes
 \[
  \left( \prod_{\substack{\alpha \in \fR_+ \\ w_{\fin}^{-1}(\alpha) \in \fR_+}} \prod_{n=0}^{\langle \alpha, \mu \rangle-1} U_{\alpha + n\hbar} \right) \times
  \left( \prod_{\substack{\alpha \in \fR_+ \\ w_{\fin}^{-1}(\alpha) \in -\fR_+}} \prod_{n=0}^{\langle \alpha, \mu \rangle} U_{\alpha + n\hbar} \right) \simto \Fl_{w}.
  \]
\end{lem}

In addition to the $\infty$-categories considered in~\S\ref{ss:Iw-eq-sheaves}, we will also consider for $w \in W_\ext$ the $\infty$-categories
\[
\Shv(\Iw \backslash \Fl_w), \quad \Shv(\Iw \backslash \overline{\Fl_w}), 
\quad \Shv(\Iwu \backslash \Fl_w), \quad \Shv(\Iwu \backslash \overline{\Fl_w}),
\]
and the push/pull functors associated with the maps $i_w$, $j_w$, $\overline{i}_w$.

The $\Iw$-orbits on $\tFl$ are also naturally labelled by $W_\ext$; in fact, setting $\tFl_w = q^{-1}(\Fl_w)$ for $w \in W_\ext$, each $\tFl_w$ is an $\Iw$-orbit, and we have
\[
\tFl = \bigsqcup_{w \in W_\ext} \tFl_w.
\]



\subsection{Orbits on \texorpdfstring{$\Gr$}{Gr}}

The $\Iw$-orbits on $\Gr$ are naturally labelled by $\bY$, in the following way.  For $\lambda \in \bY$, we will denote by $\Gr_\lambda$ the orbit of $z^\lambda \Loop^+ G \in \Gr$.
Each $\Gr_\lambda$ is also an $\Iwu$-orbit, is isomorphic to an affine space of dimension $\ell(w_\lambda)$, and for any $\lambda \in \bY$ we have
\[
\overline{\Gr_\lambda} = \bigsqcup_{\substack{\mu \in \bY \\ w_\mu \leq w_\lambda}} \Gr_\mu.
\]
We will denote by
\[
j_\lambda : \Gr_\lambda \to \overline{\Gr_\lambda}, \quad \overline{i}_\lambda : \overline{\Gr_\lambda} \to \Gr
\]
the embeddings, and set $i_\lambda = \overline{i}_\lambda \circ j_\lambda$.


The following statement is the analogue for $\Gr$ of Lemma~\ref{lem:orbits-Iw-Fl}.

\begin{lem}
\label{lem:orbits-Iw-Gr}
 Let $\mu \in \bY_{+}$. Then multiplication and action on $z^\mu \Loop^+G$ induce an isomorphism of $\F$-schemes
 \[
  \prod_{\alpha \in \fR_+} \prod_{n=0}^{\langle \alpha, \mu \rangle-1} U_{\alpha + n\hbar} \simto \Gr_{\mu}.
  \]
\end{lem}

If $w=t_\lambda w_\fin$ with $\lambda \in \bY$ and $w_\fin \in W_\fin$, then
we have $p(\Fl_w) = \Gr_\lambda$; more precisely,
\begin{equation}
\label{eqn:reldim}
p_{|\Fl_w}: \Fl_w \to \Gr_\lambda
\qquad
\begin{array}{@{}c@{}}
\text{is a trivial affine space bundle of} \\
\text{relative dimension $\ell(w) - \ell(w_\lambda)$.}
\end{array}
\end{equation}

The $\Loop^+G$-orbits on $\Gr$ are naturally labelled by $\bY_+$, and we will denote by $\Gr^{\lambda}$ the orbit attached to $\lambda$. For any $\lambda \in \bY_+$ the $\Iw$-orbit $\Gr_\lambda$ is open and dense in $\Gr^{\lambda}$, and we have
\[
 \dim(\Gr^\lambda) = \langle \lambda, 2\rho \rangle.
\]

In this setting also one can consider the categories of sheaves on $\Iw \backslash \Gr_\lambda$ and its variants involving $\overline{\Gr_\lambda}$ and $\Iwu$, and we have push/pull functors relating these categories.

\subsection{Semiinfinite orbits}
\label{ss:si-orbits}

The orbits of $\siIw$ on $\tFl$, $\Fl$ or $\Gr$ are called \emph{semiinfinite orbits}.  The semiinfinite orbits on $\Fl$, resp.~on $\tFl$, resp.~on $\Gr$, are again labelled naturally by $W_\ext$, resp.~$W_\ext$, resp.~$\bY$. For $w \in W_\ext$, resp.~$w \in W_\ext$, resp.~$\lambda \in \bY$, the corresponding semiinfinite orbit is denoted by
\[
\fS_w \subset \Fl, \quad
\text{resp.} \quad \tfS_w \subset \tFl, \quad
\text{resp.} \quad \rS_\lambda \subset \Gr.
\]
Here we have $\tfS_w = q^{-1}(\fS_w)$.

\begin{rmk}
\label{rmk:orbits-siIw}
Now that this notation is introduced, one can state as a consequence of Lemma~\ref{lem:orbits-Iw-Fl} that $\Fl_w \subset \fS_w$ if $w=t_\mu w_\fin$ with $\mu \in \bY_+$ and $w_\fin \in W_\fin$. Similarly, Lemma~\ref{lem:orbits-Iw-Gr} implies that $\Gr_\mu \subset \rS_\mu$ if $\mu \in \bY_+$.
\end{rmk}

These orbits do not have well-defined dimensions or codimensions.  Nevertheless, we assign to each $\siIw$-orbit on $\Fl$ or $\Gr$ an integer that will be called its \emph{pseudodimension}, 
by the following formulas:
\begin{align*}
\psdim (\fS_w) &= \la \lambda, 2\rho\ra + \ell(w_\fin) &&\text{if $w = t_\lambda w_\fin$ with $w_\fin \in W_\fin$ and $\lambda \in \bY$,} \\ 
\psdim (\rS_\lambda) &= \la \lambda, 2\rho\ra && \text{for $\lambda \in \bY$.}
\end{align*}

\begin{rmk}
In parallel with the construction of the semiinfinite order $\leq^\si$ recalled in~\S\ref{ss:Weyl}, one can define the semiinfinite length $\ell^\si$ on $W_\ext$ by setting $\ell^\si(w) = \ell(t_\mu w)-\ell(t_\mu)$ for any $\mu \in \bY_+$ such that $t_\mu w \in W_\ext^+$. (Note that this ``length'' can take negative values.) Then, for any $w \in W_\ext$ we have $\psdim(\fS_w) = \ell^\si(w)$.
\end{rmk}

If $w=t_\lambda w_\fin$ with $\lambda \in \bY$ and $w_\fin \in W_\fin$, then we have $p(\fS_w)=\rS_\lambda$; more precisely,
\begin{equation}
\label{eqn:reldim-semiinf}
p_{|\fS_w}: \fS_w \to \rS_\lambda
\qquad
\begin{array}{@{}c@{}}
\text{is a trivial affine space bundle of} \\
\text{relative dimension $\ell(w_\fin)$.}
\end{array}
\end{equation}

We will also consider the ind-closures
\[
\overline{\fS_w} \subset \Fl,
\qquad
\overline{\rS_\lambda} \subset \Gr.
\]
Here it is a standard fact (see~\cite[Proposition~3.1]{mv}) that
\begin{equation}
\label{eqn:closure-Slambda}
\overline{\rS_\lambda} = \bigsqcup_{\mu \preceq \lambda} \rS_\mu.
\end{equation}
We will denote by
\[
\bj_w : \fS_w \to \overline{\fS_w}, \quad \overline{\bi}_w : \overline{\fS_w} \to \Fl, \quad
\bj_\lambda : \rS_\lambda \to \overline{\rS_\lambda}, \quad \overline{\bi}_\lambda : \overline{\rS_\lambda} \to \Gr
\]
the embeddings, and set $\bi_w = \overline{\bi}_w \circ \bj_w$, $\bi_\lambda = \overline{\bi}_\lambda \circ \bj_\lambda$.

Following the same procedure as in~\S\ref{ss:si-sheaves-Fl}, one can define the categories
\[
\Shv(\siIw \backslash \fS_w), \quad \Shv(\siIwu \backslash \fS_w), \quad \Shv(\siIw \backslash \overline{\fS_w}), \quad \Shv(\siIwu \backslash \overline{\fS_w})
\]
for $w \in W_\ext$,
and the analogous categories for $\rS_\lambda$ and $\overline{\rS_\lambda}$ for $\lambda \in \bY$. Denoting by 
\[
\mathbf{k}_w : \Spec(\F) \to \fS_w, \quad \mathbf{k}_\lambda : \Spec(\F) \to \rS_\lambda
\]
the inclusions of the base points (for $w \in W_\ext$ and $\lambda \in \bY$ respectively) and by
$\Vect_\bk$ the $\infty$-category of (complexes of) $\bk$-vector spaces, as in~\cite[Proposition~1.3.5]{gaitsgory} one sees that the functors $(\mathbf{k}_w)^!$ and $(\mathbf{k}_\lambda)^!$ induce equivalences
\begin{equation}
\label{eqn:si-sheaves-orbits}
\Shv(\siIwu \backslash \fS_w) \cong \Vect_\bk, \quad \Shv(\siIwu \backslash \rS_\lambda) \cong \Vect_\bk
\end{equation}
for any $w \in W_\ext$ and $\lambda \in \bY$; under these equivalences, $\omega_{\fS_w}$ and $\omega_{\rS_\lambda}$ correspond to $\bk \in \Vect_\bk$. 
We have push/pull functors between these categories associated with the maps $\bi_w$, $\overline{\bi}_w$, $\bj_w$, $\bi_\lambda$, $\overline{\bi}_\lambda$, $\bj_\lambda$.

We define the \emph{perverse t-structure} on $\Shv(\siIwu \backslash \fS_w)$ to be the cohomological shift of the transport of the tautological
t-structure on $\Vect_\bk$ along the left-hand equivalence in~\eqref{eqn:si-sheaves-orbits} such that the object
\[
\omega_{\fS_w}[-\psdim (\fS_w)]
\]
belongs to the heart. (Here, as in~\S\ref{ss:functorialities}, to simplify notation we write $\omega_{\fS_w}$ for $\omega_{\Iwu \backslash \fS_w}$. Similar conventions will be used for variants of this construction below.) This t-structure will be denoted
\[
(\Shv(\siIwu \backslash \fS_w)^{\leq 0}, \Shv(\siIwu \backslash \fS_w)^{\geq 0}),
\]
and its heart will be denoted $\Shv(\siIwu \backslash \fS_w)^\heartsuit$. We similarly define the perverse t-structure on $\Shv(\siIwu \backslash \rS_\lambda)$ in such a way that 
\[
\omega_{\rS_\lambda}[-\psdim (\rS_\lambda)]
\]
belongs to the heart, and will use similar notation. In both cases, the corresponding cohomology functors will be denoted $\pH^n(-)$ ($n \in \Z$). Any object $\mathcal{F}$ in the heart of one of these t-structures is a direct sum of copies of the appropriate shift of the dualizing sheaf; this multiplicity of the dualizing sheaf will be called the \emph{rank} of $\mathcal{F}$, and denoted $\rank (\mathcal{F})$. (For most objects of interest this quantity will be finite, but in general it can be infinite.)

\begin{rmk}
\label{rmk:perverse-degrees-costalk}
More concretely, an object $\cF \in \Shv(\siIwu \backslash \fS_w)$ lies in $\Shv(\siIwu \backslash \fS_w)^{\leq 0}$, resp.~in $\Shv(\siIwu \backslash \fS_w)^{\geq 0}$ if and only if the complex $(\mathbf{k}_w)^!\cF$ is concentrated in degrees $\leq \psdim (\fS_w)$, resp.~in degrees $\geq \psdim (\fS_w)$, and similarly for sheaves on $\rS_\lambda$.
\end{rmk}

Similar considerations hold for the $\siIw$-equivariant versions of these $\infty$-categories and functors; in this case we have
\begin{equation}
\label{eqn:si-sheaves-orbits-2}
\Shv(\siIw \backslash \fS_w) \cong \Shv(T \backslash \pt), \quad \Shv(\siIwu \backslash \rS_\lambda) \cong \Shv(T \backslash \pt).
\end{equation}
We normalize these equivalences in such a way that the forgetful functors
\begin{gather}
\label{eqn:oblv-Sw}
\oblv^{\siIw | \siIwu} : \Shv(\siIw \backslash \fS_w) \to \Shv(\siIwu \backslash \fS_w) \\
\label{eqn:oblv-Slambda}
\oblv^{\siIw | \siIwu} : \Shv(\siIw \backslash \rS_\lambda) \to \Shv(\siIwu \backslash \rS_\lambda)
\end{gather}
corresponds to the forgetful functor
\[
\oblv^T : \Shv(T \backslash \pt) \to \Shv(\pt)=\Vect_\bk
\]
under these equivalences and those of~\eqref{eqn:si-sheaves-orbits}. We will also use \emph{perverse t-structure} to mean the unique t-structure on $\Shv(\siIw \backslash \fS_w)$, resp.~$\Shv(\siIw \backslash \rS_\lambda)$, such that~\eqref{eqn:oblv-Sw}, resp.~\eqref{eqn:oblv-Slambda}, is t-exact with respect to the perverse t-structures. We will use similar notation as above for these t-structures. 
The definition of the rank of an object in the heart of the perverse t-structure translates verbatim to this setting.

\begin{rmk}
\label{rmk:colimits-t-str}
   Since both parts of the standard t-structure on $\Vect_\bk$ are stable under filtered colimits, the same property holds for the perverse t-structures on $\Shv(\siIw \backslash \fS_w)$, $\Shv(\siIwu \backslash \fS_w)$, $\Shv(\siIw \backslash \rS_\lambda)$, $\Shv(\siIwu \backslash \rS_\lambda)$.
\end{rmk}


Consider a finite union $Z$ of closures of orbits $\fS_w$ ($w \in W_\ext$), and denote by $i : Z \to \Fl$ the (closed) embedding. If $w \in W_\ext$ is such that $\fS_w \subset Z$, for simplicity we will still denote by $\bi_w : \fS_w \to Z$ the embedding. Then one can consider as above the $\infty$-categories $\Shv(\siIw \backslash Z)$ and $\Shv(\siIwu \backslash Z)$. Below we will require the following statement.

\begin{lem}
\label{lem:*rest-0}
 Let $\cF$ be in $\Shv(\siIw \backslash Z)$ or $\Shv(\siIwu \backslash Z)$, and assume that for any $w \in W_\ext$ such that $\fS_w \subset Z$ we have $(\bi_w)^* \cF=0$. Then $\cF=0$.
\end{lem}

\begin{proof}
Writing $Z$ as an increasing union of open finite unions of orbits and
 using the isomorphism~\eqref{eqn:colim-opens} one reduces the lemma to a similar claim for sheaves on a (locally closed) finite union of orbits; in this case the result is easily checked using the fiber sequences in~\eqref{eqn:fiber-seq-open-closed}.
\end{proof}

\begin{rmk}
The analogue of Lemma~\ref{lem:*rest-0} for $Z=\Fl$ does \emph{not} hold, as shown by the example discussed in
Remark~\ref{rmk:dualizing-infty}\eqref{it:dualizing-vanishing-stalks} below.
\end{rmk}


\section{Complexes attached to \texorpdfstring{$\Iw$}{I}-orbits}
\label{sec:complexes-I-orbits}


\subsection{Standard and costandard complexes on \texorpdfstring{$\Fl$}{Fl}}
\label{ss:standard-costandard-Fl}


For $w \in W_\ext$ we set
\[
\NFl_w := (i_w)_* \omega_{\Fl_w}[-\ell(w)], \quad
\DFl_w := (i_w)_! \omega_{\Fl_w}[-\ell(w)].
\]
When $w=e$ the objects $\DFl_e$ and $\NFl_e$ are isomorphic, and coincide with the unit object $\delta_{\Fl}$ for the monoidal product $\star^{\Iw}$. For any $w,y \in W_\ext$ and $n \in \Z$ we have
\begin{equation}
\label{eqn:Hom-D-N-Fl}
\Hom_{\Shv(\Iw \backslash \Fl)}(\DFl_w, \NFl_y [n]) \cong \begin{cases}
\mathsf{H}^n_T(\pt; \bk) & \text{if $w=y$ and $n \in 2\Z_{\geq 0}$;} \\
0 & \text{otherwise,}
\end{cases}
\end{equation}
where $\mathsf{H}^n_T(\pt; \bk)$ identifies with the $n/2$-th symmetric power of the $\bk$-vector space $\bk \otimes_\Z \bX$ for $n \in 2\Z_{\geq 0}$.
In the $\Iwu$-equivariant context we have
\begin{equation}
\label{eqn:Hom-D-N-Fl-oblv}
\Hom_{\Shv(\Iwu \backslash \Fl)}(\oblv^{\Iw | \Iwu}(\DFl_w), \oblv^{\Iw | \Iwu}(\NFl_y) [n]) \cong \begin{cases}
\bk & \text{if $w=y$ and $n = 0$;} \\
0 & \text{otherwise.}
\end{cases}
\end{equation}

Recall that we have the perverse t-structures 
on the stable $\infty$-categories $\Shv(\Iw \backslash \Fl)$ and $\Shv(\Iwu \backslash \Fl)$ from~\S\ref{ss:Iw-eq-sheaves}, which restrict to t-structures on the full subcategories $\Shvfg(\Iw \backslash \Fl)$ and $\Shvfg(\Iwu \backslash \Fl)$ respectively.
Since each immersion $\Fl_w \to \overline{\Fl_w}$ is affine, the complexes $\DFl_w$ and $\NFl_w$ are perverse sheaves. For any $w \in W_\ext$ there exists (up to scalar) a unique nonzero morphism $\DFl_w \to \NFl_w$, whose image will be denoted $\ICFl_w$. It is a simple perverse sheaf, namely the intersection cohomology complex associated with the stratum $\Fl_w$ (and the trivial local system). It is also the unique simple quotient of $\DFl_w$, and the unique simple subobject of $\NFl_w$.

The following lemma gathers standard properties of standard and costandard complexes. For proofs, see e.g.~\cite[Lemmas~4.1.4,~4.1.7 and~4.1.9]{central}.

\begin{lem}
\phantomsection
\label{lem:standards-costandards}
 \begin{enumerate}
  \item 
  \label{it:standards-costandards-1}
  For any $w,y \in W_\ext$ such that $\ell(wy)=\ell(w)+\ell(y)$ we have canonical isomorphisms
\[
\DFl_w \star^{\Iw} \DFl_y \cong \DFl_{wy}, \qquad
\NFl_w \star^{\Iw} \NFl_y \cong \NFl_{wy}.
\]
\item
\label{it:standards-costandards-2}
 For any $w \in W_\ext$ there exist canonical isomorphisms
 \[
  \DFl_w \star^{\Iw} \NFl_{w^{-1}} \cong \delta_{\Fl} \cong \NFl_{w^{-1}} \star^{\Iw} \DFl_w. 
 \]
 \item 
  \label{it:standards-costandards-3}
  For any $w,y \in W_\ext$, the complexes
  \[
   \DFl_w \star^{\Iw} \NFl_y \quad \text{and} \quad \NFl_w \star^{\Iw} \DFl_y
  \]
are perverse sheaves, and are supported on $\overline{\Fl_{wy}}$.
 \end{enumerate}
\end{lem}

Below we will also use the following properties (which hold both in $\Iw$-equivariant and in $\Iwu$-equivariant perverse sheaves).

\begin{lem}
\phantomsection
\label{lem:fl-socle}
\begin{enumerate}
\item 
\label{it:fl-socle-1}
Let $w \in W_\ext$, and let $\omega$ be the only element in $\Omega \cap wW_{\aff}$. Then $\DFl_w$, resp.~$\NFl_w$, admits a unique simple subobject, resp.~quotient. This simple object is isomorphic to $\IC_{\omega}$, and the cokernel of the embedding $\IC_{\omega} \hookrightarrow \DFl_w$, resp.~the kernel of the surjection $\NFl_w \twoheadrightarrow \IC_{\omega}$, is a finite extension of perverse sheaves of the form $\IC_y$ with $y \in wW_\aff$, $y \leq w$, $y \neq \omega$.
\item 
\label{it:fl-socle-2}
For any $v, w \in W_\ext$, we have
\[
\dim \Hom(\DFl_v, \DFl_w) =
\begin{cases}
1 & \text{if $\Fl_v \subset \overline{\Fl_w}$,} \\
0 & \text{otherwise.}
\end{cases}
\]
\end{enumerate}
\end{lem}

\begin{proof}
 \eqref{it:fl-socle-1} follows from the same arguments as in~\cite[\S 2.1]{bbm}. \eqref{it:fl-socle-2} is a standard consequence, as e.g.~in~\cite[Proposition~4.10]{modrap2}.
\end{proof}

\subsection{Standard and costandard complexes on \texorpdfstring{$\Gr$}{Gr}}
\label{ss:stand-Gr}

Similarly to the setting considered in~\S\ref{ss:standard-costandard-Fl}, in the $\infty$-category $\Shv(\Iw \backslash \Gr)$, for any $\lambda \in \bY$ we have objects
\[
\DGr_\lambda
\quad \text{and} \quad
\NGr_\lambda,
\]
supported on $\overline{\Gr_\lambda}$, which are obtained by $!$- and $*$-pushforward of $\omega_{\Gr_\lambda}[-\ell(w_\lambda)]$ under the map $i_\lambda$ respectively. By~\eqref{eqn:reldim}, we have
\[
\begin{aligned}
p_*\DFl_w &= \DGr_\lambda[\ell(w_\lambda) - \ell(w)], \\
p_*\NFl_w &= \NGr_\lambda[\ell(w) - \ell(w_\lambda)],
\end{aligned}
\qquad
\text{for $\lambda \in \bY$ and $w \in t_\lambda W_\fin$.}
\]
In particular, for $w = t_\lambda w_\fin$ with $\lambda \in \bY$ and $w_\fin \in W_\fin$, in view of~\eqref{eqn:wmin-dom-antidom} we have
\begin{gather}
\label{eqn:push-DN-dom}
p_*\DFl_w \cong \DGr_\lambda[-\ell(w_\fin)], \ \ p_*\NFl_w \cong \NGr_\lambda[\ell(w_\fin)] \quad \text{if $\lambda \in \bY_+$,} \\
\label{eqn:push-DN-antidom}
p_*\DFl_w \cong \DGr_\lambda[\ell(w_\fin)-\ell(w_\circ)], \ \ p_*\NFl_w \cong \NGr_\lambda[\ell(w_\circ) - \ell(w_\fin)] \quad \text{if $\lambda \in -\bY_{++}$.}
\end{gather}
For $\lambda, \mu \in \bY$ and $n \in \Z$, as in~\eqref{eqn:Hom-D-N-Fl} we have
\begin{equation}
\label{eqn:Hom-D-N-Gr}
\Hom_{\Shv(\Iw \backslash \Gr)}(\DGr_\lambda, \NGr_\mu [n]) \cong \begin{cases}
\mathsf{H}^n_T(\pt; \bk) & \text{if $\lambda=\mu$ and $n \in 2\Z_{\geq 0}$;} \\
0 & \text{otherwise,}
\end{cases}
\end{equation}
and as in~\eqref{eqn:Hom-D-N-Fl-oblv} we have
\begin{equation}
\label{eqn:Hom-D-N-Gr-oblv}
\Hom_{\Shv(\Iwu \backslash \Gr)}(\oblv^{\Iw | \Iwu}(\DGr_\lambda), \oblv^{\Iw | \Iwu}(\NGr_\mu) [n]) \cong \begin{cases}
\bk & \text{if $\lambda=\mu$ and $n=0$;} \\
0 & \text{otherwise.}
\end{cases}
\end{equation}


With respect to the perverse t-structures on $\Shv(\Iw \backslash \Gr)$ and $\Shv(\Iwu \backslash \Gr)$, the objects $\DGr_\lambda$ and $\NGr_\lambda$ are perverse. For any $\lambda \in \bY$, the image of the unique (up to scalar) nonzero morphism $\DGr_\lambda \to \NGr_\lambda$ will be denoted
$\ICGr_\lambda$; once again, this is a simple perverse sheaf, namely the intersection cohomology complex associated with the stratum $\Gr_\lambda$ (and the constant local system).
For $\lambda=0$ we have $\DGr_0=\NGr_0=\ICGr_0$, and this object will be denoted $\delta_{\Gr}$; it is the image of the unit object in the monoidal $\infty$-category $(\Shv(\Loop^+ G \backslash \Gr), \star^{\Loop^+ G})$.

It is a standard fact that the functor $p^\dag$ (see~\S\ref{ss:Iw-eq-sheaves})
is t-exact for the perverse t-structures, that its restriction to perverse sheaves is fully faithful, and that for any $\lambda \in \bY$ we have
\[
 p^\dag \ICGr_\lambda \cong \ICFl_{w_\lambda w_\circ}.
\]

\subsection{Spherical complexes on \texorpdfstring{$\Gr$}{Gr}}
\label{ss:spherical-complexes}

Recall the $\infty$-category $\Shv(\Loop^+ G \backslash \Gr)$ considered in Remark~\ref{rmk:Iw-eq-variants}\eqref{it:sph}. On this $\infty$-category too we have a perverse t-structure, which restricts to a t-structure on the full subcategory $\Shvfg(\Loop^+ G \backslash \Gr)$, and such that the functor $\oblv^{\Loop^+G | \Iw}$ is t-exact. Its heart will be denoted $\Shv(\Loop^+ G \backslash \Gr)^\heartsuit$. By~\cite[Proposition~4.2]{mv} this heart is stable under the convolution bifunctor $\star^{\Loop^+ G}$. For $\lambda \in \bY_+$ we will denote by $\cI_*(\lambda)$, resp.~$\cI_!(\lambda)$, the $0$-th perverse cohomology of the $*$-pushforward, resp.~$!$-pushforward, associated with the embedding $\Gr^\lambda \to \Gr$ of the constant local system on $\Gr^\lambda$ placed in degree $-\dim(\Gr^\lambda)$. Then the image of the unique (up to scalar) nonzero morphism $\cI_!(\lambda) \to \cI_*(\lambda)$ is $\ICGr_\lambda$.

Let us denote by $G^\vee_\bk$ the split reductive group scheme over $\bk$ which is Langlands dual to $G$, and by 
$\Rep(G^\vee_\bk)^\heartsuit$ the abelian category of 
finite-dimensional algebraic $G^\vee_\bk$-modules. Then the geometric Satake equivalence of~\cite{mv} provides an equivalence of monoidal abelian categories
\[
\Sat : \left( \Shvfg(\Loop^+ G \backslash \Gr)^\heartsuit, \star^{\Loop^+ G} \right) \simto \left( \Rep(G^\vee_\bk)^\heartsuit, \otimes \right).
\]
The construction of this equivalence determines a canonical maximal torus $T^\vee_\bk \subset G^\vee_\bk$, whose lattice of weights is $\bY$.

Recall that, for $\lambda \in \bY_+$, the $G^\vee_\bk$-module
\[
\Sat(\cI_*(\lambda)), \quad \text{resp.} \quad \Sat(\cI_!(\lambda)), \quad \text{resp.} \quad \Sat(\ICGr_\lambda),
\]
is the dual Weyl (i.e.~induced) module $\mathsf{N}^{\vee}(\lambda)$, resp.~Weyl module $\mathsf{M}^{\vee}(\lambda)$, resp.~simple module, of highest weight $\lambda$, see~\cite[\S 13]{mv} or~\cite[Lemma~1.5.2 and Remark 1.5.3(1)]{central}.

\subsection{Wakimoto sheaves on \texorpdfstring{$\Fl$}{Fl}}
\label{ss:Waki-Fl}

In the abelian category $\Shv(\Iw \backslash \Fl)^\heartsuit$ we have a collection of objects
\[
 (\WFl_{t_\lambda} : \lambda \in \bY)
\]
called the \emph{Wakimoto sheaves}, whose definition is due to Mirkovi{\'c}. 
These objects satisfy (and are determined by) the following properties:
\begin{itemize}
\item
if $\lambda \in \bY_+$, resp.~$\lambda \in -\bY_+$, then we have $\WFl_{t_\lambda} = \DFl_{t_\lambda}$, resp.~$\WFl_{t_\lambda} = \NFl_{t_\lambda}$ (in particular
we have $\WFl_{t_0} = \delta_{\Fl}$);
\item
if $\lambda, \mu \in \bY$, we have a canonical isomorphism
\begin{equation}
 \label{eqn:convolution-Waki-coweights}
 \WFl_{t_\lambda} \star^{\Iw} \WFl_{t_\mu} \cong \WFl_{t_{\lambda+\mu}}.
\end{equation}
\end{itemize}
For details about this construction, see~\cite[\S 4.2]{central}. (Note however that in~\cite{central} the fixed Iwahori subgroup is associated with the \emph{negative} Borel subgroup, while here it is attached with the \emph{positive} Borel subgroup.)

We extend this definition to general elements of $W_\ext$ as follows:
if $w = t_\lambda w_\fin \in W_\ext$ with $\lambda \in \bY$ and $w_\fin \in W_\fin$, we set
\[
\WFl_w = \WFl_{t_\lambda} \star^{\Iw} \DFl_{w_\fin}.
\]
Here, using~\eqref{eqn:formulas-length} and Lemma~\ref{lem:standards-costandards} we see that if $\mu, \lambda+\mu \in \bY_+$ we have 
\begin{equation}
\label{eqn:formula-Wakiw-1}
 \WFl_w \cong \NFl_{t_{-\mu}} \star^{\Iw} \DFl_{t_{\lambda+\mu} w_\fin},
\end{equation}
and if $\mu$ is furthermore strongly dominant we have
\begin{equation}
\label{eqn:formula-Wakiw-2}
 \WFl_w \cong \DFl_{t_{\lambda+\mu}} \star^{\Iw} \NFl_{t_{-\mu} w_\fin}.
\end{equation}
In particular, this object is perverse and supported on $\overline{\Fl_w}$ by Lemma~\ref{lem:standards-costandards}\eqref{it:standards-costandards-3}.


Some immediate properties of these objects are gathered in the following lemma. (Here, \eqref{it:Waki-4} involves the semiinfinite order on $W_\ext$, whose definition is recalled in~\S\ref{ss:Weyl}.)

\begin{lem}
\phantomsection
\label{lem:properties-Waki}
\begin{enumerate}
\item 
\label{it:Waki-1}
For any $\lambda \in \bY$ and any $w \in W_\ext$ we have a canonical isomorphism $\WFl_{t_\lambda} \star^{\Iw} \WFl_w \cong \WFl_{t_\lambda w}$.
\item 
\label{it:Waki-2}
If $\lambda \in \bY_+$ and $w_\fin \in W_\fin$, then 
$\WFl_{t_\lambda w_\fin} \cong \DFl_{t_\lambda w_\fin}$.
\item
\label{it:Waki-3}
If $\lambda \in -\bY_{++}$ and $w_\fin \in W_\fin$, then $\WFl_{t_\lambda w_\fin} \cong \NFl_{t_\lambda w_\fin}$.
\item
\label{it:Waki-4}
For $w,y \in W_\ext$ we have
\[
\dim \Hom_{\Shv(\Iw \backslash \Fl)}(\WFl_w, \WFl_y) = \begin{cases}
1 & \text{if $w \leq^\si y$;} \\
0 & \text{otherwise.}
\end{cases}
\]
\end{enumerate}
\end{lem}

\begin{proof}
 \eqref{it:Waki-1} is a direct consequence of~\eqref{eqn:convolution-Waki-coweights}.
\eqref{it:Waki-2} and~\eqref{it:Waki-3} are special cases of~\eqref{eqn:formula-Wakiw-1} and~\eqref{eqn:formula-Wakiw-2} respectively.

For~\eqref{it:Waki-4}, we write $w=t_\lambda w_\fin$ and $y=t_{\mu} y_\fin$ with $\lambda,\mu \in \bY$ and $w_\fin, y_\fin \in W_\fin$, and choose $\nu \in \bY$ such that $\lambda+\nu$ and $\mu + \nu$ are dominant. Then we have
\[
\WFl_w \cong \NFl_{t_{-\nu}} \star^{\Iw} \DFl_{t_\nu w}, \quad
\WFl_y \cong \NFl_{t_{-\nu}} \star^{\Iw} \DFl_{t_\nu y},
\]
so that by the invertibility of $\NFl_{t_{-\nu}}$ (see Lemma~\ref{lem:standards-costandards}\eqref{it:standards-costandards-2}) and Lemma~\ref{lem:fl-socle}\eqref{it:fl-socle-2} we have
\[
\dim \Hom_{\Shv(\Iw \backslash \Fl)}(\WFl_w, \WFl_y) = \begin{cases}
1 & \text{if $t_\nu w \leq t_\nu y$;} \\
0 & \text{otherwise.}
\end{cases}
\]
This is equivalent to the condition in the lemma by definition of $\leq^\si$.
\end{proof}


\subsection{A subcategory of constructible sheaves}
\label{ss:subcat-Shvc}

For later use, in this subsection we study certain full subcategories in $\Shvfg(\Iw \backslash \tFl)$ and $\Shvfg(\Iw \backslash \Fl)$ characterized in terms of Wakimoto objects. Here we insist that, contrary to the situation in most subsections above, we work in the $\infty$-categories $\Shvfg$ and not $\Shv$. We will use the standard fact (already recalled above for $\Fl$) that the perverse t-structures on the $\infty$-categories $\Shv(\Iw \backslash \tFl)$ and $\Shv(\Iw \backslash \Fl)$ restrict to t-structures on the full subcategories $\Shvfg(\Iw \backslash \tFl)$ and $\Shvfg(\Iw \backslash \Fl)$.

\begin{lem}
\phantomsection
\label{lem:subcat}
\begin{enumerate}
\item
\label{it:subcat-1}
For any $\cF$ in $\Shvfg(\Iw \backslash \Fl)$ the following conditions are equivalent:
\begin{enumerate}
\item
\label{it:subcat-1-1}
for any $\lambda \in \bY$ the complex $\WFl_{t_\lambda} \star^{\Iw} \cF$ belongs to the ``$\leq 0$'' part of the perverse t-structure;
\item
\label{it:subcat-1-2}
there exists $\lambda \in \bY_{++}$ and $N \in \Z$ such that for any $n \geq N$ the complex $\WFl_{t_{n\lambda}} \star^{\Iw} \cF$ belongs to the ``$\leq 0$'' part of the perverse t-structure;
\item
\label{it:subcat-1-3}
$\cF$ is a (finite) iterated extension of the objects of the form $\WFl_w[n]$ with $w \in W_\ext$ and $n \in \Z_{\geq 0}$.
\end{enumerate}
\item
\label{it:subcat-2}
For any $\cF$ in $\Shvfg(\Iw \backslash \tFl)$ the following conditions are equivalent:
\begin{enumerate}
\item
for any $\lambda \in \bY$ the complex $\WFl_{t_\lambda} \star^{\Iw} \cF$ belongs to the ``$\leq 0$'' part of the perverse t-structure;
\item
there exists $\lambda \in \bY_{++}$ and $N \in \Z$ such that for any $n \geq N$ the complex $\WFl_{t_{n\lambda}} \star^{\Iw} \cF$ belongs to the ``$\leq 0$'' part of the perverse t-structure;
\item
$\cF$ is a (finite) iterated extension of the objects of the form $q^!(\WFl_w)[n]$ with $w \in W_\ext$ and $n \in \Z_{\geq 0}$.
\end{enumerate}
\end{enumerate}
\end{lem}

\begin{proof}
We will explain the proof of~\eqref{it:subcat-1}; the statement in~\eqref{it:subcat-2} can be justified similarly, identifying $\Shvfg(\Iw \backslash \tFl)$ with $\Shvfg(\Fl' / \Iwu)$, see Remark~\ref{rmk:Iw-eq-variants}\eqref{it:Ieq-cat-leftquot}.

So, we fix $\cF$ in $\Shvfg(\Iw \backslash \Fl)$. It is clear that if $\cF$ satisfies~\eqref{it:subcat-1-1}, then it satisfies~\eqref{it:subcat-1-2}. Since Wakimoto sheaves are perverse (see~\S\ref{ss:Waki-Fl}) and satisfy the formula in Lemma~\ref{lem:properties-Waki}\eqref{it:Waki-1}, it is clear also that if $\cF$ satisfies~\eqref{it:subcat-1-3} then it satisfies~\eqref{it:subcat-1-1}.

To conclude, we assume that $\cF$ satisfies~\eqref{it:subcat-1-2}, and fix data as in this condition. By~\cite[Proposition~4.4.4]{central} (a statement due to Arkhipov--Bezrukavnikov) there exists a finite subset $A \subset W_\ext$ such that for all $x,y \in W_\ext$ we have
\[
 (i_x)^*(\Delta_y \star^{\Iw} \cF) \neq 0 \quad \Rightarrow \quad x \in y \cdot A.
\]
Choose $n \geq N$ such that $t_{n \lambda} \cdot A \subset \{t_\mu w_\fin : w_\fin \in W_\fin, \, \mu \in \bY_+\}$. Then
\[
\Delta_{t_\lambda} \star^{\Iw} \cF = \WFl_{t_\lambda} \star^{\Iw} \cF
\]
is concentrated in nonpositive perverse degrees by assumption. By standard arguments (see e.g.~\cite[Lemma~4.4.3]{central}) it follows that this complex belongs to the full subcategory of $\Shvfg(\Iw \backslash \Fl)$ generated under extensions by the objects $\Delta_y[n]$ with $y \in \{t_\mu w_\fin  : w_\fin \in W_\fin, \, \mu \in \bY_+\}$ and $n \in \Z_{\geq 0}$. For such $y$, by Lemma~\ref{lem:properties-Waki}\eqref{it:Waki-2} we have $\Delta_y = \WFl_y$. Convolving on the left with $\WFl_{t_{-\lambda}}$ and using Lemma~\ref{lem:properties-Waki}\eqref{it:Waki-1} once again, we deduce that $\cF$ satisfies~\eqref{it:subcat-1-3}, which concludes the proof.
\end{proof}

\subsection{Dual Wakimoto sheaves}
\label{ss:Waki-Fl-dual}

In this subsection we introduce a family of objects in $\Shv(\Iw \backslash \Fl)$ which is ``homologically dual'' to the family of Wakimoto sheaves, and whose construction is due (in the setting of representations of affine Kac--Moody algebras) to Gaitsgory in~\cite{gaitsgory-ext}. These objects do not belong to the subcategory $\Shvfg(\Iw \backslash \Fl)$, so that they cannot be seen in the usual framework of bounded derived categories of constructible sheaves.

Let $w \in W_\ext$, and write $w = t_\lambda w_\fin$ with $\lambda \in \bY$ and $w_\fin \in W_\fin$. Then we set $\bY_{+,w} = \{\mu \in \bY_+ \mid \lambda+\mu \in \bY_+\}$, and for any $\mu \in \bY_{+,w}$ we set
\[
\WFl_w^{\vee,\mu} = \NFl_{t_{-\mu}} \star^{\Iw} \NFl_{t_{\mu} w}.
\]
For any $\mu \in \bY_{+,w}$ and $\eta \in \bY_+$ we have $\mu+\eta \in \bY_{+,w}$, and there exists a canonical morphism
\begin{equation}
\label{eqn:morph-dual-Wak}
\WFl_w^{\vee,\mu} \to \WFl_w^{\vee,\mu+\eta}
\end{equation}
which is defined as follows. Recall from~\eqref{eqn:Hom-D-N-Fl} that there exists a canonical morphism $\DFl_{t_\eta} \to \NFl_{t_\eta}$. Convolving on the left with $\NFl_{t_{-\eta}}$ and using Lemma~\ref{lem:standards-costandards}\eqref{it:standards-costandards-2} we deduce a canonical morphism
\[
\delta_{\Fl} \to \NFl_{t_{-\eta}} \star^\Iw \NFl_{t_\eta},
\]
and then a canonical morphism
\[
\WFl_w^{\vee,\mu} = \NFl_{t_{-\mu}} \star^{\Iw} \delta_{\Fl} \star^{\Iw} \NFl_{t_{\mu} w} \to \NFl_{t_{-\mu}} \star^{\Iw} \NFl_{t_{-\eta}} \star^\Iw \NFl_{t_\eta} \star^{\Iw} \NFl_{t_{\mu} w}.
\]
Here we have $\ell(t_{-\mu-\nu}) = \ell(t_{-\mu}) + \ell(t_{-\nu})$ and $\ell(t_{\eta +\mu} w) = \ell(t_\eta) + \ell(t_{\mu} w)$ (see~\eqref{eqn:formulas-length}), hence using Lemma~\ref{lem:standards-costandards}\eqref{it:standards-costandards-1} we have
\[
\NFl_{t_{-\mu}} \star^{\Iw} \NFl_{t_{-\eta}} \star^\Iw \NFl_{t_\eta} \star^{\Iw} \NFl_{t_{\mu} w} \cong \NFl_{t_{-\mu-\eta}} \star^\Iw \NFl_{t_{\mu+\eta} w},
\]
which provides the desired morphism~\eqref{eqn:morph-dual-Wak}.

Let us now denote by $\unlhd$ the preorder on $\bY$ defined by
\[
 \lambda \unlhd \mu \quad \text{iff $\mu-\lambda \in \bY_+$.}
\]
It is easy to see that the construction explained above defines a functor from (the ordinary category associated with) the preordered set $(\bY_{+,w}, \unlhd)$ to the homotopy category of the $\infty$-category $\Shv(\Iw \backslash \Fl)$.

\begin{lem}
\label{lem:Wak-dual-functor}
For any $w \in W_\ext$, there exists a functor (in the sense of $\infty$-categories)
\[
(\bY_{+,w}, \unlhd) \to \Shv(\Iw \backslash \Fl)
\]
which ``lifts'' the construction of objects and morphisms above.
\end{lem}

\begin{proof}
The justification is similar to that given in~\cite[\S 2.7]{gaitsgory}: the claim follows from the fact that the mapping spaces $\mathrm{Maps}(\WFl_w^{\vee,\mu}, \WFl_w^{\vee,\mu+\eta})$ are discrete, because they identify with the spaces
\[
\mathrm{Maps}(\delta_{\Fl}, \NFl_{t_{-\eta}} \star^\Iw \NFl_{t_\eta}) \cong \mathrm{Maps}(\DFl_{t_{\eta}}, \NFl_{t_\eta}),
\]
which are discrete by~\eqref{eqn:Hom-D-N-Fl}, so that the functor defined at the level of homotopy categories automatically lifts to the $\infty$-categorical level.
\end{proof}

\begin{rmk}
We will see in Lemma~\ref{lem:dual-Wak-heart} below that the objects $\WFl_w^{\vee,\mu}$ belong to the heart of a t-structure on $\Shv(\Iw \backslash \Fl)$. This will provide another justification for Lemma~\ref{lem:Wak-dual-functor}.
\end{rmk}

With Lemma~\ref{lem:Wak-dual-functor} at hand, using the notation above one can set
\[
\WFl^\vee_w = \colim_{\mu \in \bY_{+,w}} \WFl_{w}^{\vee,\mu}.
\]

\begin{prop}
\label{prop:Hom-Wak-dual}
For any $w,y \in W_\ext$ and $n \in \Z$ we have canonical isomorphisms
\[
\Hom_{\Shv(\Iwu \backslash \Fl)}(\oblv^{\Iw | \Iwu}(\WFl_w), \oblv^{\Iw | \Iwu}(\WFl^\vee_y) [n]) \cong \begin{cases}
\bk & \text{if $w=y$ and $n = 0$;} \\
0 & \text{otherwise.}
\end{cases}
\]
\end{prop}

\begin{proof}
The Wakimoto objects $\oblv^{\Iw | \Iwu}(\WFl_w)$ are compact. As a consequence, to prove the lemma it suffices to prove that if $\mu \in \bY_{+,w}$ is sufficiently large (with respect to the order $\unlhd$) we have
\[
\Hom_{\Shv(\Iwu \backslash \Fl)}(\oblv^{\Iw | \Iwu}(\WFl_w), \oblv^{\Iw | \Iwu}(\WFl^{\vee,\mu}_{y}) [n]) \cong 
\begin{cases}
\bk & \text{if $w=y$ and $n = 0$;} \\
0 & \text{otherwise,}
\end{cases}
\]
However, 
if $\mu \in \bY_{+,y} \cap \bY_{+,w}$ we have $\WFl_w \cong \NFl_{t_{-\mu}} \star^{\Iw} \Delta_{t_{\mu} w}$.
It is a standard fact that for any $\nu \in \bY$ there exists an auto-equivalence of $\Shv(\Iwu \backslash \Fl)$ which makes the following diagram commutative:
\[
\begin{tikzcd}[column sep=large]
\Shv(\Iw \backslash \Fl) \ar[d, "\oblv^{\Iw | \Iwu}"'] \ar[r, "\NFl_{t_\nu} \star^{\Iw} (-)"] & \Shv(\Iw \backslash \Fl) \ar[d, "\oblv^{\Iw | \Iwu}"] \\
\Shv(\Iwu \backslash \Fl) \ar[r] & \Shv(\Iwu \backslash \Fl).
\end{tikzcd}
\]
(This equivalence can e.g.~be constructed using convolution with a cofree standard object as considered in~\cite[\S 5.3]{cd} or~\cite{dlyz} and an appropriate analogue of Lemma~\ref{lem:standards-costandards}\eqref{it:standards-costandards-2}. Details will be discussed elsewhere.)
Using this,
the claims follow from~\eqref{eqn:Hom-D-N-Fl-oblv}.
\end{proof}

\begin{rmk}
Using later results of this paper
one can show the analogue of Proposition~\ref{prop:Hom-Wak-dual} in the $\Iw$-equivariant setting, i.e.~that we have
\[
\Hom_{\Shv(\Iw \backslash \Fl)}(\WFl_w, \WFl^\vee_y [n]) \cong 
\begin{cases}
\mathsf{H}^n_T(\pt; \bk) & \text{if $w=y$ and $n \in 2\Z_{\geq 0}$;} \\
0 & \text{otherwise.}
\end{cases}
\]
(More precisely, the formula reduces to~\eqref{eqn:Hom-si-D-N-Fl} using Proposition~\ref{prop:ras-waki} and Corollary~\ref{cor:ras-waki-dual}.) The proof above does not immediately apply because Wakimoto objects are not compact in this setting; see Remark~\ref{rmk:compact-objects}\eqref{it:rmk-compact-objects-2} below for more details.
\end{rmk}

\subsection{Wakimoto sheaves on \texorpdfstring{$\Gr$}{Gr}}
\label{ss:Wak-Gr}

For $\lambda \in \bY$ we also define 
\[
\WGr_\lambda = p_*\WFl_{t_\lambda}.
\]
By Lemma~\ref{lem:properties-Waki} and~\eqref{eqn:push-DN-dom}--\eqref{eqn:push-DN-antidom}, we have
\begin{equation}
\label{eqn:bWaki-dom-antidom}
\WGr_\lambda
\cong
\begin{cases}
\DGr_\lambda & \text{if $\lambda \in  \bY_+$,} \\
\NGr_\lambda[\ell(w_\circ)] & \text{if $\lambda \in -\bY_{++}$.}
\end{cases}
\end{equation}
More generally, for any $\lambda \in \bY$ and $w_\fin \in W_\fin$ we have
\begin{equation}
\label{eqn:waki-push}
p_*\WFl_{t_\lambda w_\fin} \cong \WFl_{t_\lambda} \star^{\Iw} p_*\DFl_{w_\fin} \cong \WFl_{t_\lambda} \star^{\Iw} \delta_{\Gr}[-\ell(w_\fin)] \cong \WGr_\lambda[-\ell(w_\fin)].
\end{equation}
For any $\lambda \in \bY$, the complex $\WGr_\lambda$ is supported on $\overline{\Gr_\lambda}$. Finally, it follows from~\eqref{eqn:convolution-Waki-coweights}
that for $\lambda,\mu \in \bY$ we have
\begin{equation}
\label{eqn:Waki-convolution-Gr}
\WFl_{t_\lambda} \star^{\Iw} \WGr_\mu \cong \WGr_{\lambda+\mu}.
\end{equation}

As shown already by~\eqref{eqn:bWaki-dom-antidom}, the objects $\WGr_\lambda$ are not perverse in general. The following lemma provides bounds for the possible degrees in which they can have nonzero perverse cohomology.

\begin{lem}
\label{lem:bwaki-perv}
For any $\lambda \in \bY$,
we have $\pH^n(\WGr_\lambda) = 0$ unless $-\ell(w_\circ) \le n \le 0$.
\end{lem}

\begin{proof}
Because the functor $p^\dag$ is t-exact and (fully) faithful on perverse sheaves (see~\S\ref{ss:stand-Gr}), it is enough to show that $\pH^n(p^\dag \WGr_\lambda) = 0$ unless $-\ell(w_\circ) \le n \le 0$.  Choose $\mu \in \bY_+$ such that $\lambda+\mu \in \bY_+$; then by~\eqref{eqn:push-DN-dom} and~\eqref{eqn:formula-Wakiw-1} we have
\[
\WGr_\lambda \cong 
\NFl_{t_{-\mu}} \star^{\Iw} \DGr_{\lambda+\mu},
\]
hence
\[
p^\dag \WGr_\lambda \cong \NFl_{t_{-\mu}} \star^{\Iw} p^*\DGr_{\lambda+\mu}[\ell(w_\circ)].
\]
Now we have a stratification
\[
 p^{-1}(\Gr_{\lambda+\mu}) = \bigsqcup_{w_\fin \in W_\fin} \Fl_{t_{\lambda+\mu} w_\fin}
\]
where
\[
\dim(\Fl_{t_{\lambda+\mu} w_\fin}) = \ell(t_{\lambda+\mu} w_\fin) = \ell(t_{\lambda+\mu}) + \ell(w_\fin)
\]
for any $w_\fin \in W_\fin$, see~\eqref{eqn:formulas-length}.
Applying the base change theorem and
taking successively fiber sequences as in~\eqref{eqn:fiber-seq-open-closed},
we deduce that
$p^*\DGr_{\lambda+\mu}$ is an iterated extension of objects of the form $\DFl_{t_{\lambda+\mu} w_\fin}[-\ell(w_\fin)]$ where $w_\fin$ runs over $W_\fin$.  It follows that $p^\dag\WGr_\lambda$ is an iterated extension of objects of the form
\[
\NFl_{t_{-\mu}} \star^{\Iw} \DFl_{t_{\lambda+\mu} w_\fin}[\ell(w_\circ)- \ell(w_\fin)].
\]
Here each complex $\NFl_{t_{-\mu}} \star^{\Iw} \DFl_{t_{\lambda+\mu} w_\fin}$ is perverse by Lemma~\ref{lem:standards-costandards}\eqref{it:standards-costandards-3}, hence we have
$\pH^n(p^\dag\WGr_\lambda) = 0$ unless $n = \ell(w_\fin) - \ell(w_\circ)$ for some $w_\fin \in W_\fin$, which implies the desired claim.
\end{proof}

One can also define dual Wakimoto objects on $\Gr$ as follows: for any $\lambda \in \bY$ we set
\[
\WGr^\vee_\lambda = p_* \WFl^\vee_{t_\lambda}.
\]
In fact, as in~\eqref{eqn:waki-push}, one can easily see that for any $w_\fin \in W_\fin$ and $\lambda \in \bY$ we have
\[
p_* \WFl_{t_\lambda w_\fin}^\vee \cong \WGr^\vee_\lambda[\ell(w_\fin)].
\]
As for their counterparts on $\Fl$, these objects are homologically dual to the objects $(\WGr_\lambda : \lambda \in \bY)$ in the sense of the following statement.

\begin{prop}
For any $\lambda,\mu \in \bY$ and $n \in \Z$ we have canonical isomorphisms
\[
\Hom_{\Shv(\Iwu \backslash \Gr)}(\oblv^{\Iw | \Iwu}(\WGr_\lambda), \oblv^{\Iw | \Iwu}(\WGr^\vee_\mu) [n]) \cong \begin{cases}
\bk & \text{if $\lambda=\mu$ and $n = 0$;} \\
0 & \text{otherwise.}
\end{cases}
\]
\end{prop}

\begin{proof}
The claim follows from Proposition~\ref{prop:Hom-Wak-dual} using adjunction and the fact that $p^* \WGr_\lambda$ is an iterated extension of objects $\WFl_{t_\lambda w_\fin} [-\ell(w_\fin)]$, which was observed in the course of the proof of Lemma~\ref{lem:bwaki-perv}.
\end{proof}

In the next two lemmas we prove some further properties of Wakimoto objects, to be used in the next sections.

\begin{lem}
\label{lem:bwaki-hom}
For any $\mu, \lambda \in \bY$ and any $n \ge 1$, we have 
\[
\Hom_{\Shv(\Iw \backslash \Gr)}(\WGr_\mu[n], \WGr_\lambda) = 0.
\]
\end{lem}

\begin{proof}
For any $\nu \in \bY$, the functor $\WFl_{t_\nu} \star^{\Iw} ({-})$ is an autoequivalence, which induces an isomorphism
\[
\Hom(\WGr_\mu[n],\WGr_\lambda) \cong \Hom(\WGr_{\nu+\mu}[n],\WGr_{\nu+\lambda})
\]
for any $\lambda,\mu \in \bY$ and $n \in \Z$.
Thus, to prove the lemma 
it suffices to treat the special case where $\mu$ and $\lambda$ are both dominant. 
In this case we have
\[
\WGr_\mu \cong \DGr_\mu,
\qquad
\WGr_\lambda \cong \DGr_\lambda,
\]
see~\eqref{eqn:bWaki-dom-antidom}; in particular these objects are perverse, which implies the desired vanishing.
\end{proof}

\begin{rmk}
Another statement in the spirit of Lemma~\ref{lem:bwaki-hom} is this: for any $\mu, \lambda \in \bY$ with $\mu \ne \lambda$, we have
\[
\Hom_{\Shv(\Iw \backslash \Gr)}(\WGr_\mu, \WGr_\lambda) = 0.
\]
This claim can be deduced from~\eqref{eqn:Hom-Deltasi-Gr} and Proposition~\ref{prop:ras-waki} below.
We do not know if it has an ``elementary'' proof that does not involve passing to the semiinfinite world.
\end{rmk}


\begin{lem}
\label{lem:waki-costd-hom}
Let $\mu \in \bY$ and $\lambda \in \bY_+$.  We have
\[
\Hom_{\Shv(\Iw \backslash \Gr)}(\WGr_\mu, \oblv^{\Loop^+ G | \Iw}(\cI_*(\lambda))) \cong
\begin{cases}
\bk & \text{if $\mu = \lambda$,} \\
0 & \text{if $\ell(w_\mu) < \langle \lambda, 2\rho \rangle$.}
\end{cases}
\]
\end{lem}

\begin{proof}
For simplicity we omit the functor $\oblv^{\Loop^+ G | \Iw}$ in this proof.

In case $\mu = \lambda$, we have $\WGr_\lambda = \DGr_\lambda$ (see~\eqref{eqn:bWaki-dom-antidom}); therefore, by adjunction the statement amounts to the assertion that
\[
(i_\lambda)^!\cI_*(\lambda) \cong \underline{\bk}_{\Gr_\lambda}[\dim (\Gr_\lambda)],
\]
which is clear from the definition of $\cI_*(\lambda)$ since $\Gr_\lambda$ is open in $\Gr^\lambda$.

Suppose henceforth that $\ell(w_\mu) < \langle \lambda, 2\rho \rangle$.
Since $\WGr_\mu$ is concentrated in nonpositive perverse degrees (see Lemma~\ref{lem:bwaki-perv}), we have
\begin{equation}
\label{eqn:waki-costd-hom}
\Hom(\WGr_\mu, \cI_*(\lambda)) \cong \Hom(\pH^0(\WGr_\mu), \cI_*(\lambda)).
\end{equation}
The right-hand expression involves objects in the abelian category $\Shv(\Iw \backslash \Gr)^\heartsuit$. In this category, $\cI_*(\lambda)$ has finite length and a unique simple subobject, namely $\ICGr_\lambda$. On the other hand $\pH^0(\WGr_\mu)$ is supported on $\overline{\Gr_\mu}$, a scheme of dimension $\ell(w_\mu) < \dim(\Gr^\lambda)$,
 and thus cannot have $\ICGr_\lambda$ as a composition factor.  We conclude that both sides of~\eqref{eqn:waki-costd-hom} vanish.
\end{proof}

\section{Standard and costandard semiinfinite sheaves}
\label{sec:st-cost-semiinf}

\subsection{Definitions}
\label{ss:st-cost-semiinfinite}
For $w \in W_\ext$ and $\lambda \in \bY$ we set
\begin{align*}
\DFl^\si_w &= (\bi_{w})_! \omega_{\fS_w}[-\psdim (\fS_w)], &
  \DGr^\si_\lambda &= (\bi_{\lambda})_! \omega_{\rS_\lambda}[-\psdim (\rS_\lambda)], \\
\NFl^\si_w &= (\bi_{w})_* \omega_{\fS_w}[-\psdim (\fS_w)], &
  \NGr^\si_\lambda &= (\bi_{\lambda})_* \omega_{\rS_\lambda}[-\psdim (\rS_\lambda)].
\end{align*}
%
As in~\eqref{eqn:Hom-D-N-Fl}--\eqref{eqn:Hom-D-N-Fl-oblv}, for $w,y \in W_\ext$ and $n \in \Z$, in view of~\eqref{eqn:si-sheaves-orbits} and~\eqref{eqn:si-sheaves-orbits-2} we have
\begin{equation}
\label{eqn:Hom-si-D-N-Fl}
\Hom_{\Shv(\siIw \backslash \Fl)}(\DFl^\si_w, \NFl^\si_y [n]) \cong \begin{cases}
\mathsf{H}^n_T(\pt; \bk) & \text{if $w=y$ and $n \in 2\Z_{\geq 0}$;} \\
0 & \text{otherwise}
\end{cases}
\end{equation}
and 
\begin{multline}
\label{eqn:Hom-si-D-N-Fl-oblv}
\Hom_{\Shv(\siIwu \backslash \Fl)}(\oblv^{\siIw | \siIwu}(\DFl^\si_w), \oblv^{\siIw | \siIwu}(\NFl^\si_y) [n]) \cong \\
\begin{cases}
\bk & \text{if $w=y$ and $n=0$;} \\
0 & \text{otherwise,}
\end{cases}
\end{multline}
and as in~\eqref{eqn:Hom-D-N-Gr}--\eqref{eqn:Hom-D-N-Gr-oblv},
for $\lambda, \mu \in \bY$ and $n \in \Z$ we have
\begin{equation}
\label{eqn:Hom-si-D-N-Gr}
\Hom_{\Shv(\siIw \backslash \Gr)}(\DGr^\si_\lambda, \NGr^\si_\mu [n]) \cong \begin{cases}
\mathsf{H}^n_T(\pt; \bk) & \text{if $\lambda=\mu$ and $n \in 2\Z_{\geq 0}$;} \\
0 & \text{otherwise}
\end{cases}
\end{equation}
and
\begin{multline}
\label{eqn:Hom-si-D-N-Gr-oblv}
\Hom_{\Shv(\siIwu \backslash \Gr)}(\oblv^{\siIw | \siIwu}(\DGr^\si_\lambda), \oblv^{\siIw | \siIwu}(\NGr^\si_\mu) [n]) \cong \\
\begin{cases}
\bk & \text{if $\lambda=\mu$ and $n=0$;} \\
0 & \text{otherwise.}
\end{cases}
\end{multline}

In the $\infty$-category $\Shv(\siIw \backslash \tFl)$, we also have
\[
\Hom_{\Shv(\siIw \backslash \tFl)}(q^!(\DFl^\si_w), q^!(\NFl^\si_y) [n]) \cong
\begin{cases}
\bk & \text{if $w=y$ and $n=0$;} \\
0 & \text{otherwise.}
\end{cases}
\]

Finally, using the notation~\eqref{eqn:shift-equiv}, it is clear that for any $w \in W_\ext$ and $\lambda \in \bY$ we have
\begin{equation}
\label{eqn:twist-siDN-Fl}
 \DFl^\si_w \langle \lambda \rangle \cong \DFl^\si_{t_\lambda w}, \qquad
 \NFl^\si_w \langle \lambda \rangle \cong \NFl^\si_{t_\lambda w},
\end{equation}
and that for any $\lambda,\mu \in \bY$ we have
\begin{equation}
\label{eqn:twist-siDN-Gr}
 \DGr^\si_\mu \langle \lambda \rangle \cong \DGr^\si_{\lambda + \mu}, \qquad
 \NGr^\si_\mu \langle \lambda \rangle \cong \NGr^\si_{\lambda + \mu}.
\end{equation}


As in~\cite[Lemma~1.4.7]{gaitsgory} one checks that the objects $(\DFl^\si_w : w \in W_\ext)$ generate the $\infty$-category $\Shv(\siIw \backslash \Fl)$ in the sense of~\cite[Chap.~I, \S 5.4.1]{gr}. Similar claims hold for the objects $(\DGr^\si_\lambda : \lambda \in \bY)$ in $\Shv(\siIw \backslash \Gr)$, or for the images of these collections of objects in $\Shv(\siIwu \backslash \Fl)$, $\Shv(\siIw \backslash \tFl)$ or $\Shv(\siIwu \backslash \Gr)$. 


\begin{lem}
\label{lem:compact-objects}
  Let $\cF \in \Shv(\siIwu \backslash \Fl)$. Then $\cF$ is compact if and only if it satisfies the following conditions:
  \begin{enumerate}
   \item 
   \label{it:compact-objects-1}
   $\cF$ is supported on a finite union of closures of orbits $\fS_w$;
   \item
   \label{it:compact-objects-2}
   the set $\{w \in W_\ext \mid (\bi_w)^* \cF \neq 0\}$ is finite;
   \item
   \label{it:compact-objects-3}
   for any $w \in W_\ext$ the complex of $\bk$-vector spaces corresponding to $(\bi_w)^* \cF$ under the equivalence in~\eqref{eqn:si-sheaves-orbits} has finite dimensional total cohomology.
  \end{enumerate}
  
Similar statements hold in $\Shv(\siIwu \backslash \Gr)$ and $\Shv(\siIw \backslash \tFl)$.
\end{lem}

\begin{proof}
The lemma can be deduced using Raskin's equivalences (see~\S\ref{ss:raskin-equiv}) from the known description of compact objects in $\Shv(\Iwu \backslash \Fl)$, $\Shv(\Iwu \backslash \Gr)$, $\Shv(\Iwu \backslash \tFl)$. Alternatively one can proceed as follows (in
the case of $\Shv(\siIwu \backslash \Fl)$, to fix notation). The objects $\oblv^{\siIw | \siIwu}(\DFl^\si_w)$ are compact 
because $\bk \in \Vect_\bk$ is compact. Since these objects generate $\Shv(\siIwu \backslash \Fl)$, by general results on compactly generated presentable stable $\infty$-categories (see~\cite[Chap.~I, Lemma~7.2.4]{gr}), the subcategory of compact objects in $\Shv(\siIwu \backslash \Fl)$ is the (small) idempotent-complete stable $\infty$-subcategory generated by the objects $(\oblv^{\siIw | \siIwu}(\DFl^\si_w) : w \in W_\ext)$. It is clear that any object in this subcategory satisfies the properties of the lemma. The converse inclusion follows from Lemma~\ref{lem:*rest-0} using the fiber sequences in~\eqref{eqn:fiber-seq-open-closed}.
%
%
\end{proof}

\begin{rmk}
\phantomsection
\label{rmk:compact-objects}
\begin{enumerate}
\item
Lemma~\ref{lem:compact-objects} is a version in our setting of~\cite[Lemma~1.4.10]{gaitsgory}. The latter statement omits condition~\eqref{it:compact-objects-1}, but it is really needed. Namely, by Remark~\ref{rmk:dualizing-infty}\eqref{it:dualizing-vanishing-stalks} below the dualizing complex $\omega_{\Fl}$ satisfies~\eqref{it:compact-objects-2} and~\eqref{it:compact-objects-3}, but it does not satisfy~\eqref{it:compact-objects-1} (hence is not compact).
\item
\label{it:rmk-compact-objects-2}
In Lemma~\ref{lem:compact-objects} we have only considered the $\Iwu$-equivariant categories, and not the $\Iw$-equivariant variants. Using Raskin's equivalences, it can be deduced from~\cite[Proposition~4.41]{zhu} (and an analogue for $\Gr$) that
$\Shv(\siIw \backslash \Gr)$ and $\Shv(\siIw \backslash \Fl)$ are compactly generated (by objects obtained from compact objects in $\Shv(\siIwu \backslash \Gr)$ and $\Shv(\siIwu \backslash \Fl)$ by pushforward), but the compact objects are not those that we consider interesting; in particular, standard objects are not compact (see Remark~\ref{rmk:Shv-BT}).
Note that the standard objects in the $\Iw$-equivariant categories still satisfy a weaker ``compactness'' condition, as explained in Remark~\ref{rmk:perv-t-str}\eqref{it:standards-almost-compact} below.
\end{enumerate}
\end{rmk}

\subsection{Pushforward and pullback of standard and costandard objects}

The following statement is analogous to Lemma~\ref{lem:standards-costandards}\eqref{it:standards-costandards-1}, and its proof is similar.

\begin{lem}
\label{lem:convolution-DN-si}
For any $\lambda \in \bY$ and $w_\fin \in W_\fin$ we have canonical isomorphisms
\[
\DFl^\si_{t_\lambda} \star^{\Iw} \DFl_{w_\fin} \cong \DFl^\si_{t_\lambda w_\fin}, \quad
\NFl^\si_{t_\lambda} \star^{\Iw} \NFl_{w_\fin} \cong \NFl^\si_{t_\lambda w_\fin}
\]
\end{lem}

We now study the behaviour of standard and costandard objects under the functor $p^\dag$
from~\S\ref{ss:si-sheaves-Fl}.

\begin{lem}
\phantomsection
\label{lem:si-pushpull}
\begin{enumerate}
\item 
\label{it:si-push}
Let $w \in W_\ext$, and write $w= t_\lambda w_\fin$ where $\lambda \in \bY$ and $w_\fin \in W_\fin$.
We have
\[
p_*\DFl^\si_w \cong \DGr^\si_\lambda[-\ell(w_\fin)], \quad p_*\NFl^\si_w \cong \NGr^\si_\lambda[\ell(w_\fin)].
\]
\item 
\label{it:si-pullnabla}
For any $\lambda \in \bY$,
the object $p^\dag\NGr^\si_\lambda$ is an iterated extension 
of the objects $\NFl^\si_{t_\lambda x_\fin}[\ell(x_\fin)-\ell(w_\circ)]$ for $x_\fin \in W_\fin$, each appearing with multiplicity $1$.
\item 
\label{it:si-pulldelta}
For any $\lambda \in \bY$,
the object $p^\dag\DGr^\si_\lambda$ is an iterated extension 
of the objects $\DFl^\si_{t_\lambda x_\fin}[\ell(w_\circ) - \ell(x_\fin)]$ for $x_\fin \in W_\fin$, each appearing with multiplicity $1$.
\end{enumerate}
\end{lem}

\begin{proof}
\eqref{it:si-push}
Since $p$ is (representable and) proper we have $p_*=p_!$; hence
\[
p_*\DFl^\si_w = p_! (\bi_w)_! \omega_{\fS_w}[-\psdim(\fS_w)] = (\bi_\lambda)_! (p_{|\fS_w})_! \omega_{\fS_w}[-\psdim(\fS_w)].
\]
Now using~\eqref{eqn:reldim-semiinf} we see that
\[
(p_{|\fS_w})_! \omega_{\fS_w} \cong \omega_{\rS_\lambda}[-\ell(w_\fin)],
\]
which proves the first isomorphism. The second one can be obtained similarly.

\eqref{it:si-pullnabla}
Using~\eqref{it:si-push} we see that
\[
p^\dag\NGr^\si_\lambda = p^* p_* \NFl^\si_{t_\lambda}[\ell(w_\circ)] = \NFl^\si_{t_\lambda} \star^{\Iw} \IC_{w_\circ}.
\]
Now, using the fiber sequences as in~\eqref{eqn:fiber-seq-open-closed}
one sees that for any finitely stratified scheme $X$ of finite type, the dualizing complex $\omega_X$ is an iterated extension of $*$-pushforwards of dualizing sheaves of strata, each appearing once. Hence, in our setting, $\IC_{w_\circ} = \omega_{\overline{\Fl_{w_\circ}}}[-\ell(w_\circ)]$ is an iterated extension of objects $\NFl_{x_\fin}[\ell(x_\fin)-\ell(w_\circ)]$ for $x_\fin \in W_\fin$, each appearing once. The desired claim follows, in view of Lemma~\ref{lem:convolution-DN-si}.

\eqref{it:si-pulldelta}
The proof is similar to that of~\eqref{it:si-pullnabla}.
\end{proof}

\subsection{Raskin's equivalence and (co)standard objects}


Recall the equivalence
\[
\Ras: \Shv(\siIw \backslash \Fl) \simto \Shv(\Iw \backslash \Fl)
\]
and its analogues for $\Gr$ and $\tFl$, see~\S\ref{ss:raskin-equiv}.

\begin{lem}
\label{lem:Ras-twist}
Let $X$ be either $\tFl$, $\Fl$ or $\Gr$.
 For any $\lambda \in \bY$ and $\cF$ in $\Shv(\siIw \backslash X)$, we have a canonical isomorphism
 \[
  \Ras(\cF \langle \lambda \rangle) \cong \WFl_{t_\lambda} \star^{\Iw} \Ras(\cF).
 \]
\end{lem}

\begin{proof}
 It suffices to prove the formula when $\lambda \in \bY_+$; in this case we have $\WFl_{t_\lambda}=\Delta_{t_\lambda}$. Now, using the description of the inverse equivalence $\Ras^{-1}$ recalled in~\S\ref{ss:raskin-equiv} and Lemma~\ref{lem:orbits-Iw-Fl} (see also Remark~\ref{rmk:orbits-siIw}),
 one sees that for any $\cG$ in $\Shv(\Iw \backslash X)$ we have
 \[
  \Ras^{-1}(\Delta_{t_\lambda} \star^{\Iw} \cG) \cong z^\lambda \cdot \Ras^{-1}(\cG) [-\langle \lambda,2\rho \rangle],
 \]
where (as in~\S\ref{ss:raskin-equiv-proof}) we write $z^\lambda \cdot (-)$ for the pushforward functor associated with the map given by left multiplication by $z^{\lambda}$. The desired isomorphism follows.
\end{proof}

\begin{prop}
\label{prop:ras-waki}
For any $w \in W_\ext$, resp.~$w \in W_\ext$, resp.~$\lambda \in \bY$, we have a canonical isomorphism
\[
\Ras(q^! \DFl^\si_w) \cong q^! \WFl_w, \quad \text{resp.} \quad
\Ras(\DFl^\si_w) \cong \WFl_w, \quad \text{resp.} \quad
\Ras(\DGr^\si_\lambda) \cong \WGr_\lambda.
\]
\end{prop}


\begin{proof}
Using Lemma~\ref{lem:Ras-twist} and comparing Lemma~\ref{lem:properties-Waki}\eqref{it:Waki-1} with~\eqref{eqn:twist-siDN-Fl} (resp.~\eqref{eqn:Waki-convolution-Gr}
with~\eqref{eqn:twist-siDN-Gr}), the proof is reduced to the case $w \in W_\fin$ (resp.~$\lambda=0$). This case is clear from the definitions and the description of the functor $\Ras^{-1}$ in~\S\ref{ss:raskin-equiv}.
\end{proof}

One can also describe the images of costandard objects under Raskin's equivalence using the dual Wakimoto objects introduced in~\S\ref{ss:Waki-Fl-dual}--\ref{ss:Wak-Gr}, as follows. At this point we will only prove this in the $\Iwu$-equivariant setting; we will later ``lift'' this statement to the $\Iw$-equivariant setting using other tools to be introduced below; see Corollary~\ref{cor:ras-waki-dual}.

\begin{cor}
\label{cor:ras-waki-dual-oblv}
For any $w \in W_\ext$, resp.~$w \in W_\ext$, resp.~$\lambda \in \bY$, we have a canonical isomorphism
\begin{multline*}
\Rasu(q^! \oblv^{\siIw | \siIwu} (\NFl^\si_w)) \cong q^! \oblv^{\Iw | \Iwu}(\WFl^\vee_w), \quad \text{resp.} \\
\Rasu( \oblv^{\siIw | \siIwu}(\NFl^\si_w)) \cong \oblv^{\Iw | \Iwu}(\WFl^\vee_w), \quad \text{resp.} \\
\Rasu( \oblv^{\siIw | \siIwu}(\NGr^\si_\lambda)) \cong \oblv^{\Iw | \Iwu}(\WGr^\vee_\lambda).
\end{multline*}
\end{cor}

\begin{proof}
It is enough to prove the second isomorphism. And since the collection $(\oblv^{\Iw | \Iwu}(\DFl^\si_w) : w \in W_\ext)$ generates $\Shv(\Iwu \backslash \Fl)$ (see~\S\ref{ss:st-cost-semiinfinite}), the desired isomorphisms follow from Proposition~\ref{prop:Hom-Wak-dual}, \eqref{eqn:Hom-si-D-N-Fl-oblv} and Proposition~\ref{prop:ras-waki}.
\end{proof}

As another consequence of Proposition~\ref{prop:ras-waki},
one can describe the morphisms (of degree $0$) between the standard objects on $\Fl$ as follows. (This description involves the semiinfinite order on $W_\ext$, whose definition is recalled in~\S\ref{ss:Weyl}.)

\begin{cor}
For $w,y \in W_\ext$ we have
\[
\dim \Hom_{\Shv(\siIw \backslash \Fl)}(\DFl^\si_w, \DFl^\si_y) = \begin{cases}
1 & \text{if $w \leq^\si y$;} \\
0 & \text{otherwise.}
\end{cases}
\]
\end{cor}

\begin{proof}
This follows from Lemma~\ref{lem:properties-Waki}\eqref{it:Waki-4} and Proposition~\ref{prop:ras-waki}.
\end{proof}

\section{The perverse t-structure for semiinfinite sheaves}
\label{sec:perv-t-str-si}

\subsection{Definition}
\label{ss:perv-semiinf-def}

We will now define t-structures on the categories
\[
\Shv(\siIw \backslash \tFl), \quad 
\Shv(\siIw \backslash \Fl), \quad 
\Shv(\siIw \backslash \Gr), \quad 
\Shv(\siIwu \backslash \Fl), \quad 
\Shv(\siIwu \backslash \Gr)
\]
which we will call the perverse t-structures. 

We start with the case of $\Fl$.
By~\cite[Proposition~1.4.4.11]{lurie-ha} there exists a unique t-structure
\[
\bigl( \Shv(\siIw \backslash \Fl)^{\leq 0}, \, \Shv(\siIw \backslash \Fl)^{\geq 0} \bigr), \quad \text{resp.} \quad
\bigl( \Shv(\siIwu \backslash \Fl)^{\leq 0}, \, \Shv(\siIwu \backslash \Fl)^{\geq 0} \bigr)
\]
on $\Shv(\siIw \backslash \Fl)$, resp.~on $\Shv(\siIwu \backslash \Fl)$, such that $\Shv(\siIw \backslash \Fl)^{\leq 0}$, resp.~$\Shv(\siIwu \backslash \Fl)^{\leq 0}$, is the smallest full subcategory containing the collection of objects
\[
(\DFl^\si_w : \, w \in W_\ext), \quad \text{resp.} \quad
(\oblv^{\siIw | \siIwu}(\DFl^\si_w) : \, w \in W_\ext),
\]
and which is closed under extensions and small colimits. In view of~\cite[Remark~1.2.1.3]{lurie-ha}, the nonnegative part of this t-structure consists of the objects $\mathcal{F} \in \Shv(\siIw \backslash \Fl)$ which satisfy
\[
\Hom_{\Shv(\siIw \backslash \Fl)}(\DFl^\si_w[n], \mathcal{F})=0 \quad \text{for any $w \in W_\ext$ and $n \in \Z_{>0}$,}
\]
resp.~of the objects $\mathcal{F} \in \Shv(\siIwu \backslash \Fl)$ which satisfy
\[
\Hom_{\Shv(\siIwu \backslash \Fl)}(\oblv^{\siIw | \siIwu}(\DFl^\si_w)[n], \mathcal{F})=0 \quad \text{for any $w \in W_\ext$ and $n \in \Z_{>0}$.}
\]
This condition amounts to requiring that for any $w \in W_\ext$ the complex $(\bi_w)^! \mathcal{F}$ belongs to $\Shv(\siIw \backslash \fS_w)^{\geq 0}$, resp.~to $\Shv(\siIwu \backslash \fS_w)^{\geq 0}$, see~\S\ref{ss:si-orbits}. 

\begin{rmk}
\phantomsection
\label{rmk:perv-t-str}
\begin{enumerate}
\item
\label{it:standards-almost-compact}
As explained in Remark~\ref{rmk:compact-objects}\eqref{it:rmk-compact-objects-2}, the objects $\DFl^\si_w$ are not compact. But they are ``almost compact'' in the sense that for any $w \in W_\ext$ the functor
\[
\Hom_{\Shv(\siIw \backslash \Fl)}(\DFl^\si_w, -) : \Shv(\siIw \backslash \Fl) \to \Vect_\bk
\]
commutes with uniformly bounded below filtered colimits with respect to the perverse t-structure. I.e., given a filtered family of objects $(\cF_i : i \in I)$ such that all $\cF_i$'s belong to $\Shv(\siIw \backslash \Fl)^{\geq N}$ for some $N$, then we have
\[
\Hom_{\Shv(\siIw \backslash \Fl)}(\DFl_w, \colim_i \cF_i) \cong \colim_i \Hom_{\Shv(\siIw \backslash \Fl)}(\DFl_w,\cF_i).
\] 
To prove this, by adjunction we are reduced to the analogous claim on a single stratum.
%
%
For the latter, one can use the t-exact equivalence in~\eqref{eqn:si-sheaves-orbits-2}, and then the t-exact equivalence (deduced from the Barr--Beck theorem) between $\Shv(\mathrm{pt}/T)$ and modules for chains on $T$, viewed as a convolution (animated) algebra. This algebra has as cohomology the exterior algebra of a $\bk$-vector space of dimension the rank of $T$; in particular it is left coherent in the sense of~\cite[Definition~7.2.4.16]{lurie-ha}, and the constant sheaf corresponds to the natural augmentation module. Hence the claim follows from~\cite[Proposition~7.2.4.17]{lurie-ha}.
\item
\label{it:tstr-degenerate}
The fact that the standard objects are generators (see~\S\ref{ss:st-cost-semiinfinite}) implies that
for any nonzero $\cF$ in $\Shv(\siIw \backslash \Fl)$ or $\Shv(\siIwu \backslash \Fl)$, there exists $w \in W_\ext$ such that $(\bi_w)^! \cF \neq 0$. This shows that
\[
\bigcap_{n \in \Z} \Shv(\siIw \backslash \Fl)^{\geq n} = \{0\}.
\]
However, we will show in Section~\ref{sec:dualizing} that the dualizing complex $\omega_{\Fl}$ belongs to
the intersection
\[
\Shv(\siIw \backslash \Fl)^{\leq -\infty} := \bigcap_{n \in \Z} \Shv(\siIw \backslash \Fl)^{\leq n}.
\]
In particular, the perverse t-structure \emph{is} degenerate. (Such a phenomenon, i.e.~the existence of nonzero objects ``concentrated in degree $-\infty$'' for a t-structure, is a usual feature in constructions related to the Geometric Langlands Program, see e.g.~\cite{beraldo}.) 
\item
\label{it:perv-colim}
Using Remark~\ref{rmk:colimits-t-str}, it is easy to see that both parts of the perverse t-structures on $\Shv(\siIw \backslash \Fl)$ and $\Shv(\siIwu \backslash \Fl)$ are stable under filtered colimits. Similar comments apply for the cases of $\Gr$ and $\tFl$ considered below.
\end{enumerate}
\end{rmk}

In the ``usual'' theory of the perverse t-structure (as e.g.~in~\cite{bbdg}), the nonpositive part can be described in terms of $*$-restriction of complexes to strata. This characterization does \emph{not} hold for the perverse t-structure on $\Shv(\siIw \backslash \Fl)$ or $\Shv(\siIwu \backslash \Fl)$.\footnote{The published version of~\cite{gaitsgory} contains an incorrect description of the nonpositive part of the perverse t-structure in terms of $*$-restriction to strata, see~\cite[\S 1.5.2]{gaitsgory}. This mistake is corrected in the arXiv version of this paper mentioned in the references.} It does however hold for complexes supported on a finite union of closures of $\siIw$-orbits, as explained in the following lemma. 

\begin{lem}
\label{lem:lower-0-part-stalks}
Let $J$ be either $\siIw$ or $\siIwu$, and
let $\cF$ be an object in $\Shv(J \backslash \Fl)$. Assume that $\cF$ is supported on a finite union of closures of $\siIw$-orbits. Then $\cF$ belongs to $\Shv(J \backslash \Fl)^{\leq 0}$ iff for any $w \in W_\ext$ the complex $(\bi_w)^* \cF$ belongs to the ``$\leq 0$'' part of the perverse t-structure on $\Shv(J \backslash \fS_w)$ (see~\S\ref{ss:si-orbits}).
\end{lem}

\begin{proof}
The statement can be proved using the isomorphism~\eqref{eqn:colim-opens}, where the $U_n$'s are chosen to be finite unions of $\siIw$-orbits which are open in the support of $\cF$. See also~\cite[Proposition~6.4.2]{bkv} for a statement of this kind in a more general setting.
\end{proof}

\begin{rmk}
\label{rmk:perverse-degrees-costalk-2}
In view of Remark~\ref{rmk:perverse-degrees-costalk}, Lemma~\ref{lem:lower-0-part-stalks} says that, under the stated condition, $\cF$ belongs to $\Shv(J \backslash \Fl)^{\leq 0}$ iff for any $w \in W_\ext$ the complex $(\mathbf{k}_w)^! (\bi_w)^* \cF$ is concentrated in degrees $\leq \psdim(\fS_w)$.
\end{rmk}

The heart of the perverse t-structure will be denoted
\[
\Shv(\siIw \backslash \Fl)^\heartsuit, \quad \text{resp.} \quad \Shv(\siIwu \backslash \Fl)^\heartsuit.
\]
It is clear from this construction that the functor
\[
\oblv^{\siIw | \siIwu} : \Shv(\siIw \backslash \Fl) \to \Shv(\siIwu \backslash \Fl)
\]
is right t-exact, and from the fact that~\eqref{eqn:oblv-Sw} is t-exact for any $w \in W_\ext$ we deduce that it is also left t-exact, hence t-exact. In fact, since this functor is conservative (see~\S\ref{ss:functorialities}), it even ``detects perversity'' in the sense that a complex $\cF$ belongs to $\Shv(\siIw \backslash \Fl)^{\leq 0}$, resp.~$\Shv(\siIw \backslash \Fl)^{\geq 0}$, iff $\oblv^{\siIw | \siIwu}(\cF)$ belongs to $\Shv(\siIwu \backslash \Fl)^{\leq 0}$, resp.~$\Shv(\siIwu \backslash \Fl)^{\geq 0}$.
For any $w \in W_\ext$, by definition the object $\DFl^\si_w$ belongs to $\Shv(\siIw \backslash \Fl)^{\leq 0}$, and
by~\eqref{eqn:Hom-si-D-N-Fl} the object $\NFl^\si_w$ belongs to $\Shv(\siIw \backslash \Fl)^{\geq 0}$.

Very similar considerations lead to the definition of the perverse t-structures for complexes on $\Gr$. Namely, there exists a unique t-structure (called \emph{perverse})
\[
\bigl( \Shv(\siIw \backslash \Gr)^{\leq 0}, \ \Shv(\siIw \backslash \Gr)^{\geq 0} \bigr), \quad \text{resp.} \quad
\bigl( \Shv(\siIwu \backslash \Gr)^{\leq 0}, \ \Shv(\siIwu \backslash \Gr)^{\geq 0} \bigr)
\]
on $\Shv(\siIw \backslash \Gr)$, resp.~$\Shv(\siIwu \backslash \Gr)$, such that $\Shv(\siIw \backslash \Gr)^{\leq 0}$, resp.~$\Shv(\siIwu \backslash \Gr)^{\leq 0}$, is the smallest full subcategory containing the collection of objects
\[
(\DGr^\si_\lambda : \lambda \in \bY), \quad \text{resp.} \quad
(\oblv^{\siIw | \siIwu}(\DGr^\si_\lambda) : \lambda \in \bY),
\]
and which is closed under extensions and small colimits. As above, the nonnegative part of this t-structure consists of the objects $\mathcal{F} \in \Shv(\siIw \backslash \Gr)$ which satisfy
\[
\Hom_{\Shv(\siIw \backslash \Gr)}(\DGr^\si_\lambda[n], \mathcal{F})=0 \quad \text{for any $\lambda \in \bY$ and $n \in \Z_{>0}$,}
\]
resp.~of the objects $\mathcal{F} \in \Shv(\siIwu \backslash \Gr)$ which satisfy
\[
\Hom_{\Shv(\siIwu \backslash \Gr)}(\oblv^{\siIw | \siIwu}(\DGr^\si_\lambda)[n], \mathcal{F})=0 \quad \text{for any $\lambda \in \bY$ and $n \in \Z_{>0}$.}
\]
This condition amounts to requiring that for any $\lambda \in \bY$ the complex $(\bi_\lambda)^! \mathcal{F}$ belongs to the nonnegative part of the perverse t-structure on $\Shv(\siIw \backslash \rS_\lambda)$, resp.~on $\Shv(\siIwu \backslash \rS_\lambda)$, see~\S\ref{ss:si-orbits}. 

The heart of the perverse t-structure will be denoted
\[
\Shv(\siIw \backslash \Gr)^\heartsuit, \quad \text{resp.} \quad \Shv(\siIwu \backslash \Gr)^\heartsuit.
\]
As above the functor
\[
\oblv^{\siIw | \siIwu} : \Shv(\siIw \backslash \Gr) \to \Shv(\siIwu \backslash \Gr)
\]
is t-exact, and it ``detects perversity.'' For any $\lambda \in \bY$, by definition the object $\DGr^\si_\lambda$ belongs to $\Shv(\siIw \backslash \Gr)^{\leq 0}$, and by~\eqref{eqn:Hom-si-D-N-Gr} the object $\NGr^\si_\lambda$ belongs to $\Shv(\siIw \backslash \Gr)^{\geq 0}$.

Finally, in the case of $\tFl$, we will define the perverse t-structure on $\Shv(\siIw \backslash \tFl)$ in such a way that $\Shv(\siIw \backslash \tFl)^{\leq 0}$ is the smallest full subcategory containing the collection of objects
\[
 (q^!(\DFl^{\si}_w) : w \in W_\ext)
\]
and closed under extensions and small colimits. Similar comments as above apply in this case too.



\subsection{First properties}


The following statement is an immediate consequence of the definitions.

\begin{lem}
\phantomsection
\label{lem:heart-zero}
\begin{enumerate}
\item 
\label{it:heart-zero-1}
Let $\cF \in \Shv(\siIw \backslash \Fl)^\heartsuit$.  We have $\cF = 0$ if and only if
\[
\Hom(\DFl^\si_w,\cF) = 0 \quad \text{for all $w \in W_\ext$.}
\]
\item 
\label{it:heart-zero-2}
Let $\cF \in \Shv(\siIwu \backslash \Fl)^\heartsuit$.  We have $\cF = 0$ if and only if
\[
\Hom(\oblv^{\siIw | \siIwu}(\DFl^\si_w),\cF) = 0 \quad \text{for all $w \in W_\ext$.}
\]
\item 
\label{it:heart-zero-3}
Let $\cF \in \Shv(\siIw \backslash \Gr)^\heartsuit$.  We have $\cF = 0$ if and only if
\[
\Hom(\DGr^\si_\lambda,\cF) = 0 \quad \text{for all $\lambda \in \bY$.}
\]
\item 
\label{it:heart-zero-4}
Let $\cF \in \Shv(\siIwu \backslash \Gr)^\heartsuit$.  We have $\cF = 0$ if and only if
\[
\Hom(\oblv^{\siIw | \siIwu}(\DGr^\si_\lambda),\cF) = 0 \quad \text{for all $\lambda \in \bY$.}
\]
\item 
\label{it:heart-zero-5}
Let $\cF \in \Shv(\siIw \backslash \tFl)^\heartsuit$.  We have $\cF = 0$ if and only if
\[
\Hom(q^!(\DFl^\si_w),\cF) = 0 \quad \text{for all $w \in W_\ext$.}
\]
\end{enumerate}
\end{lem}

\begin{proof}
We prove the first case; the other ones are similar.
The ``only if'' direction is obvious.  Conversely, if $\cF$ belongs to $\Shv(\siIw \backslash \Fl)^\heartsuit$ and satisfies $\Hom(\DFl^\si_w,\cF) = 0$ for all $w$, then the definition of the $t$-structure shows that $\cF$ actually lies in $\Shv(\siIw \backslash \Fl)^{\ge 1} \cap \Shv(\siIw \backslash \Fl)^{\heartsuit}$, so it is the zero object.  
\end{proof}

Next we prove a counterpart of a standard property of the forgetful functor from $\Iw$-equivariant to $\Iwu$-equivariant perverse sheaves.

\begin{lem}
\label{lem:oblv-siIw}
 Let $X$ be either $\Fl$ or $\Gr$. The restriction of the t-exact functor
 \[
  \oblv^{\siIw | \siIwu} : \Shv(\siIw \backslash X) \to \Shv(\siIwu \backslash X)
 \]
to the hearts of the perverse t-structures is fully faithful, and its essential image is stable under subquotients.
\end{lem}

\begin{proof}
To fix notation we assume that $X=\Fl$; the other case is similar.
 As explained in~\S\ref{ss:si-sheaves-Fl}, the functor $\oblv^{\siIw | \siIwu}$ admits a left adjoint $\Av_!^{\siIw | \siIwu}$ and a right adjoint $\Av_*^{\siIw | \siIwu}$. Since $\oblv^{\siIw | \siIwu}$ is t-exact, $\Av_!^{\siIw | \siIwu}$ is right t-exact and $\Av_*^{\siIw | \siIwu}$ is left t-exact, and the functors
 \begin{align*}
  \Av_{!,\circ}^{\siIw | \siIwu} := \pH^0 \circ \Av_!^{\siIw | \siIwu} &: \Shv(\siIwu \backslash \Fl)^\heartsuit \to \Shv(\siIw \backslash \Fl)^\heartsuit \\
   \Av_{*,\circ}^{\siIw | \siIwu} := \pH^0 \circ \Av_*^{\siIw | \siIwu} &: \Shv(\siIwu \backslash \Fl)^\heartsuit \to \Shv(\siIw \backslash \Fl)^\heartsuit
 \end{align*}
are left and right adjoint to the restriction
\[
\oblv^{\siIw | \siIwu}_\circ : \Shv(\siIw \backslash \Fl)^\heartsuit \to \Shv(\siIwu \backslash \Fl)^\heartsuit
\]
of $\oblv^{\siIw | \siIwu}$ respectively. 
It is also easy to see that $\Av_{*,\circ}^{\siIw | \siIwu} \circ \oblv_\circ^{\siIw | \siIwu} \cong \id$, which implies that $\oblv^{\siIw | \siIwu}_\circ$ is fully faithful.

To prove the final part of the statement, by~\cite[\S 4.2.6]{bbdg}\footnote{Note that there are typos in the conditions stated there: (b) should say that the adjunction morphism $B \to u^* u_! B$ is surjective, and (b') should say that the adjunction morphism $u^* u_* B \to B$ is injective.} it suffices to prove that for any $\cF$ in $\Shv(\siIwu \backslash \Fl)^\heartsuit$ the adjunction morphism
\[
 \oblv^{\siIw | \siIwu}_\circ \circ \Av_{*,\circ}^{\siIw | \siIwu}(\cF) \to \cF
\]
is injective. This morphism is obtained by taking degree-$0$ perverse cohomology of the adjunction morphism
\[
 \oblv^{\siIw | \siIwu} \circ \Av_{*}^{\siIw | \siIwu}(\cF) \to \cF.
\]
To conclude, it therefore suffices to prove that the third term $\cG$ of the fiber sequence based on the latter morphism belongs to $\Shv(\siIwu \backslash \Fl)^{\geq 0}$. However, the functors $\oblv^{\siIw | \siIwu}$ and $\Av_{*}^{\siIw | \siIwu}$ have counterparts for the categories of sheaves on $\fS_w$ for any $w \in W_\ext$, which we will denote similarly, and we have canonical isomorphisms
\[
 (\bi_w)^! \circ \oblv^{\siIw | \siIwu} \cong \oblv^{\siIw | \siIwu} \circ (\bi_w)^!, \qquad
 (\bi_w)^! \circ \Av_{*}^{\siIw | \siIwu} \cong \Av_{*}^{\siIw | \siIwu} \circ (\bi_w)^!.
\]
The desired claim therefore reduces to the fact that, for any $w \in W_\ext$ and $\cG \in \Shv(\siIwu \backslash \fS_w)^{\geq 0}$, the cone of the adjunction morphism
\[
 \oblv^{\siIw | \siIwu} \circ \Av_{*}^{\siIw | \siIwu}(\cG) \to \cG
\]
belongs to $\Shv(\siIwu \backslash \fS_w)^{\geq 0}$, which is clear.
\end{proof}


The following statement follows from similar considerations.

\begin{lem}
\label{lem:pukllback-exact-q}
 The functor
 \[
  q^! : \Shv(\siIw \backslash \Fl) \to \Shv(\siIw \backslash \tFl)
 \]
 is t-exact. Moreover its restriction
to the hearts of the perverse t-structures is fully faithful, and its essential image is stable under subquotients.
\end{lem}

\begin{rmk}
\label{rmk:functors-detect-perversity}
One can easily check that the functor $q^!$ of Lemma~\ref{lem:pukllback-exact-q} is also conservative. Hence it ``detects perversity'' in the same sense as for the functor $\oblv^{\siIw | \siIwu}$. A similar comment applies to the functor studied in Lemma~\ref{lem:pullback-exact} below.
\end{rmk}




Finally, we study the functor $p^\dag$.

\begin{lem}
\label{lem:pullback-exact}
The functor
\[
p^\dag : \Shv(\siIw \backslash \Gr) \to \Shv(\siIw \backslash \Fl) 
\]
is t-exact. Moreover, its restriction to the heart of the perverse t-structures is fully faithful, and its essential image is stable under subquotients and extensions.
Similar comments apply to the $\Iwu$-equivariant versions of these $\infty$-categories and functor.
\end{lem}

\begin{proof}
We only treat the $\siIw$-equivariant case; the other case is similar.
It follows from Lemma~\ref{lem:si-pushpull}\eqref{it:si-pulldelta} that $p^\dag$ is right t-exact. To prove left t-exactness, we have to prove that if $\mathcal{F} \in \Shv(\siIw \backslash \Gr)$ belongs to $\Shv(\siIw \backslash \Gr)^{\geq 0}$, then for any $w \in W_\ext$ and $n \in \Z_{>0}$ we have
\[
\Hom_{\Shv(\siIw \backslash \Fl)}(\DFl^\si_w [n], p^\dag \mathcal{F})=0.
\]
However, for any $w \in W_\ext$ and $n \in \Z$ we have
\[
\Hom(\DFl^\si_w [n], p^\dag \mathcal{F})=\Hom(\DFl^\si_w [n], p^! \mathcal{F}[-\ell(w_\circ)]) \cong \Hom(p_! \DFl^\si_w [n+\ell(w_\circ)], \mathcal{F}).
\]
Writing $w = t_\lambda w_\fin$ with $\lambda \in \bY$ and $w_\fin \in W_\fin$, using Lemma~\ref{lem:si-pushpull}\eqref{it:si-push} we deduce that
\[
\Hom(\DFl^\si_w [n], p^\dag \mathcal{F}) \cong \Hom(\DGr^\si_\lambda [n+\ell(w_\circ)-\ell(w_\fin)], \mathcal{F}).
\]
Here the right-hand side vanishes if $n>0$ because $\ell(w_\circ)-\ell(w_\fin) \geq 0$, which provides the desired claim.

Next, if $\mathcal{F}, \mathcal{G}$ belong to $\Shv(\siIw \backslash \Gr)$, by adjointness we have
\[
\Hom(p^\dag \mathcal{F}, p^\dag \mathcal{G}) \cong \Hom(\mathcal{F}, p_* p^* \mathcal{G}).
\]
Assume that $\cG$ is in $\Shv(\siIw \backslash \Gr)^\heartsuit$, and consider the fiber sequence
\[
\cG \to p_* p^* \mathcal{G} \to \cK
\]
where the first morphism is induced by adjunction. Using the projection formula and the known description of the cohomology of $G/B$, one sees that the first morphism induces an isomorphism $\cG \simto \pH^0(p_* p^* \mathcal{G})$, and that $\cK$ is concentrated in perverse degrees $\geq 2$.
If $\cF$ also belongs to $\Shv(\siIw \backslash \Gr)^\heartsuit$, we deduce that $\Hom(p^\dag \mathcal{F}, p^\dag \mathcal{G}) \cong \Hom(\mathcal{F}, \mathcal{G})$, which proves fully faithfulness of $p^\dag$ on the hearts.

We will now prove that the essential image of $p^\dag$ is stable under subquotients. For that, as in the proof of Lemma~\ref{lem:oblv-siIw}, it suffices to prove that for any $\cF$ in $\Shv(\siIw \backslash \Fl)^\heartsuit$ the adjunction morphism
\[
p^\dag \pH^0(p_* \cF[-\ell(w_\circ)]) \to \cF
\]
is injective, which will follow if we prove that the third term $\cG$ of the fiber sequence based on the adjunction morphism
\[
p^* p_* \cF \to \cF
\]
belongs to $\Shv(\siIw \backslash \Fl)^{\geq 0}$. What we have to prove is therefore that
\begin{equation}
\label{eqn:vanishing-pdag}
 \Hom(\DFl^{\si}_w[n], \cG)=0
\end{equation}
for any $w \in W_\ext$ and $n \in \Z_{>0}$.

If $w \in W_\ext$ and $n \in \Z$ we have
\begin{multline*}
\Hom(\DFl^\si_w[n], p^* p_* \cF) \cong \Hom(\DFl^\si_w, p^! p_* \cF[-n-2\ell(w_\circ)]) \\
\cong \Hom(p_!\DFl^\si_w, p_* \cF[-n-2\ell(w_\circ)]) \cong \Hom(p^\dag p_!\DFl^\si_w, \cF[-n-\ell(w_\circ)]).
\end{multline*}
Writing $w=t_\lambda w_\fin$ with $\lambda \in \bY$ and $w_\fin \in W_\fin$ and using Lemma~\ref{lem:si-pushpull}\eqref{it:si-push} we deduce an isomorphism
\[
\Hom(\DFl^\si_w[n], p^* p_* \cF) \cong \Hom(p^\dag \DGr_\lambda^\si, \cF[-n-\ell(w_\circ)+\ell(w_\fin)]).
\]
Here $p^\dag \DGr_\lambda^\si$ belongs to $\Shv(\siIw \backslash \Fl)^{\leq 0}$, and $\cF$ belongs to $\Shv(\siIw \backslash \Fl)^\heartsuit$. If $w_\fin \neq w_\circ$, the cohomological shift in the right-hand side is negative for all $n \geq 0$, which implies~\eqref{eqn:vanishing-pdag} for these $w$'s (and all $n>0$).

Now, assume that $w_\fin=w_\circ$. The same considerations as above show that~\eqref{eqn:vanishing-pdag} holds for any $n>1$, and that for $n=1$ it will follow if we prove that the morphism
\begin{equation}
\label{eqn:vanishing-pdag-2}
 \Hom(\DFl^\si_{t_\lambda w_\circ}, p^* p_* \cF) \to \Hom(\DFl^\si_{t_\lambda w_\circ}, \cF)
\end{equation}
is injective. By the computations above and those in the proof of Lemma~\ref{lem:si-pushpull}\eqref{it:si-pulldelta} we have
\[
 \Hom(\DFl^\si_{t_\lambda w_\circ}, p^* p_* \cF) \cong \Hom(\DFl^\si_{t_\lambda} \star^{\Iw} \IC_{w_\circ}, \cF).
\]
We have a fiber sequence $\mathcal{L} \to \DFl_{w_\circ} \to \IC_{w_\circ}$ where $\mathcal{L}$ is an iterated extension of objects $\DFl_x[m]$ with $x \in W_\fin$ and $m \in \Z_{\geq 0}$. We deduce a fiber sequence
\[
 \DFl^\si_{t_\lambda} \star^{\Iw} \mathcal{L} \to \DFl^\si_{t_\lambda} \star^{\Iw} \DFl_{w_\circ} \to \DFl^\si_{t_\lambda} \star^{\Iw} \IC_{w_\circ}.
\]
Here, by Lemma~\ref{lem:convolution-DN-si}, the second term identifies with $\DFl^\si_{t_\lambda w_\circ}$, and the first one is an iterated extension of objects $\DFl^\si_{t_\lambda x}[m]$ with $x \in W_\fin$ and $m \in \Z_{\geq 0}$. We therefore have $\Hom(\DFl^\si_{t_\lambda} \star^{\Iw} \mathcal{L}[1], \cF)=0$, which implies the injectivity of~\eqref{eqn:vanishing-pdag-2} by a long exact sequence consideration.

Finally, to prove that the essential image of $p^\dag$ is closed under extensions it suffices to prove that for $\cF,\cG$ in $\Shv(\siIw \backslash \Gr)^{\heartsuit}$ the morphism
\[
 \Hom_{\Shv(\siIw \backslash \Gr)}(\cF,\cG[1]) \to \Hom_{\Shv(\siIw \backslash \Fl)}(p^\dag\cF,p^\dag\cG[1])
\]
induced by $p^\dag$ is an isomorphism. Using adjunction, this follows from our comments on the object ``$\cK$'' above.
%
\end{proof}


\subsection{Standard objects are perverse}

We will now prove that the standard objects are in the heart of our perverse t-structure, both on $\Fl$ and on $\Gr$. (For the latter case, we will later prove a stronger property, see Proposition~\ref{prop:si-simple} below.) 

\begin{lem}
\phantomsection
\label{lem:dsi-heart}
\begin{enumerate}
\item 
\label{it:dsi-heart-Fl}
For each $w \in W_\ext$, $\DFl^\si_w$ belongs to $\Shv(\siIw \backslash \Fl)^\heartsuit$.
\item 
\label{it:dsi-heart-Gr}
For each $\lambda \in \bY$, $\DGr^\si_\lambda$ belongs to $\Shv( \siIw \backslash \Gr)^\heartsuit$.
\end{enumerate}
\end{lem}

\begin{proof}
\eqref{it:dsi-heart-Fl}
As explained in~\S\ref{ss:perv-semiinf-def}, by construction
the object $\DFl^\si_w$ belongs to the subcategory $\Shv(\siIw \backslash \Fl)^{\leq 0}$. To show that it belongs to $\Shv(\siIw \backslash \Fl)^{\geq 0}$ we need to show that
\[
\Hom_{\Shv(\siIw \backslash \Fl)}(\DFl^\si_y[n],\DFl^\si_w) = 0
\]
for all $y \in W_\ext$ and all $n \ge 1$ or equivalently (by Proposition~\ref{prop:ras-waki}) that
\[
\Hom_{\Shv(\Iw \backslash \Fl)}(\WFl_y[n], \WFl_w) = 0
\]
for all $y \in W_\ext$ and all $n \ge 1$. This property holds 
because $\WFl_y$ and $\WFl_w$ are both perverse sheaves, see~\S\ref{ss:Waki-Fl}.

\eqref{it:dsi-heart-Gr}
The proof is similar to that of~\eqref{it:dsi-heart-Fl}, using
Lemma~\ref{lem:bwaki-hom} instead of the fact that Wakimoto sheaves are perverse.
\end{proof}


We will say that an object in $\Shv(\siIw \backslash \Fl)^\heartsuit$, resp.~in $\Shv(\siIw \backslash \Gr)^\heartsuit$, \emph{admits a standard filtration} if it admits a finite filtration in this abelian category whose associated graded pieces are of the form $\DFl^\si_w$ with $w \in W_\ext$, resp.~of the form $\DGr^\si_\lambda$ with $\lambda \in \bY$. We will mostly be interested in this notion for sheaves on $\Gr$. In this case, by adjunction and in view of~\eqref{eqn:closure-Slambda} we have
\[
\Ext^1_{\Shv(\siIw \backslash \Gr)^\heartsuit}(\DGr^\si_\lambda, \DGr^\si_\mu) \neq 0 \quad \Rightarrow \quad \lambda \prec \mu.
\]
Hence, if we choose (once and for all) a total order $\preceq^{\mathrm{tot}}$ on $\bY$ which refines the order $\preceq$, then for any object $\mathcal{F}$ in $\Shv(\siIw \backslash \Gr)^\heartsuit$ which admits a standard filtration there exists a \emph{unique} filtration
\[
0 = \mathcal{F}_0 \subset \mathcal{F}_1 \subset \cdots \subset \mathcal{F}_N = \mathcal{F},
\]
a unique collection of coweights
\[
\mu_1 \succ^{\mathrm{tot}} \mu_2 \succ^{\mathrm{tot}} \cdots \succ^{\mathrm{tot}} \mu_N
\]
and a unique collection of positive integers $n_1, \dots, n_N$ such that
\[
\mathcal{F}_i / \mathcal{F}_{i-1} \cong \bigl( \DGr^\si_{\mu_i} \bigr)^{\oplus n_i}
\]
for any $i \in \{1, \dots, N\}$. In this case, for any $i$ the complex $(\bi_{\mu_i})^* \mathcal{F}$ belongs to the heart of the perverse t-structure, and using the notation of~\S\ref{ss:si-orbits} we have
\[
n_i = \rank \bigl( (\bi_{\mu_i})^* \mathcal{F} \bigr).
\]
This integer will be called the \emph{multiplicity} of $\mu_i$ in $\mathcal{F}$.

\subsection{Restriction to constructible complexes}
\label{ss:restriction-const}

Let us now denote by
\[
\Shv_\omega(\siIwu \backslash \Gr)
\]
the inverse image of the subcategory $\Shvfg(\Iwu \backslash \Gr)$ under the Raskin equivalence
\[
\Shv(\siIwu \backslash \Gr) \simto \Shv(\Iwu \backslash \Gr),
\]
see~\S\ref{ss:raskin-equiv}. One defines similarly the $\infty$-categories
\[
 \Shv_\omega(\siIw \backslash \Gr), \quad \Shv_\omega(\siIwu \backslash \Fl), \quad \Shv_\omega(\siIw \backslash \Fl), \quad \Shv_\omega(\siIw \backslash \tFl).
\]
Note that $\Shv_\omega(\siIwu \backslash \Gr)$, $\Shv_\omega(\siIwu \backslash \Fl)$, $\Shv_\omega(\siIw \backslash \tFl)$ are the subcategories of compact objects in $\Shv(\siIwu \backslash \Gr)$, $\Shv(\siIwu \backslash \Fl)$ and $\Shv(\siIw \backslash \tFl)$ respectively; their objects are described in Lemma~\ref{lem:compact-objects}. The subcategories $\Shv_\omega(\siIw \backslash \Gr) \subset \Shv(\siIw \backslash \Gr)$, $\Shv_\omega(\siIw \backslash \Fl) \subset \Shv(\siIw \backslash \Fl)$ are \emph{not} the subcategories of compact objects, but their objects have a similar description; for instance, an object $\cF$ of $\Shv(\siIw \backslash \Fl)$ belongs to $\Shv_{\omega}(\siIw \backslash \Fl)$ iff:
  \begin{enumerate}
   \item $\cF$ is supported on a finite union of closures of orbits $\fS_w$;
   \item
   the set $\{w \in W_\ext \mid (\bi_w)^* \cF \neq 0\}$ is finite;
   \item
   for any $w \in W_\ext$ the image of $(\bi_w)^* \cF$ under the equivalence in~\eqref{eqn:si-sheaves-orbits-2} belongs to $\Shvc(T \backslash \pt)$.
  \end{enumerate}


\begin{lem}
\label{lem:t-struct-restrict}
 The perverse t-structure on $\Shv(\siIw \backslash \Fl)$ restricts to a t-structure on $\Shv_\omega(\siIw \backslash \Fl)$. A similar statement holds for 
 $\Shv(\siIwu \backslash \Fl)$, $\Shv(\siIw \backslash \Gr)$, $\Shv(\siIwu \backslash \Gr)$ and $\Shv(\siIw \backslash \tFl)$.
\end{lem}

\begin{proof}
 We explain the proof in the case of $\Shv(\siIw \backslash \Fl)$; the other cases are similar. What we have to prove is that the truncation functors with respect to the perverse t-structure preserve $\Shv_\omega(\siIw \backslash \Fl)$. This will be checked by induction on the number $\mathrm{N}(\cF)$ of elements $w \in W_\ext$ such that $(\bi_w)^* \cF \neq 0$. 
 If $\mathrm{N}(\cF) \leq 1$ then $\cF=(\bi_w)_! \cG$ for some $w \in W_\ext$ and some $\cG \in \Shv(\siIw \backslash \fS_w)$ which corresponds under the equivalence in~\eqref{eqn:si-sheaves-orbits-2} to a complex in $\Shvc(T \backslash \pt)$. By Lemma~\ref{lem:dsi-heart}, the perverse truncations of $\cF$ are obtained from the truncations of the latter complex (with respect to the perverse t-structure on $\Shvc(T \backslash \pt)$) via application of $(\bi_w)_!$, so that these objects are indeed in $\Shv_\omega(\siIw \backslash \Fl)$.
 
 For the induction step, we assume that $\cF$ is an object of $\Shv_\omega(\siIw \backslash \Fl)$ such that $\mathrm{N}(\cF) \geq 1$, and choose some $w \in W_\ext$ such that $\fS_w$ is maximal (with respect to the order given by inclusions of closures) among orbits on which the $*$-restriction of $\cF$ is nonzero. Then (see~\eqref{eqn:fiber-seq-open-closed}) we have a fiber sequence
 \[
  (\bi_w)_! (\bi_w)^* \cF \to \cF \to \cG
 \]
where $\cG$ belongs to $\Shv_\omega(\siIw \backslash \Fl)$ and satisfies $\mathrm{N}(\cG) = \mathrm{N}(\cF) -1$. Fix $n \in \Z$, and consider the corresponding truncation functors $\tau_{\leq n}$ and $\tau_{>n}$. By induction the objects 
\[
 \tau_{\leq n}((\bi_w)_! (\bi_w)^* \cF), \quad \tau_{>n}((\bi_w)_! (\bi_w)^* \cF), \quad \tau_{\leq n}(\cG), \quad \tau_{>n}(\cG)
\]
are in $\Shv_\omega(\siIw \backslash \Fl)$. Applying the nine-lemma (see~\cite[Proposition~1.1.11]{bbdg}) 
to the commutative diagram
\[
  \begin{tikzcd}
   \cG  \ar[r] \ar[d] & (\bi_w)_! (\bi_w)^* \cF[1] \ar[d] \\
   \tau_{>n}(\cG) \ar[r] & \tau_{>n}((\bi_w)_! (\bi_w)^* \cF)[1]
  \end{tikzcd}
\]
we obtain a diagram in the homotopy category
\[
 \begin{tikzcd}
 \tau_{\leq n}((\bi_w)_! (\bi_w)^* \cF) \ar[r] \ar[d] & \cF_1 \ar[r] \ar[d] & \tau_{\leq n}(\cG) \ar[r] \ar[d] & \tau_{\leq n}((\bi_w)_! (\bi_w)^* \cF)[1] \ar[d] \\
  (\bi_w)_! (\bi_w)^* \cF \ar[r] \ar[d] & \cF \ar[r] \ar[d] & \cG  \ar[r] \ar[d] & (\bi_w)_! (\bi_w)^* \cF[1] \ar[d] \\
  \tau_{>n}((\bi_w)_! (\bi_w)^* \cF) \ar[r] \ar[d] & \cF_2 \ar[r] \ar[d] & \tau_{>n}(\cG) \ar[r] \ar[d] & \tau_{>n}((\bi_w)_! (\bi_w)^* \cF)[1] \ar[d] \\
  \tau_{\leq n}((\bi_w)_! (\bi_w)^* \cF)[1] \ar[r] & \cF_1[1] \ar[r] & \tau_{\leq n}(\cG)[1] \ar[r] & \tau_{\leq n}((\bi_w)_! (\bi_w)^* \cF)[2]
 \end{tikzcd}
\]
for some objects $\cF_1$ and $\cF_2$,
in which all rows and columns are distinguished triangles.
Here $\cF_1$ belongs to $\Shv_\omega(\siIw \backslash \Fl)$ and to the ``$\leq n$'' part of the t-structure, and $\cF_2$ belongs to $\Shv_\omega(\siIw \backslash \Fl)$ and to the ``$> n$'' part of the t-structure. Hence we must have $\cF_1 \cong \tau_{\leq n}(\cF)$ and $\cF_2 \cong \tau_{> n}(\cF)$, which proves that these objects are in $\Shv_\omega(\siIw \backslash \Fl)$ and finishes the proof.
\end{proof}


Transporting the t-structures obtained from Lemma~\ref{lem:t-struct-restrict} along the corresponding Raskin equivalences, we obtain t-structures on the small stable $\infty$-categories 
\[
\Shvfg(\Iwu \backslash \Gr), \quad
 \Shvfg(\Iw \backslash \Gr), \quad \Shvfg(\Iwu \backslash \Fl), \quad \Shvfg(\Iw \backslash \Fl), \quad \Shvfg(\Iw \backslash \tFl).
\]
In the rest of this subsection we describe these t-structures intrinsically for the last two cases. These descriptions are inspired by a construction of Bezrukavnikov--Lin, see~\cite{bl}; the precise connection will be explained in~\S\ref{ss:comparison-new-tstr} below.

\begin{lem}
\phantomsection
\label{lem:t-str-Shvc}
\begin{enumerate}
\item
\label{it:t-str-Shvc-1}
An object $\cF$ in $\Shvfg(\Iw \backslash \Fl)$ belongs to the ``$\leq 0$'' part of the t-structure discussed above iff it satisfies the equivalent conditions of Lemma~\ref{lem:subcat}\eqref{it:subcat-1}.
\item
\label{it:t-str-Shvc-2}
An object $\cF$ in $\Shvfg(\Iw \backslash \tFl)$ belongs to the ``$\leq 0$'' part of the t-structure discussed above iff it satisfies the equivalent conditions of Lemma~\ref{lem:subcat}\eqref{it:subcat-2}.
\end{enumerate}
\end{lem}

\begin{proof}
We explain the proof of~\eqref{it:t-str-Shvc-1}; the other case is similar. Let $\mathsf{C}$ be the full subcategory of $\Shvfg(\Iw \backslash \Fl)$ consisting of the objects which satisfy the equivalent conditions of Lemma~\ref{lem:subcat}\eqref{it:subcat-1}; in other words, $\mathsf{C}$ is the full subcategory generated under (finite) extensions by the objects $\WFl_w[n]$ with $w \in W_\ext$ and $n \in \Z_{\geq 0}$. Let also $\mathsf{C}'$ be the full subcategory of $\Shv_\omega(\siIw \backslash \Fl)$ corresponding to $\mathsf{C}$; what we have to prove is that $\mathsf{C}' = \Shv_\omega(\siIw \backslash \Fl) \cap \Shv(\Iw \backslash \Fl)^{\leq 0}$.

By Proposition~\ref{prop:ras-waki}, $\mathsf{C}'$ is the full subcategory generated under (finite) extensions by the objects $\DFl^\si_w[n]$ with $w \in W_\ext$ and $n \in \Z_{\geq 0}$; it is therefore clear that $\mathsf{C}' \subset \Shv_\omega(\siIw \backslash \Fl) \cap \Shv(\Iw \backslash \Fl)^{\leq 0}$. Reciprocally, if $\cF$ belongs to $\Shv_\omega(\siIw \backslash \Fl) \cap \Shv(\Iw \backslash \Fl)^{\leq 0}$, then for any $w \in W_\ext$ the $*$-restriction of $\cF$ to $\fS_w$ is in the ``$\leq 0$'' part of the perverse t-structure, hence the corresponding object in $\Shvc(T \backslash \pt)$ belongs to the subcategory generated under extensions by nonnegative shifts of the constant sheaf. Now, as in~\cite[Lemma~4.4.3]{central} the complex $\cF$ belongs to the full subcategory generated under extensions by the $!$-pushforwards of the $*$-restrictions of $\cF$ to such orbits, where $w$ runs over the finitely many elements such that this restriction is nonzero. Hence $\cF$ belongs to $\mathsf{C}'$, which finishes the proof.
\end{proof}

\subsection{Comparison with the Bezrukav\-nikov--Lin ``new'' t-structure}
\label{ss:comparison-new-tstr}

In this subsection we explain that the t-structures considered in~\S\ref{ss:restriction-const} are closely related to t-structures that appeared in previous work of Bezrukavnikov--Lin~\cite{bl} and Losev~\cite{losev2}.
In order to match the combinatorics we are using here with that of~\cite{bl} one needs to
work with the versions of our $\infty$-categories in which the roles of left and right multiplication on $\Loop G$ are switched, cf.~Remark~\ref{rmk:Iw-eq-variants}\eqref{it:Ieq-cat-leftquot}. 

First, one needs to change the conventions for the definition of Wakimoto objects: one considers the collection $(\cJ_w : w \in W_\ext)$ of objects in $\Shv(\Iw \backslash \Fl)$ such that $\cJ_w$ is the image of $\WFl_{w^{-1}}$ under the t-exact autoequivalence of $\Shv(\Iw \backslash \Fl)$ induced by the map $g \mapsto g^{-1}$ on $\Loop G$. Therefore, 
this family satisfies
 \[
 \text{$\cJ_{t_\lambda} = \NFl_{t_\lambda}$ when $\lambda \in \bY_+$ and $\cJ_{t_\lambda} = \DFl_{t_\lambda}$ when $\lambda \in -\bY_+$,}
 \]
 and $\cJ_{t_\lambda} \star^{\Iw} \cJ_{t_\mu} \cong \cJ_{t_{\lambda+\mu}}$ for all $\lambda,\mu \in \bY$.
(The object $J_\lambda$ of~\cite{bl} is our $\cJ_{t_\lambda}$.) For a general $w \in W_\ext$ we have
 \[
 \cJ_w = \Delta_{w_\fin} \star^\Iw \cJ_{t_\lambda}
 \]
 if $w = w_\fin t_\lambda$ with $w_\fin \in W_\fin$ and $\lambda \in \bY$. 
 
 Next, one considers the equivalence
 \[
 \Shv(\Iw \backslash \tFl) \simto \Shv(\tFl' / \Iw)
 \]
 induced as above by the map $g \mapsto g^{-1}$ on $\Loop G$, and defines the perverse t-structure on $\Shv(\tFl' / \Iw)$ as the image of the perverse t-structure on $\Shv(\Iw \backslash \tFl)$ considered above. This t-structure restricts to a t-structure on $\Shvfg(\tFl' / \Iw)$. But, as explained in~\eqref{eqn:equiv-left-right}, we have a canonical equivalence
 \[
 \Shvfg(\tFl' / \Iw) \cong \Shvfg(\Iwu / \Fl).
 \]
 We have therefore obtained a t-structure on $\Shvfg(\Iwu \backslash \Fl)$, such that an object $\cF$ belongs to the ``$\leq 0$'' part of this t-structure iff for any $\lambda \in \bY$ the complex $\cF \star^{\Iw} \cJ_{t_\lambda}$ belongs to the ``$\leq 0$'' part of the perverse t-structure on $\Shvfg(\Iwu \backslash \Fl)$.

In case $\bk$ has characteristic $0$,
the condition stated above is exactly the condition that defines the ``$\leq 0$'' part of the ``new'' t-structure\footnote{I.~Losev argues in~\cite{losev2} that this terminology might not be appropriate, essentially since nothing can remain new forever. Instead he proposes the term \emph{stabilized} t-structure. In view of the current discussion, this t-structure might also be called the \emph{semiinfinite} t-structure.} considered in~\cite{bl}; see in particular~\cite[Corollary~1]{bl}. Since a t-structure is determined by its ``$\leq 0$'' part, this shows that our t-structure on $\Shvfg(\Iwu \backslash \Fl)$ coincides in this case with that studied in~\cite{bl}.

\subsection{Costandard objects on \texorpdfstring{$\Fl$}{Fl} are perverse}

In this subsection we prove the following counterpart for costandard objects of Lemma~\ref{lem:dsi-heart}\eqref{it:dsi-heart-Fl}.

\begin{prop}
\label{prop:nsi-heart}
For any $w \in W_\ext$, the object $\NFl_w^{\si}$ belongs to $\Shv(\siIw \backslash \Fl)^\heartsuit$.
\end{prop}

Before proceeding to the proof of this proposition, we will prove the following lemma, which is concerned with the objects $\WFl_w^{\vee,\mu}$ introduced in~\S\ref{ss:Waki-Fl-dual}.

\begin{lem}
\label{lem:dual-Wak-heart}
For any $w \in W_\ext$ and $\mu \in \bY_{+,w}$, the object $\WFl_w^{\vee,\mu} \in \Shvfg(\Iw \backslash \Fl)$ belongs to the heart of the t-structure considered in~\S\ref{ss:restriction-const}.
\end{lem}

\begin{proof}
For any $\nu \in \bY$ we have 
\[
\WFl_{t_\nu} \star^{\Iw} \WFl_w^{\vee,\mu} = \WFl_{t_\nu} \star^{\Iw} \NFl_{t_{-\mu}} \star^{\Iw} \NFl_{t_\mu w} = \WFl_{t_\nu} \star^{\Iw} \WFl_{t_{-\mu}} \star^{\Iw} \NFl_{t_\mu w} = \WFl_{t_{\nu-\mu}} \star^{\Iw} \NFl_{t_\mu w}.
\]
Now, recall that (left or right) convolution with $\NFl_{t_\mu w}$ is right t-exact with respect to the perverse t-structure on $\Shvfg(\Iw \backslash \Fl)$, see the proof of~\cite[Lemma~4.1.7]{central}. (This property is closely related with Lemma~\ref{lem:standards-costandards}\eqref{it:standards-costandards-3}.) Since $\WFl_{t_{\nu-\mu}}$ is perverse, this shows that $\WFl_{t_\nu} \star^{\Iw} \WFl_w^{\vee,\mu}$ in concentrated in nonpositive perverse degrees for any $\nu$, hence (in view of Lemma~\ref{lem:t-str-Shvc}\eqref{it:t-str-Shvc-1}) that $\WFl_w^{\vee,\mu}$ belongs to the ``$\leq 0$'' part of the t-structure under consideration.

To prove that it also belongs to the ``$\geq 0$'' part, we need to prove that for any $\cF \in \Shvfg(\Iw \backslash \Fl)^{<0}$ we have
\[
\Hom_{\Shvfg(\Iw \backslash \Fl)}(\cF, \WFl_w^{\vee,\mu})=0.
\]
Now we have
\begin{multline*}
\Hom_{\Shvfg(\Iw \backslash \Fl)}(\cF, \WFl_w^{\vee,\mu}) = \Hom_{\Shvfg(\Iw \backslash \Fl)}(\cF, \WFl_{t_{-\mu}} \star^{\Iw} \NFl_{t_\mu w}) \\
\cong \Hom_{\Shvfg(\Iw \backslash \Fl)}(\WFl_{t_{\mu}} \star^{\Iw} \cF, \NFl_{t_\mu w}).
\end{multline*}
By Lemma~\ref{lem:t-str-Shvc}\eqref{it:t-str-Shvc-1}, the complex $\WFl_{t_{\mu}} \star^{\Iw} \cF$ is concentrated in negative perverse degrees. Since $\NFl_{t_\mu w}$ is perverse, we deduce the desired vanishing statement.
\end{proof}

\begin{proof}[Proof of Proposition~\ref{prop:nsi-heart}]
Since the forgetful functor $\oblv^{\siIw | \siIwu}$ detects perversity (see~\S\ref{ss:perv-semiinf-def}), it suffices to prove that each $\oblv^{\siIw | \siIwu}(\NFl^\si_w)$ belongs to the heart of the perverse t-structure. This follows from Corollary~\ref{cor:ras-waki-dual-oblv}, Lemma~\ref{lem:dual-Wak-heart}, and the fact that the perverse t-structure is compatible with filtered colimits.
\end{proof}

Note that the counterpart of Lemma~\ref{lem:dsi-heart}\eqref{it:dsi-heart-Gr} for costandard objects in \emph{not} true in general. In fact we will describe in Proposition~\ref{prop:properties-Gaits} below the $0$-th perverse cohomology of $\NGr^{\si}_0$. Since this object has nontrivial costalks (at least when $\bk$ has good characteristic), see Theorem~\ref{thm:properties-Ga}, it does not coincide with $\NGr^{\si}_0$, so that this object is not perverse.

We can now ``lift'' Corollary~\ref{cor:ras-waki-dual-oblv} to the $\Iw$-equivariant setting.

\begin{cor}
\label{cor:ras-waki-dual}
For any $w \in W_\ext$, resp.~$w \in W_\ext$, resp.~$\lambda \in \bY$, we have a canonical isomorphism
\[
\Ras(q^! \NFl^\si_w) \cong q^! \WFl^\vee_w, \quad \text{resp.} \quad
\Ras( \NFl^\si_w) \cong \WFl^\vee_w, \quad \text{resp.} \quad
\Ras( \NGr^\si_\lambda) \cong \WGr^\vee_\lambda.
\]
\end{cor}

\begin{proof}
It suffices to prove the second isomorphism, or in other words that $\NFl^\si_w \cong \Ras^{-1}(\WFl^\vee_w)$. Now $\NFl^\si_w$ belongs to the heart of the perverse t-structure, and so does $\Ras^{-1}(\WFl^\vee_w)$ since its image under $\oblv^{\siIw | \siIwu}$ is $\oblv^{\siIw | \siIwu}(\NFl^\si_w)$ by Corollary~\ref{cor:ras-waki-dual-oblv}, which is perverse. Hence the desired isomorphism follows from Lemma~\ref{lem:oblv-siIw} and Corollary~\ref{cor:ras-waki-dual-oblv}.
\end{proof}

\begin{rmk}
The fact that costandard objects are perverse has the following very concrete reformulation, that was suggested to us by Alexis Bouthier. As explained above, the nontrivial fact is that each $\NFl^\si_w$ belongs to $\Shv(\siIw \backslash \Fl)^{\leq 0}$, which by Remark~\ref{rmk:perverse-degrees-costalk-2} means that for any $y \in W_\ext$ the complex $(\mathbf{k}_y)^! (\bi_y)^* \NFl^\si_w$ is concentrated in degrees $\leq \psdim(\fS_y)$. Now by Braden's theorem (see e.g.~\cite[Lemma~2.2.4]{gaitsgory}) we have a canonical isomorphism
\[
(\mathbf{k}_y)^! (\bi_y)^* \NFl^\si_w \cong \mathsf{H}^\bullet(\fS_y^-, (\bi_y^-)^! \NFl^\si_w).
\]
Here $\fS_y^- \subset \Fl$ is the $\Loop U^-$-orbit of the point associated with $y$, and $\bi_y^- : \fS_y^- \to \Fl$ is the embedding. Using the base change theorem, one sees that we have
\[
\mathsf{H}^\bullet(\fS_y^-, (\bi_y^-)^! \NFl^\si_w) \cong \mathsf{H}^\bullet(\fS_y^- \cap \fS_w, \omega_{\fS_y^- \cap \fS_w}[\psdim(\fS_w)]).
\]
Hence the fact that $\NFl^\si_w$ is perverse amounts to the property that
\[
\mathsf{H}^m(\fS_y^- \cap \fS_w, \omega_{\fS_y^- \cap \fS_w}) \neq 0 \quad \Rightarrow \quad m \leq \psdim(\fS_y) - \psdim(\fS_w).
\]
We have not been able to find a direct proof of this property, however.
\end{rmk}

\subsection{Some semiinfinite perverse sheaves arising from spherical perverse sheaves on \texorpdfstring{$\Gr$}{Gr}}
\label{ss:some-perv-sheaves}

Recall the Satake equivalence $\Sat$ from~\S\ref{ss:spherical-complexes}.
For an object $V$ in $\Rep(G^\vee_\bk)^\heartsuit$ and $\lambda \in \bY$, we will denote by $V_\lambda$ the $\lambda$-weight space of $V$.

\begin{lem}
\label{lem:Satake-standard-fil}
Let $\mathcal{F} \in \Shvfg(\Loop^+ G \backslash \Gr)^\heartsuit$.
 For any $\lambda \in \bY$, the object
\[
\Ras^{-1}(\WFl_{t_{\lambda}} \star^{\Iw} \oblv^{\Loop^+ G | \Iw}(\mathcal{F})) 
\quad \in \Shv(\siIw \backslash \Gr)
\]
belongs to the heart of the perverse t-structure, and admits a standard filtration. The associated collection of coweights consists of the elements of the form $\lambda + \mu$ where $\mu$ is a $T^\vee_\bk$-weight of $\Sat(\mathcal{F})$, and the multiplicity of $\lambda + \mu$ is $\dim(\Sat(\mathcal{F})_\mu)$.
\end{lem}

\begin{proof}
Recall Gaitsgory's central functor
\[
\mathrm{Z} : \Shvfg(\Loop^+ G \backslash \Gr)^\heartsuit \to \Shvfg(\Iw \backslash \Fl)^\heartsuit,
\]
see~\cite[\S 2.4.5]{central}. Then by results of Arkhipov--Bezrukavnikov, if $\cF$ is as in the statement, $\mathrm{Z}(\mathcal{F})$ admits a finite filtration whose associated graded pieces are Wakimoto sheaves of the form $\WFl_{t_\nu}$ with $\nu \in \bY$, see~\cite[Theorem~4.4.5]{central}, and moreover, for any $\nu \in \bY$, the number of occurrences of $\WFl_{t_\nu}$ as a subquotient in such a filtration is the dimension of the $\nu$-weight space 
of $\Sat(\mathcal{F})$, see~\cite[Lemma~4.8.1]{central}.
Using the fact that
\[
\oblv^{\Loop^+ G | \Iw}(\mathcal{F}) \cong p_* \mathrm{Z}(\mathcal{F}),
\]
see~\cite[Lemma~2.5.1]{central}, and Proposition~\ref{prop:ras-waki}, we deduce the desired claims.
\end{proof}

\begin{rmk}
\phantomsection
\label{rmk:Satake-standard-fil}
\begin{enumerate}
 \item 
\label{it:exactness-convolution}
 It is a standard fact that the bifunctor
\[
 \Shv(\Iw \backslash \Fl) \times \Shv(\Loop^+ G \backslash \Gr) \to \Shv(\Iw \backslash \Gr)
\]
given by $(\cF,\cG) \mapsto \cF \star^{\Iw} \oblv^{\Loop^+ G | \Iw}(\cG)$ identifies with the bifunctor given by $(\cF, \cG) \mapsto (p_* \cF) \star^{\Loop^+ G} \cG$. In particular, for $\cF$ and $\lambda$ as in the lemma we have an identification
\[
 \WFl_{t_{\lambda}} \star^{\Iw} \oblv^{\Loop^+ G | \Iw}(\mathcal{F}) \cong \WGr_{\lambda} \star^{\Loop^+ G} \cF.
\]
Then, using the fact that $\Ras$ commutes with the right action of the $\infty$-category $\Shv(\Loop^+ G \backslash \Gr)$ and Proposition~\ref{prop:ras-waki}, we deduce an isomorphism
\[
 \Ras^{-1}(\WFl_{t_{\lambda}} \star^{\Iw} \oblv^{\Loop^+ G | \Iw}(\mathcal{F})) \cong \DGr^{\si}_\lambda \star^{\Loop^+ G} \cF.
\]
From this point of view, Lemma~\ref{lem:Satake-standard-fil} implies that right convolution with an object of $\Shvfg(\Loop^+ G \backslash \Gr)^\heartsuit$ 
is right t-exact. As in~\cite[Proposition~2.8.2]{gaitsgory}, it easily follows that right convolution with any object of $\Shv(\Loop^+ G \backslash \Gr)^\heartsuit$ on $\Shv(\siIw \backslash \Gr)$ is t-exact.
\item
\label{it:multiplicities-objects-Sat}
One can be more explicit in the description of multiplicities in the filtration of Lemma~\ref{lem:Satake-standard-fil}. Namely, for $\cF$ as in the lemma, the Wakimoto filtration of $\mathrm{Z}(\cF)$ is canonical, and the multiplicity space of the subquotient corresponding to $\nu$ identifies canonically with $\Sat(\cF)_\nu$; see~\cite[Lemma~4.8.1]{central}.\footnote{Since some normalizations used in the present paper are different from those of~\cite{central}, one should be careful when translating statements from this reference. Concretely, in view of the standard description of $T^\vee_\bk$-weight spaces in terms of the Satake equivalence via semiinfinite orbits, the statement considered here boils down to the fact that $\mathsf{H}^n_c(\fS_{t_\lambda}, \WFl_{t_\mu})$ is $1$-dimensional (with a canonical generator) if $\lambda=\mu$ and $n=\langle 2\rho, \lambda \rangle$, and $0$ otherwise.} 
Using this we deduce that, for any $\mu \in \bY$, under the equivalence of~\eqref{eqn:si-sheaves-orbits}, the object
$\pH^0 \bigl( (\bi_\mu)^* \Ras^{-1}(\WFl_{t_{\lambda}} \star^{\Iw} \oblv^{\Loop^+ G | \Iw}(\mathcal{F})) \bigr)$ corresponds to $\Sat(\cF)_{\lambda+\mu}$.
\item
In view of Lemma~\ref{lem:Ras-twist}, the object considered in Lemma~\ref{lem:Satake-standard-fil} can also be described as $\Ras^{-1}(\oblv^{\Loop^+ G | \Iw}(\mathcal{F})) \langle \lambda \rangle$.
\end{enumerate}
\end{rmk}

For $\lambda \in \bY_{+}$, we set
\[
\cM_\lambda = \Ras^{-1} \bigl( \WFl_{t_{-\lambda}} \star^{\Iw} \oblv^{\Loop^+ G | \Iw}(\cI_*(\lambda)) \bigr) 
\quad \in \Shv(\siIw \backslash \Gr).
\]
To state the main result about these objects, we need some notation. 
Recall the positive coroots $\fR^\vee_+ \subset \bY$, see~\S\ref{ss:Weyl}. Given a function $\phi: \fR^\vee_+ \to \Z_{\ge 0}$,
we set
\[
\sigma(\phi) = \sum_{\alpha \in \fR^\vee_+} \phi(\alpha)\alpha \in \bY,
\qquad
|\phi| = \sum_{\alpha \in \fR^\vee_+} \phi(\alpha) \in \Z_{\ge 0}.
\]
With this notation,
the \emph{$q$-analogue of Kostant's partition function} (see~\cite{lusztig}) is the rule that assigns to any $\lambda \in \bY$ a polynomial $\cP(\lambda,q) \in \Z[q]$ by the formula
\[
\cP(\lambda,q) = \sum_{\substack{\phi: \fR^\vee_+ \to \Z_{\ge 0} \\ \sigma(\phi) = \lambda}} q^{|\phi|}.
\]
Of course, $\cP(\lambda, q)$ can be nonzero only when $\lambda \in \Z_{\geq 0} \fR^\vee_+$.

For the next statement, we 
will denote by $\rho^\vee$ the halfsum of the positive coroots. (This element belongs to $\Q \otimes_\Z \bY$ in general, but for any $w \in W_\fin$ the difference $w(\rho^\vee)-\rho^\vee$ belongs to $\bY$.)


\begin{lem}
\label{lem:mlambda-stalk}
Let $\lambda \in \bY_{+}$.
\begin{enumerate}
\item 
\label{it:mlambda-filt}
The object $\cM_\lambda$ belongs to $\Shv(\siIw \backslash \Gr)^\heartsuit$, and admits a standard filtration. The coweights $\mu$ appearing in this filtration satisfy
\[
0 \succeq \mu \succeq w_\circ(\lambda) - \lambda;
\]
in particular, $\cM_\lambda$ is supported on $\overline{\rS_0}$.
\item 
\label{it:mlambda-stalk}
The stalks of $\cM_\lambda$ are given by the following formula for $\nu \in \bY$:
\begin{multline*}
\rank \left( \pH^n((\bi_\nu)^*\cM_\lambda) \right) = \\
\begin{cases}
\sum\limits_{w \in W_\fin} (-1)^{\ell(w)} \cP(w(\lambda + \rho^\vee) - (\lambda + \nu + \rho^\vee), 1) & \text{if $n = 0$,} \\
0 & \text{otherwise.}
\end{cases}
\end{multline*}
\item 
\label{it:mlambda-h01}
Let $\nu \in \bY$, and
assume that $\lambda+\nu$ is dominant. Then we have
\[
\rank \left( \pH^0((\bi_\nu)^!\cM_\lambda) \right) =
\begin{cases}
1 & \text{if $\nu = 0$,} \\
0 & \text{otherwise},
\end{cases}
\qquad
\pH^1(\bi_\nu^!\cM_\lambda) = 0.
\]
\item 
\label{it:mlambda-costalk}
Assume that $\chr(\bk)$ is good for $G$. If $\lambda+\nu$ is dominant, the corestriction of $\cM_\lambda$ to $\rS_\nu$ is given by
\[
\sum_{n \in \Z} \rank \left( \pH^n((\bi_\nu)^!\cM_\lambda) \right) \cdot q^{n} =\sum_{w \in W_\fin } (-1)^{\ell(w)} \cP(w(\lambda + \rho^\vee) - (\lambda + \nu + \rho^\vee), q^2).
\]
\end{enumerate}
\end{lem}

\begin{proof}
In this proof, to simplify notation we omit the functor $\oblv^{\Loop^+ G | \Iw}$.

Parts~\eqref{it:mlambda-filt} and~\eqref{it:mlambda-stalk} follow from Lemma~\ref{lem:Satake-standard-fil}, the description of $\Sat(\cI_*(\lambda))$ recalled in~\S\ref{ss:spherical-complexes}, and Kostant's multiplicity formula for the dimension of weight spaces of dual Weyl modules.

For parts~\eqref{it:mlambda-h01} and~\eqref{it:mlambda-costalk}, using Proposition~\ref{prop:ras-waki} we have
\begin{multline*}
\rank \left( \pH^n((\bi_\nu)^!\cM_\lambda) \right) = \dim \Hom(\oblv^{\siIw | \siIwu}(\DGr^\si_\nu), \oblv^{\siIw | \siIwu}(\cM_\lambda)[n]) \\
= \dim \Hom(\oblv^{\Iw | \Iwu}(\WGr_\nu), \oblv^{\Iw | \Iwu}(\WFl_{t_{-\lambda}} \star^{\Iw} \cI_*(\lambda))[n]).
\end{multline*}
Now, similarly as in the proof of Proposition~\ref{prop:Hom-Wak-dual}, it is a standard fact that there exists an auto-equivalence of $\Shv(\Iwu \backslash \Gr)$ which makes the following diagram commutative:
\[
\begin{tikzcd}[column sep=large]
\Shv(\Iw \backslash \Gr) \ar[d, "\oblv^{\Iw | \Iwu}"'] \ar[r, "\WFl_{t_\lambda} \star^{\Iw} (-)"] & \Shv(\Iw \backslash \Gr) \ar[d, "\oblv^{\Iw | \Iwu}"] \\
\Shv(\Iwu \backslash \Gr) \ar[r] & \Shv(\Iwu \backslash \Gr).
\end{tikzcd}
\]
 Using this equivalence and~\eqref{eqn:Waki-convolution-Gr} we obtain an equality
\begin{multline*}
\dim \Hom(\oblv^{\Iw | \Iwu}(\WGr_\nu), \oblv^{\Iw | \Iwu}(\WFl_{t_{-\lambda}} \star^{\Iw} \cI_*(\lambda))[n]) \\
= \dim \Hom(\oblv^{\Iw | \Iwu}(\WGr_{\nu+\lambda}), \oblv^{\Iw | \Iwu}(\cI_*(\lambda))[n]).
\end{multline*}


Now, assume that $\nu+\lambda$ is dominant. In this case, by~\eqref{eqn:bWaki-dom-antidom} the number considered above is
\begin{equation}
\label{eqn:sph-costalk}
\dim \Hom(\oblv^{\Iw | \Iwu}(\DGr_{\nu+\lambda}), \oblv^{\Iw | \Iwu}(\cI_*(\lambda))[n]),
\end{equation}
i.e., the rank in perverse degree $n$ of the corestriction of $\cI_*(\lambda)$ to $\Gr_{\lambda+\nu}$.
Using the truncation fiber sequence
\[
\cI_*(\lambda) \to (j^\lambda)_* \underline{\bk}_{\Gr^\lambda}[\langle \lambda, 2\rho \rangle] \to \tau_{>0}((j^\lambda)_* \underline{\bk}_{\Gr^\lambda}[\langle \lambda, 2\rho \rangle])
\]
where $j^\lambda$ is the embedding of $\Gr^\lambda$ in $\Gr$,
one sees that for $n=1$ the number above is $0$ for all $\nu$, and if $n=0$ it is also zero unless $\Gr_{\lambda+\nu} \subset \Gr^\lambda$, i.e.~unless $\nu=0$.
This proves~\eqref{it:mlambda-h01}.

Finally, by the corrected version of the Mirkovi\'c--Vilonen conjecture~\cite[Conjecture~13.3]{mv} proved in~\cite{arider, mr}, the dimension in~\eqref{eqn:sph-costalk} is independent of $\bk$ as long as $\bk$ has good characteristic.  When $\mathrm{char}(\bk)=0$, we have $\cI_*(\lambda)=\ICGr_\lambda$; the dimension under consideration is therefore given in~\cite[Theorem~1.8]{kato} (see also~\cite[Eq.~(9.4)]{lusztig}) and is equal to the formula in the statement of part~\eqref{it:mlambda-costalk} of the lemma.
\end{proof}

\begin{rmk}
\label{rmk:multiplicity-canonicity}
 In view of Remark~\ref{rmk:Satake-standard-fil}\eqref{it:multiplicities-objects-Sat}, 
 for any $\nu \in \bY$ and $\lambda \in \bY_+$, under the equivalence from~\eqref{eqn:si-sheaves-orbits} the object $\pH^0((\bi_\nu)^* \cM_\lambda)$ corresponds to the weight space $\mathsf{N}^{\vee}(\lambda)_{\lambda+\nu}$. (See~\S\ref{ss:spherical-complexes} for the notation.)
Geometrically, denoting by $B^{\vee,-}_\bk$ the Borel subgroup of $G^\vee_\bk$ containing $T^\vee_\bk$ whose Lie algebra has as nonzero $T^\vee_\bk$-weights the negative coroots, $\mathsf{N}^{\vee}(\lambda)$ identifies with the space of functions $f \in \mathscr{O}(G^\vee_\bk)$ which satisfy
\[
 f(gb) = (\lambda+\nu)^{-1}(b) f(g) \quad \text{for any $g \in G^\vee_\bk$ and $b \in B^{\vee,-}_\bk$.}
\]
(Here, we still denote by $\lambda+\nu$ the unique extension of this character to a character of $B^{\vee,-}_\bk$.)
Denoting by $U^{\vee,+}_\bk$ the unipotent radical of the Borel subgroup opposite to $B^{\vee,-}_\bk$ with respect to $T^\vee_\bk$, and using the fact that the multiplication morphism $U^{\vee,+}_\bk \times B^{\vee,-}_\bk \to G^\vee_\bk$ is an open immersion, we obtain an embedding
\[
 \mathsf{N}^{\vee}(\lambda) \hookrightarrow \mathscr{O}(U^{\vee,+}_\bk).
\]
It is a standard fact that for a fixed $\nu \in \bY$, if $\lambda$ is sufficiently dominant this embedding induces an isomorphism
\[
 \mathsf{N}^{\vee}(\lambda)_{\lambda+\nu} \simto \mathscr{O}(U^{\vee,+}_\bk)_{\nu},
\]
where the right-hand side denotes the $\nu$-weight space for the action induced by conjugation of $T^\vee_\bk$ on $U^{\vee,+}_\bk$. For such $\lambda$, we deduce that under the equivalence from~\eqref{eqn:si-sheaves-orbits} the object $\pH^0((\bi_\nu)^* \cM_\lambda)$ corresponds to $\mathscr{O}(U^{\vee,+}_\bk)_{\nu}$.
\end{rmk}

\subsection{Simple semiinfinite perverse sheaves on \texorpdfstring{$\Gr$}{Gr}}

We are now in a position to show that the costandard objects in $\Shv(\siIw \backslash \Gr)$ are simple perverse sheaves.

\begin{prop}
\label{prop:si-simple}
For all $\lambda \in \bY$, $\DGr^\si_\lambda$ is a simple object in $\Shv(\siIw \backslash \Gr)^\heartsuit$.  Moreover, every simple object is isomorphic to one of this form.
\end{prop}

\begin{proof}
By Lemma~\ref{lem:heart-zero}\eqref{it:heart-zero-3} and Lemma~\ref{lem:dsi-heart}\eqref{it:dsi-heart-Gr}, any simple object in $\Shv(\siIw \backslash \Gr)^\heartsuit$ is a quotient of an object $\DGr^\si_\lambda$ with $\lambda \in \bY$. Hence the second claim is a consequence of the first one. Moreover, by the same considerations the first claim will follow if we prove that
\begin{equation}
\label{eqn:Hom-Deltasi-Gr}
\Hom(\DGr^\si_\lambda, \DGr^\si_\mu) = \begin{cases}
\bk \cdot \id & \text{if $\lambda=\mu$;} \\
0 & \text{otherwise.}
\end{cases}
\end{equation}
The case $\lambda=\mu$ is clear by adjunction and~\eqref{eqn:si-sheaves-orbits-2}, so only the case $\lambda \neq \mu$ requires some argument.
Moreover, by~\eqref{eqn:twist-siDN-Gr},
this property in general will follow if we prove it in case $\mu=0$.

So, fix $\lambda \in \bY \smallsetminus \{0\}$, and choose $\nu \in \bY_+$ such that $\lambda+\nu$ is dominant.
By Lemma~\ref{lem:mlambda-stalk}\eqref{it:mlambda-filt}, there is an embedding $\DGr^\si_0 \hookrightarrow \cM_\nu$
in $\Shv(\siIw \backslash \Gr)^\heartsuit$.
Applying the left exact functor $\Hom(\DGr^\si_\lambda, -)$ we deduce an embedding
\[
\Hom(\DGr^\si_\lambda, \DGr^\si_0) \hookrightarrow \Hom(\DGr^\si_\lambda,\cM_\nu).
\]
Here the second term vanishes by adjunction and Lemma~\ref{lem:mlambda-stalk}\eqref{it:mlambda-h01}, so the first term does as well.
%
\end{proof}

\begin{rmk}
\phantomsection
\label{rmk:simples-Gr}
\begin{enumerate}
\item
\label{it:classification-simples-Gr}
It follows from~\eqref{eqn:Hom-Deltasi-Gr} that the objects $\DGr^\si_\lambda$ are pairwise nonisomorphic. Hence the set of isomorphism classes of simple objects in the abelian category $\Shv(\siIw \backslash \Gr)^\heartsuit$ is in a canonical bijection with $\bY$. Each object $\oblv^{\siIw | \siIwu}(\DGr^\si_\lambda)$ is also simple by Lemma~\ref{lem:oblv-siIw}, and using Lemma~\ref{lem:heart-zero}\eqref{it:heart-zero-4} one sees that any simple object in $\Shv(\siIwu \backslash \Gr)^\heartsuit$ is of this form. The set of isomorphism classes of simple objects in $\Shv(\siIwu \backslash \Gr)^\heartsuit$ is therefore also in a canonical bijection with $\bY$. (These facts can also be obtained using the techniques from~\S\ref{ss:recollement} below.)
\item
 In case $\bk$ is a finite extension of $\Q_\ell$, Proposition~\ref{prop:si-simple} is established in~\cite[Theorem~1.5.5]{gaitsgory}. Note that Gaitsgory's proof relies on the description of the intersection cohomology of Drinfeld's compactifications established (in this case) in~\cite{bfgm}.
 \item
 \label{it:Verdier-duality}
 Reducing the computation to the case $G=\mathrm{SL}_2$ and then using Corollary~\ref{cor:ras-waki-dual-oblv}, one can check that if $\alpha^\vee \in \fRs^\vee$ we have
 \[
 \Hom_{\Shv(\siIwu \backslash \Gr)}(\DGr^\si_0, \DGr^\si_{\alpha^\vee}[1])=\bk, \quad  \Hom_{\Shv(\siIwu \backslash \Gr)}( \DGr^\si_{\alpha^\vee}, \DGr^\si_0[1])=0.
 \]
 Hence there cannot exist any anti-autoequivalence of $\Shv(\siIwu \backslash \Gr)$ (or any stable full subcategory containing the simple perverse sheaves) that fixes all simple perverse sheaves. In this sense, this $\infty$-category ``has no Grothen\-dieck--Verdier duality.'' Using also Lemma~\ref{lem:pullback-exact}, this comment can be extended to $\Shv(\siIwu \backslash \Fl)$.
 \item
 On an $\F$-scheme of finite type endowed with an appropriate stratification, if the intersection cohomology complex associated with the constant local system on a stratum coincides with the $!$-extension of this local system, then by Grothendieck--Verdier duality it also coincides with its $*$-extension. (Such simple perverse sheaves are usually called ``clean.'') In the present situation, it is \emph{not} true that $\NGr^\si_0$ is a simple perverse sheaf, as we will see in Proposition~\ref{prop:properties-Gaits} below.
 \end{enumerate}
\end{rmk}


Once this proposition is established we can describe the composition factors of the objects $\cM_\lambda$, as follows.

\begin{cor}
\label{cor:mlambda-soc}
For all $\lambda \in \bY_+$, $\cM_\lambda$ has finite length and a unique semisimple subobject, isomorphic to $\DGr^\si_0$.  Its composition multiplicities are given by
\[
[\cM_\lambda: \DGr^\si_\nu] = \sum_{w \in W_\fin} (-1)^{\ell(w)} \cP(w(\lambda + \rho^\vee) - (\lambda + \nu + \rho^\vee), 1).
\]
In particular, we have $[\cM_\lambda : \DGr^\si_0] = 1$.
\end{cor}

\begin{proof}
In view of Proposition~\ref{prop:si-simple}, every object in $\Shv(\siIw \backslash \Gr)$ which admits a standard filtration has finite length, and its composition factors are the standard objects appearing in a standard filtration. All the claims in the statement therefore directly follow from Lemma~\ref{lem:mlambda-stalk}\eqref{it:mlambda-filt}--\eqref{it:mlambda-stalk}, except for the one concerning the socle. To show this claim,
we must show that for any $\nu \in \bY$ we have
\[
\Hom(\DGr^\si_\nu, \cM_\lambda) \cong
\begin{cases}
\bk & \text{if $\nu = 0$,} \\
0 & \text{otherwise.}
\end{cases}
\]
(This fact has been established in Lemma~\ref{lem:mlambda-stalk}\eqref{it:mlambda-h01} in case $\lambda+\nu \in \bY_+$, but here we do not assume that this condition is satisfied.)
If $\nu+\lambda$ is not a weight of $\mathsf{N}^{\vee}(\lambda)$,
this is obvious as $[\cM_\lambda : \DGr^\si_\nu] = 0$.  
Otherwise, applying the equivalence $\Waki_{t_\lambda} \star^{\Iw} \Ras({-})$ and using Proposition~\ref{prop:ras-waki} and~\eqref{eqn:Waki-convolution-Gr} we obtain that
\[
\Hom(\DGr^\si_\nu, \cM_\lambda) \cong \Hom(\WGr_{\lambda+\nu}, \oblv^{\Loop^+ G | \Iw}(\cI_*(\lambda))).
\]
Here we have $\ell(w_{\lambda+\nu}) < \langle \lambda,2\rho \rangle$ if $\nu \neq 0$, so that
the desired claim follows from Lemma~\ref{lem:waki-costd-hom}.
\end{proof}

\subsection{Recollement and classification of simple objects}
\label{ss:recollement}

In this subsection, to fix notation we only treat the case of the ind-scheme $\Fl$, in the $\siIw$-setting, but all our considerations apply with minor changes to $\Gr$ or $\tFl$, and to the $\siIwu$-equivariant versions.

Let $Z \subset \Fl$ be either $\Fl$ or a finite union of closures $\overline{\fS_w}$ ($w \in W_\ext$), and let $Z' \subset Z$ be a finite union of such closures. Let also $V = Z \smallsetminus Z'$, and denote by
\[
 i : Z' \to Z, \quad j : V \to Z
\]
the embeddings.
Consider the presentable stable $\infty$-categories
\begin{equation}
\label{eqn:cat-recollement}
 \Shv(\siIw \backslash Z'), \quad \Shv(\siIw \backslash Z), \quad \Shv(\siIw \backslash V).
\end{equation}
As in~\S\ref{ss:functorialities}, these $\infty$-categories, together with the push/pull functors associated with $i$ and $j$, satisfy the recollement formalism of~\cite{bbdg,loregian}.

The definition of the perverse t-structure given in~\S\ref{ss:perv-semiinf-def} applies in the $\infty$-categories in~\eqref{eqn:cat-recollement}, simply restricting to the indices $w$ such that $\fS_w$ is contained in the corresponding piece of $\Fl$, and replacing the standard objects by their $*$-restrictions to the corresponding piece. We will use obvious notation for these t-structures, like $(\Shv(\siIw \backslash Z)^{\leq 0}, \Shv(\siIw \backslash Z)^{\geq 0})$.

\begin{lem}
\label{lem:recollement}
 The perverse t-structure on $\Shv(\siIw \backslash Z)$ is obtained by recollement from the perverse t-structures on $\Shv(\siIw \backslash Z')$ and $\Shv(\siIw \backslash V)$.
\end{lem}

\begin{proof}
 What we need to prove is that an object $\cF \in \Shv(\siIw \backslash Z)$ belongs to $\Shv(\siIw \backslash Z)^{\leq 0}$, resp.~$\Shv(\siIw \backslash Z)^{\geq 0}$, iff it satisfies
 \[
  j^* \cF \in \Shv(\siIw \backslash V)^{\leq 0} \quad \text{and} \quad i^* \cF \in \Shv(\siIw \backslash Z')^{\leq 0},
 \]
resp.
 \[
  j^* \cF \in \Shv(\siIw \backslash V)^{\geq 0} \quad \text{and} \quad i^! \cF \in \Shv(\siIw \backslash Z')^{\geq 0}.
 \]
 Now the functors $i_*=i_!$ and $j^*=j^!$ are clearly t-exact, so that $i^*$ and $j_!$ are right t-exact, and $i^!$ and $j_*$ are left t-exact. As a consequence, if $\cF$ belongs to one of the parts of the perverse t-structure on $\Shv(\siIw \backslash Z)$, then it satisfies the corresponding properties above. The converse implication also follows from these t-exactness properties, using the fiber sequences
 \[
  j_! j^! \cF \to \cF \to i_* i^* \cF \quad \text{and} \quad i_! i^! \cF \to \cF \to j_* j^! \cF
 \]
respectively.
\end{proof}

Using this remark one can obtain a parametrization of simple perverse sheaves on $\Fl$, similar to the one obtained by ad-hoc methods for $\Gr$ in Remark~\ref{rmk:simples-Gr}\eqref{it:classification-simples-Gr}.

\begin{cor}
\label{cor:classif-simples-Fl}
 For any $w \in W_\ext$, the image $\ICFl^\si_w$ of the unique (up to scalar) nonzero morphism $\DFl^\si_w \to \NFl^\si_w$ is simple, and is both the unique semisimple quotient of $\DFl^\si_w$ and the unique semisimple subobject of $\NFl_w^\si$. Moreover these simple objects are pairwise nonisomorphic, and any simple object in $\Shv(\siIw \backslash \Fl)$ is of this form. Finally, the functors
 \begin{gather*}
  \oblv^{\siIw | \siIwu} : \Shv(\siIw \backslash \Fl)^\heartsuit \to \Shv(\siIwu \backslash \Fl)^\heartsuit \\
  q^! : \Shv(\siIw \backslash \Fl)^\heartsuit \to \Shv(\siIw \backslash \tFl)^\heartsuit
 \end{gather*}
induce bijections on isomorphism classes of simple objects.
\end{cor}

\begin{proof}
 Fix $w \in W_\ext$. By the general theory of recollement (in particular,~\cite[Corollaire~1.4.26]{bbdg}) applied in the setting of Lemma~\ref{lem:recollement} for $Z=\overline{\fS_w}$ and $Z'=\overline{\fS_w} \smallsetminus \fS_w$, we obtain that in the category $\Shv(\siIw \backslash \overline{\fS_w})$, the image $\ICFl^\si_w$ of the canonical morphism $\DFl^\si_w \to \NFl^\si_w$ is simple, and that it is both the unique semisimple quotient of $\DFl^\si_w$ and the unique semisimple subobject of $\NFl_w^\si$. The first claim follows, since the full subcategory $\Shv(\siIw \backslash \overline{\fS_w}) \subset \Shv(\siIw \backslash \Fl)$ is stable under subquotients.
 
 The objects $(\ICFl^\si_w : w \in W_\ext)$ are pairwise nonisomorphic because they have distinct supports. They remain simple and pairwise nonisomorphic in $\Shv(\siIwu \backslash \Fl)$ by Lemma~\ref{lem:oblv-siIw}, and in $\Shv(\siIw \backslash \tFl)$ by Lemma~\ref{lem:pullback-exact}. Finally, any simple object in $\Shv(\siIw \backslash \Fl)$ or $\Shv(\siIwu \backslash \Fl)$ or $\Shv(\siIw \backslash \tFl)$ is of this form by Lemma~\ref{lem:heart-zero}.
\end{proof}

\begin{rmk}
\label{rmk:simples-locally-closed-piece}
 In the setting of Lemma~\ref{lem:recollement}, the simple objects in $\Shv(\siIw \backslash Z)^\heartsuit$ are the objects $\ICFl^\si_w$ with $w \in W_\ext$ such that $\fS_w \subset Z$, and those of $\Shv(\siIw \backslash V)^\heartsuit$ are the $*$-restrictions of those corresponding to elements $w$ such that $\fS_w \subset U$.
\end{rmk}

Some of the simple objects in $\Shv(\siIw \backslash \Fl)$ can be obtained by pullback from the simple objects in $\Shv(\siIw \backslash \Gr)$, as follows.

\begin{lem}
\label{lem:pullback-simples}
For any $\lambda \in \bY$, we have
\[
p^\dag \DGr^\si_\lambda \cong \ICFl^\si_{t_\lambda w_\circ}.
\]
\end{lem}

\begin{proof}
By Lemma~\ref{lem:pullback-exact} and Proposition~\ref{prop:si-simple}, we know that $p^\dag \DGr^\si_\lambda$ is a simple perverse sheaf. To prove the lemma, it therefore suffices to show that there exists a nonzero morphism $\DFl^\si_{t_\lambda w_\circ} \to p^\dag \DGr^\si_\lambda$. This follows from adjunction and Lemma~\ref{lem:si-pushpull}\eqref{it:si-push}.
%
\end{proof}

\section{Study of the dualizing complex}
\label{sec:dualizing}

\subsection{Statements}

As in Remark~\ref{rmk:perv-t-str}\eqref{it:tstr-degenerate} we consider the subcategory
\[
\Shv(\siIw \backslash \Gr)^{\leq -\infty} = \bigcap_{n \in \Z} \Shv(\siIw \backslash \Gr)^{\leq n}.
\]
(Objects in such a subcategory defined from a t-structure are sometimes called \emph{infinitely connective}, see e.g.~\cite{beraldo}.)

Our goal in this section is to prove the following statement.

\begin{prop}
\label{prop:dualizing-infty}
If $G$ is not a torus,
the object $\omega_{\Gr}$ belongs to $\Shv(\siIw \backslash \Gr)^{\leq -\infty}$.
\end{prop}

\begin{rmk}
\phantomsection
\label{rmk:dualizing-infty}
\begin{enumerate}
\item
By t-exactness of $p^\dag$ (see Lemma~\ref{lem:pullback-exact}) and $\oblv^{\siIw | \siIwu}$ (see~\S\ref{ss:perv-semiinf-def}) for sheaves on $\Gr$ and $\Fl$, Proposition~\ref{prop:dualizing-infty} implies similar statements for the dualizing complex of $\Fl$, and for $\infty$-categories of $\siIwu$-equivariant sheaves.
\item
\label{it:dualizing-vanishing-stalks}
Proposition~\ref{prop:dualizing-infty} implies that
\[
\Hom_{\Shv(\siIw \backslash \Gr)}(\oblv^{\siIw | \siIwu}(\omega_{\Gr}), \oblv^{\siIw | \siIwu}(\NGr_\mu)[n])=0
\]
for any $\mu \in \bY$ and $n \in \Z$, hence that $(\bi_\mu)^* \omega_{\Gr}=0$ for any $\mu \in \bY$. Similarly, we have $(\bi_w)^* \omega_{\Fl}=0$ for any $w \in W_\ext$.
\end{enumerate}
\end{rmk}

This proposition will be obtained as a corollary of the following statement, which we believe is of independent interest. Here we consider the $\infty$-categories
\[
\Shv(\siIw \backslash \Gr), \quad \Shv(\siIw \backslash \Fl)
\]
endowed with the perverse t-structure studied in Section~\ref{sec:perv-t-str-si}, and the $\infty$-categories
\[
\Shv(\Iw \backslash \Gr), \quad \Shv(\Iw \backslash \Fl)
\]
endowed with the usual perverse t-structures (see~\S\ref{ss:Iw-eq-sheaves}).

\begin{prop}
\label{prop:amplitude-Raskin}
The equivalences
\[
\Ras : \Shv(\siIw \backslash \Gr) \simto \Shv(\Iw \backslash \Gr), \quad \Ras : \Shv(\siIw \backslash \Fl) \simto \Shv(\Iw \backslash \Fl)
\]
(see~\S\ref{ss:raskin-equiv}) are of bounded cohomological amplitude with respect to the t-structures described above. More specifically, $\Ras$ is right t-exact and $\Ras[-\ell(w_\circ)]$ is left t-exact.
\end{prop}

\begin{rmk}
As mentioned in~\S\ref{ss:intro-Raskin-equiv}, a very similar statement, in the setting of (critically twisted) $D$-modules on $\Fl$ and without imposing $\Iw$-equivariance, appears in~\cite[\S 2]{fg}.
\end{rmk}

Proposition~\ref{prop:amplitude-Raskin} will be proved in~\S\ref{prop:amplitude-Raskin} below. The other ingredient for the proof of Proposition~\ref{prop:dualizing-infty} is the following claim.

\begin{prop}
\label{prop:dualizing-infty-Iw}
If $G$ is not a torus,
the object $\omega_{\Gr}$ belongs to $\bigcap_{n \in \Z} \Shv(\Gr)^{\leq n}$, where the subcategories of $\Shv(\Gr)$ considered here are associated with the standard perverse t-structure.
\end{prop}

\begin{proof}[Sketch of proof]
The analogue of this statement in the setting of $D$-modules is~\cite[Corollary~1.5.3]{beraldo}. It turns out that Beraldo's proof applies verbatim in our present sheaf-theoretic context. Namely, $\Gr$ admits a canonical open sub-ind-scheme $\Gr^\circ$ isomorphic to $\Loop^-G / G$, where $\Loop^-G$ is the ind-scheme which represents the functor $R \mapsto G(R[z^{-1}])$. Since $\Gr$ is covered by the translates (under the action of $(\Loop G)(\F)$) of this open subset, and since the perverse t-structure is Zariski local, the proposition will follow if we prove that for any smooth affine curve $\Sigma$ the ind-scheme $G[\Sigma]$ of maps from $\Sigma$ to $G$ has its dualizing complex in degree $-\infty$ for the perverse t-structure, which is Theorem~D in~\cite{beraldo}. 

The proof of this statement is given in Section~7 of~\cite{beraldo}. One notes that $G[\Sigma]$ is covered by translates of the open sub-ind-scheme $G[\Sigma]^{G_\circ\mhyphen\text{gen}}$ of maps which generically take values in the ``big cell'' $G_\circ \cong U^- \times T \times U$, which reduces as above the proof to that of the similar statement for $G[\Sigma]^{G_\circ\mhyphen\text{gen}}$. For this, Beraldo considers some variants of this ind-scheme denoted $G[\Sigma]^{G_\circ\mhyphen\text{gen}}_{\Sigma^{I,\text{disj}}}$ where $I$ is a nonempty finite set. As in~\cite[\S\S 7.2--7.3]{beraldo} one checks that the dualizing complex of $G[\Sigma]^{G_\circ\mhyphen\text{gen}}_{\Sigma^{I,\text{disj}}}$ is in degree $-\infty$ for the perverse t-structure. (The essential ingredient is~\cite[Lemma~7.2.10]{beraldo}, which holds also in our setting.) 

Finally, the claim about $G[\Sigma]^{G_\circ\mhyphen\text{gen}}$ is deduced in~\cite[\S 7.4]{beraldo} using some geometric constructions of Gaitsgory from~\cite{gaitsgory-contractibility}; here the essential ingredient is the fact that the variant of the Ran space of $\Sigma$ parametrizing finite subsets containing a fixed finite set is cohomologically contractible, which follows from the same arguments as for the analogous property of the usual Ran space, explained in~\cite[Appendix]{gaitsgory-contractibility}.
\end{proof}

We are now in a position to explain the end of the proof of Proposition~\ref{prop:dualizing-infty}.

\begin{proof}[Proof of Proposition~\ref{prop:dualizing-infty}]
From the construction of the equivalence $\Ras$ and the fact that $\omega_{\Gr}$ is $\Iw$-equivariant, one sees that $\Ras(\omega_\Gr) \cong \omega_{\Gr}$. 
Now, since the forgetful functor $\Shv(\Iw \backslash \Gr) \to \Shv(\Gr)$ is t-exact and conservative, Proposition~\ref{prop:dualizing-infty-Iw} implies that $\omega_{\Gr} \in \Shv(\Iw \backslash \Gr)$ is in degree $-\infty$ for the perverse t-structure. The desired claim follows, since
by Proposition~\ref{prop:dualizing-infty} the functor $\Ras^{-1}$ is right t-exact up to a shift.
\end{proof}

\subsection{Cohomological amplitude of Raskin's equivalences}

In this subsection we explain the proof of Proposition~\ref{prop:amplitude-Raskin}. We will only consider the case of sheaves on $\Fl$; the case of $\Gr$ can be deduced using the fact that the functor $p^\dag$ is t-exact and detects perversity (see Remark~\ref{rmk:functors-detect-perversity}), or proved along similar lines.

We will use the following easy claim.

\begin{lem}
\label{lem:conv-t-str}
Consider any t-structure on the $\infty$-category $\Shv(\Iw \backslash \Fl)$, and fix a simple reflection $s \in W_\ext$. If an object $\cF$ belongs to the ``$\leq 0$'' part of the t-structure, then the complex $\DFl_{s} \star^\Iw \cF$ belongs the ``$\leq 1$'' part of the t-structure iff the complex $\NFl_{s} \star^\Iw \cF$ belongs to the ``$\leq 1$'' part of the t-structure.
\end{lem}

\begin{proof}
Recall that we have exact sequences of perverse sheaves
\[
\IC_1 \hookrightarrow \DFl_{s} \twoheadrightarrow \IC_{s}, \quad \IC_{s} \hookrightarrow \nabla_{s} \twoheadrightarrow \IC_1.
\]
The desired claim follows, after convolving of the right with $\cF$.
\end{proof}

First, the fact that $\Ras$ is right t-exact immediately follows from Proposition~\ref{prop:ras-waki}, using the fact that this functor is continuous, that the Wakimoto sheaves are perverse (see~\S\ref{ss:Waki-Fl}), and the definition of the perverse t-structure on $\Shv(\siIw \backslash \Fl)$.

Then, in order to prove that $\Ras[-\ell(w_\circ)]$ is left t-exact, we have to prove that for any $\cF$ in $\Shv(\siIw \backslash \Fl)^{\ge 0}$, any $w \in W_\ext$ and any $n \in \Z_{>0}$ we have
\[
\Hom_{\Shv(\Iw \backslash \Fl)}(\DFl_w[n], \Ras(\cF)[-\ell(w_\circ)])=0,
\]
which is equivalent to proving that for any $w \in W_\ext$ the object $\Ras^{-1}(\DFl_w)[\ell(w_\circ)]$ belongs to $\Shv(\siIw \backslash \Fl)^{\leq 0}$, i.e.~that $\Ras^{-1}(\DFl_w)$ belongs to $\Shv(\siIw \backslash \Fl)^{\leq \ell(w_\circ)}$.

The main remaining ingredient of the proof is the following statement.

\begin{lem}
\label{lem:Ras-NFl-WFl}
For any $z \in W_\fin$, the functor
\[
\Ras^{-1} \left( \NFl_z \star^\Iw \Ras(-) \right)[\ell(z)]
\]
is right t-exact with respect to the perverse t-structure.
\end{lem}

\begin{proof}
Using Lemma~\ref{lem:standards-costandards} and induction on $\ell(z)$, it suffices to prove the lemma when $z=s_\alpha$ is the reflection associated with a simple root $\alpha$. Moreover, in view of Proposition~\ref{prop:ras-waki}, the desired statement amounts to the property that for any $y \in W_\ext$ we have $\Ras^{-1}(\NFl_{s_\alpha} \star^\Iw \WFl_y) \in \Shv(\siIw \backslash \Fl)^{\leq 1}$.

By standard arguments, one can assume that the quotient of $\bX$ by the root lattice is free. (This condition is equivalent to the property that the dual group $G^\vee_\bk$ has simply connected derived subgroup.) In this case, writing $y=t_\lambda x_\fin$ with $\lambda \in \bY$ and $x_\fin \in W_\fin$, we will proceed by induction on $|\langle \lambda, \alpha \rangle |$. We will use the following facts:
\begin{enumerate}
\item
\label{it:formula-Baff-1}
if $\langle \lambda, \alpha \rangle = 0$, then $\NFl_{s_\alpha} \star^{\Iw} \WFl_{t_\lambda} \cong \WFl_{t_\lambda} \star^\Iw \NFl_{s_\alpha}$;
\item
\label{it:formula-Baff-2}
if $\langle \lambda, \alpha \rangle = 1$, then we have $\DFl_{s_\alpha} \star^\Iw \WFl_{t_{s_\alpha(\lambda)}} \star^\Iw \DFl_{s_\alpha} \cong \WFl_{t_\lambda}$.
\end{enumerate}
(To justify these claims, one observes that Lemma~\ref{lem:standards-costandards} implies that the assignment $w \mapsto \DFl_w$ extends to a morphism from the braid group associated with $W_\ext$ to the group of isomorphism classes of invertible objects in $\Shv(\Iw \backslash \Fl)$. By construction the Wakimoto objects $(\WFl_{t_\lambda} : \lambda \in \bY)$ are the images under this map of the Bernstein elements in this braid group, see e.g.~\cite[\S 2.6]{lusztig-affine}. Then our claims follow from standard identities in this braid group, established in~\cite[Lemma~2.7]{lusztig-affine}.)

To start the induction we consider the case when $\langle \lambda, \alpha \rangle = 0$. Here, for any $x_\fin \in W_\fin$ we have
\[
\NFl_{s_\alpha} \star^\Iw \WFl_{t_\lambda x_\fin} \cong \NFl_{s_\alpha} \star^\Iw \WFl_{t_\lambda} \star^\Iw \DFl_{x_\fin} \cong \WFl_{t_\lambda} \star^\Iw \NFl_{s_\alpha} \star^\Iw \DFl_{x_\fin},
\]
where we have used~\eqref{it:formula-Baff-1} above.
By Lemma~\ref{lem:standards-costandards}\eqref{it:standards-costandards-3} the complex $\NFl_{s_\alpha} \star^\Iw \DFl_{x_\fin}$ is a perverse sheaf supported on $G/B \subset \Fl$; in particular it belongs to the full subcategory of $\Shv(\Iw \backslash \Fl)$ generated under extensions and colimits by the objects $\DFl_{z}$ with $z \in W_\fin$. Using once again Proposition~\ref{prop:ras-waki} we deduce that $\Ras^{-1}(\NFl_{s_\alpha} \star^\Iw \WFl_{t_\lambda x_\fin}) \in \Shv(\siIw \backslash \Fl)^{\leq 0}$ in this case.

Next we fix some $\lambda \in \bY$ such that $\langle \lambda, \alpha \rangle \neq 0$, and assume that the desired claim is known for any element $\mu \in \bY$ such that $|\langle \mu, \alpha \rangle | < |\langle \lambda, \alpha \rangle |$. Assume first that $\langle \lambda, \alpha \rangle > 0$. Our assumption above implies that one can write $\lambda = \lambda_1 + \lambda_2$ with $\langle \lambda_1, \alpha \rangle = 1$. Then if $x_\fin \in W_\fin$ we have
\[
\NFl_{s_\alpha} \star^\Iw \WFl_{t_\lambda x_\fin} \cong \NFl_{s_\alpha} \star^\Iw \WFl_{t_{\lambda_1}} \star^\Iw \WFl_{t_{\lambda_2} x_\fin} \cong \WFl_{t_{s_\alpha(\lambda_1)}} \star^\Iw \DFl_{s_\alpha} \star^\Iw \WFl_{t_{\lambda_2} x_\fin},
\]
where we have used~\eqref{it:formula-Baff-2} above. By induction we know that $\Ras^{-1}(\NFl_{s_\alpha} \star^\Iw \WFl_{t_{\lambda_2} x_\fin})$ belongs to $\Shv(\siIw \backslash \Fl)^{\leq 1}$. 
Using Lemma~\ref{lem:conv-t-str}
we deduce that $\Ras^{-1}(\DFl_{s_\alpha} \star^\Iw \WFl_{t_{\lambda_2} x_\fin}) \in \Shv(\siIw \backslash \Fl)^{\leq 1}$, and then that $\Ras^{-1}(\NFl_{s_\alpha} \star^\Iw \WFl_{t_{\lambda} x_\fin}) \in \Shv(\siIw \backslash \Fl)^{\leq 1}$.

Finally we consider the case when $\langle \lambda, \alpha \rangle < 0$. Here we write $\lambda = \lambda_1 + \lambda_2$ with $\langle \lambda_1, \alpha \rangle = -1$. Then we have
\[
\DFl_{s_\alpha} \star^\Iw \WFl_{t_\lambda x_\fin} \cong \DFl_{s_\alpha} \star^\Iw \WFl_{t_{\lambda_1}} \star^\Iw \WFl_{t_{\lambda_2} x_\fin} \cong \WFl_{t_{s_\alpha(\lambda_1)}} \star^\Iw \NFl_{s_\alpha} \star^\Iw \WFl_{t_{\lambda_2} x_\fin},
\]
so that using induction one sees that the image under $\Ras^{-1}$ of this object belongs to $\Shv(\siIw \backslash \Fl)^{\leq 1}$. Using again Lemma~\ref{lem:conv-t-str} one deduces that $\Ras^{-1}(\NFl_{s_\alpha} \star^\Iw \WFl_{t_\lambda x_\fin}) \in \Shv(\siIw \backslash \Fl)^{\leq 1}$, which concludes the proof.
\end{proof}

One can now conclude the proof of Proposition~\ref{prop:amplitude-Raskin} as follows. It is a standard fact that the maximal length representatives of the cosets in $W_\fin \backslash W_\ext$ are of the form $t_\lambda x$ with $\lambda \in \bY_+$ and $x \in W_\fin$. (This can e.g.~be deduced from the description of minimal coset representatives in $W_\ext / W_\fin$ recalled in~\S\ref{ss:Weyl}.) For such elements the Wakimoto sheaf coincides with the associated standard perverse sheaf, see Lemma~\ref{lem:properties-Waki}\eqref{it:Waki-2}. Using also Lemma~\ref{lem:standards-costandards}, one deduces that for any $w \in W_\ext$ there exist $z \in W_\fin$ and $y \in W_\ext$ such that $\Delta_w \cong \NFl_z \star^\Iw \WFl_y$. Then we have
\[
\Ras^{-1}(\DFl_w) \cong \Ras^{-1}(\NFl_z \star^\Iw \WFl_y) \cong \Ras^{-1}(\NFl_z \star^\Iw \Ras(\DFl^\si_y))
\]
by Proposition~\ref{prop:ras-waki}, so that
the desired claim follows from Lemma~\ref{lem:Ras-NFl-WFl}.

\section{The Gaitsgory sheaf}
\label{sec:Gaits-sheaf}

\subsection{Construction}
\label{ss:construction}

Recall the objects $(\cM_\lambda : \lambda \in \bY^+)$ considered in~\S\ref{ss:some-perv-sheaves}. 
Recall also from Remark~\ref{rmk:Satake-standard-fil}\eqref{it:exactness-convolution} that we have an identification
\[
 \cM_\lambda \cong \Ras^{-1}(\WGr_{-\lambda} \star^{\Loop^+ G} \cI_*(\lambda)).
\]

Consider the preorder $\unlhd$ on $\bY$ defined by
\[
 \lambda \unlhd \mu \quad \text{iff} \quad \mu-\lambda \in \bY_+.
\]
If $\mu,\lambda \in \bY_+$ satisfy $\mu \unlhd \lambda$, we will now explain how to define a canonical map
\[
\theta_{\mu,\lambda}: \cM_\mu \to \cM_\lambda
\]
in $\Shv(\siIw \backslash \Gr)^{\heartsuit}$.

The construction will use the following ingredients. First, if $\nu \in \bY_+$ we have $\WGr_{\nu} = \DGr_{\nu}$. Since $\cI_*(\nu)_{|\Gr_\nu} = \underline{\bk}_{\Gr_\nu}[\langle \nu, 2\rho \rangle]$, by adjunction there exists a canonical map
\begin{equation}
\label{eqn:hw-arrow}
\WGr_{\nu} \to \cI_*(\nu).
\end{equation}
Next, if $\nu,\nu' \in \bY_+$, it is a standard fact that
\[
\bigl( \cI_*(\nu) \star^{\Loop^+G} \cI_*(\nu') \bigr)_{| \Gr^{\nu+\nu'}} \cong \underline{\bk}_{\Gr^{\nu+\nu'}}[\langle \nu+\nu', 2\rho \rangle].
\]
By adjunction, we deduce that there exists a canonical map
\begin{equation}
 \label{eqn:map-I*-sum}
 \cI_*(\nu) \star^{\Loop^+G} \cI_*(\nu') \to \cI_*(\nu+\nu').
\end{equation}

Now we come back to the setting above.
We need to construct a map
\[
\WGr_{-\mu} \star^{\Loop^+G} \cI_*(\mu) \to \WGr_{-\lambda} \star^{\Loop^+G} \cI_*(\lambda).
\]
Applying the equivalence $\Waki_{t_\lambda} \star^{\Iw} ({-})$ and using~\eqref{eqn:Waki-convolution-Gr}, this is equivalent to constructing a map
\[
\WGr_{\lambda - \mu} \star^{\Loop^+G} \cI_*(\mu) \to \cI_*(\lambda).
\]
We take this map to be the composition
\[
\WGr_{\lambda - \mu} \star^{\Loop^+G} \cI_*(\mu) \to \cI_*(\lambda-\mu) \star^{\Loop^+G} \cI_*(\mu) \to \cI_*(\lambda),
\]
where the first map is obtained from~\eqref{eqn:hw-arrow} (for the dominant weight $\lambda-\mu$) by right convolution with $\cI_*(\mu)$,
and the second one is~\eqref{eqn:map-I*-sum} applied to the dominant weights $\lambda-\mu$ and $\mu$.

The collection of maps $\theta_{\mu,\lambda}: \cM_\mu \to \cM_\lambda$ makes $(\cM_{\lambda})_{\lambda \in \bY_+}$ a directed system in the abelian category $\Shv(\siIw \backslash \Gr)^\heartsuit$ with respect to $(\bY_+, \unlhd)$. 
Since the mapping spaces between objects in the heart of a t-structure on a stable $\infty$-category are discrete, this system canonically lifts to a functor
\[
 \bY_+ \to \Shv(\siIw \backslash \Gr),
\]
so that we can define the \emph{Gaitsgory sheaf} to be the object
\[
\Gaits := \colim_{\lambda \in \bY_+} \cM_\lambda
\]
(where the colimit is taken with respect to the preorder $\unlhd$).
By Remark~\ref{rmk:perv-t-str}\eqref{it:perv-colim}, this object belongs to the heart of the perverse t-structure.

Let us note the following property for later use.

\begin{lem}
\label{lem:mlt-inj}
If $\mu \unlhd \lambda$, the transition map $\cM_\mu \to \cM_\lambda$ is injective.
\end{lem}

\begin{proof}
It is clear that for any $\nu \in \bY_+$ the map $\theta_{0,\nu}$ is nonzero; in view of Corollary~\ref{cor:mlambda-soc}, this map therefore identifies with the embedding of the socle of $\cM_\nu$. Since $\theta_{\mu,\lambda} \circ \theta_{0,\mu} = \theta_{0,\lambda}$, we deduce the claim.
\end{proof}

\subsection{Main properties}

The following statement describes the stalks and costalks of $\Gaits$.

\begin{thm}
\phantomsection
\label{thm:properties-Ga}
\begin{enumerate}
\item 
\label{it:gaits-stalk}
The object $\Gaits$ is supported on $\overline{\rS_0}$, and its
stalks are given for $\nu \in \bY$ by
\[
\rank \left( \pH^n((\bi_\nu)^*\Gaits) \right) = 
\begin{cases}
\cP(-\nu,1) & \text{if $n = 0$,} \\
0 & \text{otherwise.}
\end{cases}
\]
\item 
\label{it:gaits-h01}
For $\nu \in \bY$, we have
\[
\rank \left( \pH^0((\bi_\nu)^!\Gaits) \right) =
\begin{cases}
1 & \text{if $\nu = 0$,} \\
0 & \text{otherwise,}
\end{cases}
\qquad
\rank \left( \pH^1((\bi_\nu)^!\Gaits) \right) = 0.
\]
\item 
\label{it:gaits-costalk}
Assume that $\chr(\bk)$ is good for $G$. For $\nu \in \bY$, the corestriction of $\Gaits$ to $\rS_\nu$ is given by
\[
\sum_n \rank \left( \pH^n((\bi_\nu)^!\Gaits) \right) \cdot q^{n} = \cP(-\nu,q^2).
\]
\end{enumerate}
\end{thm}

\begin{proof}
As explained in Lemma~\ref{lem:mlambda-stalk}\eqref{it:mlambda-filt}, each object $\cM_\lambda$ is supported on $\overline{\rS_0}$; the same is therefore true for $\Gaits$, so that all the statements 
are obvious unless $\nu \preceq 0$.

Fix now $\nu \in \bY$ such that $\nu \preceq 0$, and a finite union $Z$ of closures $\overline{\rS_\eta}$ such that $Z \subset \overline{\rS_0}$ and $V:=\overline{\rS_0} \smallsetminus Z$ is a finite union of $\siIw$-orbits containing $\rS_\nu$.
We can consider the objects $\Gaits$ and $\cM_\lambda$ as objects of $\Shv(\siIw \backslash \overline{\rS_0})^\heartsuit$, and the restriction and corestriction of $\Gaits$ to $\rS_\nu$ clearly only depend on the object
\[
(\Gaits)_{|V} \cong \colim_\lambda \cM_\lambda{}_{|V}.
\]

We claim that there exists $\lambda \in \bY_+$ such that for any $\mu,\mu' \in \bY_+$ such that $\lambda \unlhd \mu \unlhd \mu'$ the transition morphism
\[
 (\theta_{\mu,\mu'})_{|V}: \cM_\mu{}_{|V} \to \cM_{\mu'}{}_{|V}
\]
is an isomorphism. In fact, by Lemma~\ref{lem:mlt-inj} and t-exactness of restriction to $V$ these morphisms are injective. We will prove that the objects have the same finite length, which will imply the claim.

By (the analogue for $\Gr$ of) Lemma~\ref{lem:recollement} and the general theory of recollement, we know that the restriction functor $\Shv(\siIw \backslash \overline{\rS_0})^\heartsuit \to \Shv(\siIw \backslash V)^\heartsuit$ sends simple objects either to $0$ or to simple objects. Since each $\cM_\mu$ has finite length (see Corollary~\ref{cor:mlambda-soc}), it follows that the same is true for each $\cM_\mu{}_{|V}$. Moreover, as in Remark~\ref{rmk:simples-locally-closed-piece} the simple objects in $\Shv(\siIw \backslash V)^\heartsuit$ are the restrictions of the objects $\DGr^\si_\eta$ where $\eta$ runs over the (finitely many) elements $\eta \in \bY$ such that $\rS_\eta \subset V$. The multiplicity of the restriction of $\DGr^\si_\eta$ in an object of $\Shv(\siIw \backslash V)^\heartsuit$ is equal to the rank of the image of the latter object under the functor $(\bi_\eta)^*$. For the objects $\cM_\mu{}_{|V}$, this rank has been computed in Lemma~\ref{lem:mlambda-stalk}\eqref{it:mlambda-stalk}; to conclude it therefore suffices to prove that the integer given by this formula is independent of $\mu$ for all $\eta$ as above, provided $\mu$ is sufficiently large with respect to $\unlhd$. However, for any fixed $\eta$, if $\mu$ is sufficiently large we have
\[
 w(\mu + \rho^\vee) - (\mu + \eta + \rho^\vee) \not \succeq 0 \quad \text{for all $w \in W_\fin \smallsetminus \{0\}$.}
\]
(Of course, this property is closely related to the facts stated in Remark~\ref{rmk:multiplicity-canonicity}.)
Therefore, for such $\mu$, in the sum appearing in Lemma~\ref{lem:mlambda-stalk}\eqref{it:mlambda-stalk}, only the term corresponding to $w=e$ contributes. Since this contribution does not depend on $\mu$, this finishes the proof of our claim.



The statements now follow from this claim, its proof, and Lemma~\ref{lem:mlambda-stalk}.
\end{proof}

\begin{rmk}
\label{rmk:multiplicity-canonicity-2}
 In view of Remark~\ref{rmk:multiplicity-canonicity}, for any $\nu \in \bY$ the image under the equivalence from~\eqref{eqn:si-sheaves-orbits} of the object $(\bi_\nu)^* \Gaits$ identifies with $\mathscr{O}(U^{\vee,+}_\bk)_\nu$.
For characteristic-$0$ coefficients this statement is~\cite[Proposition~2.5.4]{gaitsgory}. In this setting Gaitsgory also provides a canonical description of the cofibers of $\Gaits$ in~\cite[Proposition~2.5.2]{gaitsgory}. We do not establish this description for general coefficients here, although we expect that it still holds in good characteristic.
\end{rmk}

We establish additional properties of $\Gaits$ in the following statement.

\begin{prop}
 \phantomsection
 \label{prop:properties-Gaits}
 \begin{enumerate}
  \item 
  \label{it:gaits-nabla}
 We have $\Gaits \cong \pH^0(\NGr^\si_0)$.
  \item
  \label{it:Gaits-injective}
  The object $\oblv^{\siIw | \siIwu}(\Gaits) \in \Shv(\siIwu \backslash \Gr)^{\heartsuit}$ has a unique simple subobject, isomorphic to $\oblv^{\siIw|\siIwu}(\DGr^{\si}_0)$, and it is injective.
 \end{enumerate}
\end{prop}

\begin{proof}
\eqref{it:gaits-nabla}
By (the analogue for $\Gr$ of) Lemma~\ref{lem:recollement} applied to $\overline{\rS_0}$ and $\overline{\rS_0} \smallsetminus \rS_0$, and the general theory of recollement (see in particular~\cite[comments following Corollaire~1.4.24]{bbdg}), $\pH^0(\NGr^\si_0)$ is uniquely characterized in $\Shv(\siIw \backslash \overline{\rS_0})^\heartsuit$ by the following properties:
\begin{itemize}
\item $\pH^0(\NGr^\si_0)_{|\rS_0} \cong \omega_{\rS_0}[-\psdim (\rS_0)]$;
\item if $i : \overline{\rS_0} \smallsetminus \rS_0 \to \overline{\rS_0}$ is the embedding, $i^! \pH^0(\NGr^\si_0)$ belongs to $\Shv(\siIw \backslash \Gr)^{\geq 2}$.
\end{itemize}
The object $\Gaits$ satisfies these properties, by definition and~Theorem~\ref{thm:properties-Ga}\eqref{it:gaits-h01}, so it must be isomorphic to $\pH^0(\NGr^\si_0)$.

\eqref{it:Gaits-injective}
 It follows from Theorem~\ref{thm:properties-Ga}\eqref{it:gaits-h01} and adjunction that for any $\nu \in \bY$ we have
  \begin{equation}
 \label{eqn:Hom-simple-Gaits}
  \Hom_{\Shv(\siIwu \backslash \Gr)^{\heartsuit}}(\oblv^{\siIw | \siIwu}(\DGr^\si_\nu), \oblv^{\siIw | \siIwu}(\Gaits))= \begin{cases}
\bk & \text{if $\nu=0$;} \\
0 & \text{otherwise}
\end{cases}
 \end{equation}
 and
 \begin{equation}
 \label{eqn:Ext1-vanishing-Gaits}
  \Ext^1_{\Shv(\siIwu \backslash \Gr)^{\heartsuit}}(\oblv^{\siIw | \siIwu}(\DGr^\si_\nu), \oblv^{\siIw | \siIwu}(\Gaits))=0.
 \end{equation}
 
 Here~\eqref{eqn:Hom-simple-Gaits} immediately implies the first claim, in view of Remark~\ref{rmk:simples-Gr}\eqref{it:classification-simples-Gr}. We will show that~\eqref{eqn:Ext1-vanishing-Gaits} implies the second one.
  Consider an injective morphism $X \hookrightarrow Y$ in $\Shv(\siIwu \backslash \Gr)^{\heartsuit}$. What we have to prove is that any morphism $g : X \to \oblv^{\siIw | \siIwu}(\Gaits)$ factors through $Y$. For that we consider the set of pairs $(Z,h)$ where $Z$ is a subobject of $Y$ containing $X$ and $h : Z \to \oblv^{\siIw | \siIwu}(\Gaits)$ is a morphism such that $h_{|X}=g$, which we endow with the order given by $(Z,h) \leq (Z',h')$ if $Z \subset Z'$ and $h'_{|Z}=h$. By Zorn's lemma this poset admits a maximal element $(Z,h)$. But if $Z \neq Y$, then by Lemma~\ref{lem:heart-zero}\eqref{it:heart-zero-4} there exists a nonzero morphism $\oblv^{\siIw | \siIwu}(\DGr^\si_\nu) \to Y/Z$. By Remark~\ref{rmk:simples-Gr}\eqref{it:classification-simples-Gr} the domain of this map is simple, so the morphism must be injective. If $Z'$ is the preimage in $Y$ of the image of $\oblv^{\siIw | \siIwu}(\DGr^\si_\nu)$, then the exact sequence
 \begin{multline*}
  \Hom(Z',\oblv^{\siIw | \siIwu}(\Gaits)) \to \Hom(Z,\oblv^{\siIw | \siIwu}(\Gaits)) \\
  \to \Ext^1(\oblv^{\siIw | \siIwu}(\DGr^\si_\nu),\oblv^{\siIw | \siIwu}(\Gaits))
 \end{multline*}
and~\eqref{eqn:Ext1-vanishing-Gaits} imply that $h$ extends to $Z'$, contradicting the maximality of $(Z,h)$. We conclude that $Z=Y$, which finishes the proof.
\end{proof}

\begin{rmk}
\label{rmk:Gaits-injective-siIw}
 Since the restriction of $\oblv^{\siIw | \siIwu}$ to $\Shv(\siIw \backslash \Gr)^{\heartsuit}$ is fully faithful (see Lemma~\ref{lem:oblv-siIw}), Proposition~\ref{prop:properties-Gaits}\eqref{it:Gaits-injective} also implies that the object $\Gaits$ is injective in $\Shv(\siIw \backslash \Gr)^{\heartsuit}$.
\end{rmk}

%

\subsection{Another description of the Gaitsgory sheaf}

The definition of $\Gaits$ in~\cite{gaitsgory} (for characteristic-$0$ coefficients) is different from ours. In this subsection we show that the two definitions are equivalent (for general coefficients). Namely, for $\lambda \in \bY_+$ we consider the shifted perverse sheaf
\[
 z^{-\lambda} \cdot \cI_*(\lambda) [\langle \lambda, 2\rho \rangle] \quad \in \Shv(T \backslash \Gr).
\]
(Here, as in the proof of Lemma~\ref{lem:Ras-twist}, we denote by $z^{-\lambda} \cdot (-)$ the pushforward functor under the automorphism of $\Gr$ given by left multiplication by the element $z^{-\lambda} \in (\Loop T)(\F)$.) For $\lambda,\mu \in \bY_+$ such that $\lambda \unlhd \mu$, there exists a canonical morphism
\begin{equation}
\label{eqn:morphism-shifts-I*}
 z^{-\lambda} \cdot \cI_*(\lambda) [\langle \lambda, 2\rho \rangle] \to z^{-\mu} \cdot \cI_*(\mu) [\langle \mu, 2\rho \rangle]
\end{equation}
defined as follows. 

First, for any $\nu \in \bY_+$, as explained in the construction of~\eqref{eqn:hw-arrow} there exists a canonical isomorphism
\[
 (i_\nu)^! \cI_*(\nu) \cong \underline{\bk}_{\Gr_\nu}[\langle \nu,2\rho \rangle].
\]
Denoting by $k_\nu$ the embedding of $z^\nu \Loop^+ G$ in $\Gr$, we deduce a canonical isomorphism
\[
 (k_\nu)^! \cI_*(\nu) \cong \bk[-\langle \nu,2\rho \rangle].
\]
Now we have
\[
(i_0)^! (z^{-\nu} \cdot \cI_*(\nu) [\langle \nu, 2\rho \rangle]) \cong (k_\nu)^! \cI_*(\nu) [\langle \nu, 2\rho \rangle],
\]
hence this identification provides by adjunction a canonical morphism
\begin{equation}
\label{eqn:morphism-shifts-I*-2}
\delta_{\Gr} \to z^{-\nu} \cdot \cI_*(\nu) [\langle \nu, 2\rho \rangle].
\end{equation}

Now we come back to the construction of the morphism~\eqref{eqn:morphism-shifts-I*}.
Fix $\lambda,\mu \in \bY_+$ such that $\lambda \unlhd \mu$. Consider the morphism~\eqref{eqn:morphism-shifts-I*-2} for $\nu = \mu-\lambda$, and then convolve on the right with $\cI_*(\lambda)$, and apply $z^{-\lambda} \cdot (-) [\langle \lambda, 2\rho \rangle]$; we obtain a morphism
\[
 z^{-\lambda} \cdot \cI_*(\lambda) [\langle \lambda, 2\rho \rangle] \to  z^{-\mu} \cdot \bigl( \cI_*(\mu-\lambda) \star^{\Loop^+ G} \cI_*(\lambda) \bigr) [\langle \mu, 2\rho \rangle].
\]
Composing this morphism with the morphism induced by~\eqref{eqn:map-I*-sum} for the weights $\mu-\lambda$ and $\lambda$, we deduce the wished-for morphism~\eqref{eqn:morphism-shifts-I*}.

Gaitsgory explains in~\cite[\S 2.7]{gaitsgory} that the datum of the objects $z^{-\lambda} \cdot \cI_*(\lambda) [\langle \lambda, 2\rho \rangle]$ and the morphisms~\eqref{eqn:morphism-shifts-I*-2} upgrades to a functor from (the $\infty$-category associated with) the directed set $(\bY_+, \unlhd)$ to $\Shv(T \backslash \Gr)$. His arguments apply to general coefficients. Once this fact is established, one can consider the colimit
\[
\colim_{\lambda \in \bY_+} z^{-\lambda} \cdot \cI_*(\lambda) [\langle \lambda, 2\rho \rangle] \quad \in \Shv(T \backslash \Gr).
\]

\begin{prop}
\label{prop:Gaits-colim}
There exists a canonical isomorphism
\[
\Gaits \cong \colim_{\lambda \in \bY_+} z^{-\lambda} \cdot \cI_*(\lambda) [\langle \lambda, 2\rho \rangle].
\]
\end{prop}

\begin{proof}
First, the same arguments as in~\cite[\S 2.4.1]{gaitsgory} show that the colimit in the right-hand side belongs to $\Shv(\siIw \backslash \Gr)$. This object therefore coincides with its image under $\Ras^{-1} \circ \Ras$.
From the definition of $\Ras$ (see~\S\ref{ss:raskin-equiv}) we obtain that
\[
\Ras \left( \colim_{\lambda \in \bY_+} z^{-\lambda} \cdot \cI_*(\lambda) [\langle \lambda, 2\rho \rangle] \right) \cong \colim_{\lambda \in \bY_+} \Av^{\Iw | \Loop^+ T}_*(z^{-\lambda} \cdot \cI_*(\lambda)) [\langle \lambda, 2\rho \rangle].
\]
Now that for any $\lambda \in \bY_+$ we have
\[
\Av^{\Iw | \Loop^+ T}_*(z^{-\lambda} \cdot \cI_*(\lambda)) \cong \bigl( (i_{-\lambda})_* \underline{\bk}_{\Gr_{-\lambda}} \bigr) \star^{\Loop^+ G} \cI_*(\lambda).
\]
Since $\ell(t_{-\lambda})=\langle \lambda,2\rho \rangle$, we deduce that 
\begin{multline*}
\Av^{\Iw | \Loop^+ T}_*(z^{-\lambda} \cdot \cI_*(\lambda)) [\langle \lambda, 2\rho \rangle] \cong \NFl_{t_{-\lambda}} \star^\Iw \oblv^{\Loop^+ G | \Iw}(\cI_*(\lambda)) \\
\cong \WFl_{t_{-\lambda}} \star^\Iw \oblv^{\Loop^+ G | \Iw}(\cI_*(\lambda)) \cong \Ras(\cM_\lambda).
\end{multline*}
Since $\Ras^{-1}$ commutes with colimits we deduce an isomorphism
\[
\colim_{\lambda \in \bY_+} z^{-\lambda} \cdot \cI_*(\lambda) [\langle \lambda, 2\rho \rangle] \cong \colim_{\lambda \in \bY_+} \cM_\lambda.
\]

One can check that if $\lambda,\mu \in \bY_+$ satisfy $\lambda \unlhd \mu$ the image under the functor $\Ras^{-1} \circ \Av^{\Iw | \Loop^+ T}_*$ of the morphism~\eqref{eqn:morphism-shifts-I*} coincides with the morphism $\theta_{\lambda,\mu}$ of~\S\ref{ss:construction}, which finishes the proof. (This property can be checked at the level of homotopy categories, since these are morphisms between objects in the heart of a t-structure.)
\end{proof}



\begin{thebibliography}{99}

\bibitem{adr}
P.~Achar, G.~Dhillon, and S.~Riche, \emph{Modular intersection cohomology of Drinfeld's compactifications}, preprint~\href{https://arxiv.org/abs/2505.17953}{arXiv:2505.17953}.

\bibitem{modrap2}
P.~Achar and S.~Riche, \emph{Modular perverse sheaves on flag varieties, II: Koszul duality and formality}, Duke Math. J. \textbf{165} (2016), no. 1, 161--215.

\bibitem{ar-steinberg}
P.~Achar and S.~Riche, \emph{A geometric Steinberg formula},
Transform. Groups \textbf{28} (2023), no. 3, 1001--1032.

\bibitem{ar-model}
P.~Achar and S.~Riche,
\emph{A geometric model for blocks of Frobenius kernels},
Ark. Mat. \textbf{62} (2024), no.~2, 217--329.

\bibitem{central}
P.~Achar and S.~Riche, \emph{Central sheaves on affine flag varieties}, 
Panor. Synth\`eses~64, Soc. Math. France, 2025.

\bibitem{arider}
P.~Achar and L.~Rider, {\em Parity sheaves on the affine Grassmannian and the
  Mirkovi\'{c}--Vilonen conjecture}, Acta Math. {\bf 215} (2015), 183--216.
  
  
  \bibitem{alwy}
J.~Ansch\"utz, J.~Louren{\c c}o, Z.~Wu, and J.~Yu, \emph{Gaitsgory's central functor and the Arkhipov--Bezrukavnikov equivalence in mixed characteristic}, preprint~\href{https://arxiv.org/abs/2311.04043}{arXiv:2311.04043}.

\bibitem{agkrrv}
D.~Arinkin, D.~Gaitsgory, D.~Kazhdan, S.~Raskin, N.~Rozenblyum, and Y.~Varshavsky, \emph{The stack of local systems with restricted variation and geometric Langlands theory with nilpotent singular support}, preprint~\href{arXiv:2010.01906}{https://arxiv.org/abs/2010.01906}.

\bibitem{abbgm}
S.~Arkhipov, A.~Braverman, R.~Bezrukavnikov, D.~Gaitsgory, and I.~Mirkovi{\'c},
\emph{Modules over the small quantum group and semi-infinite flag manifold},
Transform. Groups \textbf{10} (2005), no. 3--4, 279--362.


\bibitem{br}
P.~Baumann and S.~Riche, \emph{Notes on the geometric Satake equivalence}, 
in \emph{Relative Aspects in Representation Theory, Langlands Functoriality and Automorphic Forms, CIRM Jean-Morlet Chair, Spring 2016} (V. Heiermann, D. Prasad, Eds.), 1--134, Lecture Notes in Math. 2221, Springer, 2018.

\bibitem{bbdg}
A.~Be{\u\i}linson, J.~Bernstein, P.~Deligne, and O.~Gabber,
\emph{Faisceaux pervers}, in \emph{Analysis and topology on singular spaces, I (Luminy, 1981)}, 5--171,
Ast\'erisque \textbf{100}, Soc. Math. France, Paris, 1982.
  
\bibitem{bbm}
A.~Be{\u\i}linson, R.~Bezrukavnikov, and I.~Mirkovi{\'c},
\emph{Tilting exercises},
Mosc. Math. J. \textbf{4} (2004), no. 3, 547--557, 782.


\bibitem{beraldo}
D.~Beraldo, \emph{Tempered D-modules and Borel-Moore homology vanishing},
Invent. Math. \textbf{225} (2021), no. 2, 453--528. 

\bibitem{bl}
R.~Bezrukavnikov and Q.~Lin, \emph{Highest weight modules at the critical level and noncommutative Springer resolution}, in \emph{Algebraic groups and quantum groups}, 15--27,
Contemp. Math. 565, Amer. Math. Soc., Providence, RI, 2012. 


\bibitem{bezr2}
R.~Bezrukavnikov and S.~Riche, \emph{On two modular geometric realizations of an affine Hecke algebra}, preprint~\href{https://arxiv.org/abs/2402.08281}{arXiv:2402.08281}.

\bibitem{bouthier}
A.~Bouthier, \emph{Cohomologie \'etale des espaces d'arcs}, preprint~\href{https://arxiv.org/abs/1509.02203}{arXiv:1509.02203}.

\bibitem{bkv}
A.~Bouthier, D.~Kazhdan, and Y.~Varshavsky, \emph{Perverse sheaves on infinite-dimensional stacks, and affine Springer theory},
Adv. Math. \textbf{408} (2022), part A, Paper No.~108572, 132 pp.

\bibitem{bfgm}
A.~Braverman, M.~Finkelberg, D.~Gaitsgory, and I.~Mirkovi{\'c},
\emph{Intersection cohomology of Drinfeld's compactifications},
Selecta Math. (N.S.) 8 (2002), no. 3, 381--418. 



\bibitem{cvdhs}
R.~Cass, T.~van den Hove, and J.~Scholbach, \emph{Central motives on parahoric flag varieties}, preprint~\href{https://arxiv.org/abs/2403.11007}{arXiv:2403.11007}.

\bibitem{cd}
H.~Chen and G.~Dhillon, \emph{A Langlands dual realization of coherent sheaves on the nilpotent cone}, preprint~\href{https://arxiv.org/abs/2310.10539}{arXiv:2310.10539}.

\bibitem{dlyz}
G.~Dhillon, Y.~W.~Li, Z.~Yun, and X.~Zhu, \emph{Endoscopy for metaplectic affine Hecke categories}, preprint~\href{https://arxiv.org/abs/2507.16667}{arXiv:2507.16667}.

\bibitem{dl}
G.~Dhillon and S.~Lysenko,
\emph{Semi-infinite parabolic IC-sheaf},
preprint~\href{https://arxiv.org/abs/2310.06386}{arXiv:2310.06386}.



\bibitem{ff}
B.~Fe{\u\i}gin and E.~Frenkel, \emph{Affine Kac--Moody algebras and semi-infinite flag manifolds},
Comm. Math. Phys. \textbf{128} (1990), no. 1, 161--189.
  
\bibitem{fm}
M.~Finkelberg and I.~Mirkovi{\'c}, \emph{Semi-infinite flags. I. Case of global curve $\mathbb{P}^1$}, in \emph{Differential topology, infinite-dimensional Lie algebras, and applications}, 81--112,
Amer. Math. Soc. Transl. Ser. 2, 194, Adv. Math. Sci. 44, Amer. Math. Soc., Providence, RI, 1999.

\bibitem{fg}
E.~Frenkel and D.~Gaitsgory, \emph{$D$-modules on the affine flag variety and representations of affine Kac--Moody algebras},
Represent. Theory \textbf{13} (2009), 470--608.

\bibitem{gaitsgory-contractibility}
D.~Gaitsgory, \emph{Contractibility of the space of rational maps}, Invent. Math. \textbf{191} (2013), no.~1, 91--196.
  
\bibitem{gaitsgory}
D.~Gaitsgory, {\em The semi-infinite intersection cohomology sheaf}, Adv. Math.
  {\bf 327} (2018), 789--868. (Corrected version available as preprint~\href{arXiv:1703.04199v6}{https://arxiv.org/abs/1703.04199v6}.)
  
\bibitem{gaitsgory-ext}
D.~Gaitsgory, \emph{A conjectural extension of the Kazhdan--Lusztig equivalence},
Publ. Res. Inst. Math. Sci. \textbf{57} (2021), no. 3-4, 1227--1376.
  
  
\bibitem{gr}
D.~Gaitsgory and N.~Rozenblyum, {\em A study in derived algebraic geometry. Vol. I. Correspondences and duality}, Math. Surveys Monogr. 221,
American Mathematical Society, Providence, RI, 2017.



  
\bibitem{kato}
S.~Kato, {\em Spherical functions and a $q$-analogue of Kostant's weight
  multiplicity formula}, Invent. Math. {\bf 66} (1982), 461--468.

\bibitem{kato-ICM}
S.~Kato, \emph{The formal model of semi-infinite flag manifolds}, in \emph{ICM---International Congress of Mathematicians. Vol. 3. Sections 1--4}, 1600–1622, EMS Press, Berlin, 2023.


  
\bibitem{lz}
Y.~Liu and W.~Zheng, \emph{Enhanced six operations and base change theorem for higher Artin stacks}, preprint~\href{https://arxiv.org/abs/1211.5948}{arXiv:1211.5948}.

\bibitem{loregian}
F.~Loregian, \emph{t-structures on stable $\infty$-categories}, preprint~\href{https://arxiv.org/abs/2005.14295}{arXiv:2005.14295}.

\bibitem{losev}
I.~Losev, \emph{On modular Soergel bimodules, Harish-Chandra bimodules, and category O}, preprint~\href{https://arxiv.org/abs/2302.05782}{arXiv:2302.05782}.

\bibitem{losev2}
I.~Losev, \emph{Quantum category O vs affine Hecke category}, preprint~\href{https://arxiv.org/abs/2310.03153}{arXiv:2310.03153}.

\bibitem{lurie-htt}
J.~Lurie, \emph{Higher topos theory},
Annals of Mathematics Studies 170, Princeton University Press, Princeton, NJ, 2009.

\bibitem{lurie-ha}
J.~Lurie, \emph{Higher algebra}, book available at~\url{https://people.math.harvard.edu/~lurie/papers/HA.pdf}.

\bibitem{lusztig-generic}
G.~Lusztig,
\emph{Hecke algebras and Jantzen's generic decomposition patterns},
Adv. in Math. \textbf{37} (1980), no. 2, 121--164. 

\bibitem{lusztig}
G.~Lusztig, {\em Singularities, character formulas, and a $q$-analogue of
  weight multiplicities}, in \emph{Analysis and topology on singular spaces, II, III
  (Luminy, 1981)}, Ast\'erisque~\textbf{101--102}, Soc. Math. France, Paris, 1983,
  pp.~208--229.
  
\bibitem{lusztig-affine}
G.~Lusztig, \emph{Affine Hecke algebras and their graded version},
J. Amer. Math. Soc. \textbf{2} (1989), no. 3, 599--635. 
  
\bibitem{lusztig-ICM}
G.~Lusztig, \emph{Intersection cohomology methods in representation theory}, in \emph{Proceedings of the International Congress of Mathematicians, Vol. I, II (Kyoto, 1990)}, 155--174, Math. Soc. Japan, Tokyo, 1991.
  
\bibitem{mr1}
C.~Mautner and S.~Riche, \emph{On the exotic t-structure in positive characteristic},
Int. Math. Res. Not. IMRN \textbf{2016}, no.~18, 5727--5774. 

\bibitem{mr}
C.~Mautner and S.~Riche, {\em Exotic tilting sheaves, parity sheaves on affine
  Grassmannians, and the Mirkovi{\'c}--Vilonen conjecture}, J. Eur. Math. Soc.
  (JEMS) {\bf 20} (2018), 2259--2332.

\bibitem{mv}
I.~Mirkovi{\'c} and K.~Vilonen, {\em Geometric {L}anglands duality and representations of algebraic groups over commutative rings}, Ann. of Math. {\bf 166} (2007), 95--143.


\bibitem{raskin}
S.~Raskin, \emph{Chiral principal series categories II: the factorizable Whittaker category}, preprint available from~\url{https://gauss.math.yale.edu/~sr2532/}.

\bibitem{raskin-ws}
S.~Raskin, \emph{Second adjointness for loop groups}, notes for a talk at the \emph{Winter school on local geometric Langlands theory} in Paris (2018), available at~\url{https://lysenko.perso.math.cnrs.fr/Notes_talks_winter2018/Ja-1(Raskin).pdf}.

\bibitem{richarz}
T.~Richarz, \emph{Basics on affine Grassmannians}, available at~\url{https://www.mathematik.tu-darmstadt.de/media/algebra/homepages/richarz/Notes_on_affine_Grassmannians.pdf}.

\bibitem{rs}
T.~Richarz and J.~Scholbach, \emph{The intersection motive of the moduli stack of shtukas},
Forum Math. Sigma \textbf{8} (2020), Paper No. e8, 99 pp. (Corrigendum available at~\url{https://www.mathematik.tu-darmstadt.de/media/algebra/homepages/richarz/corrigendum.pdf}.)

\bibitem{rsi}
S.~Riche and Q.~Situ,
\emph{Equivariant Koszul duality, modular category $\mathcal{O}$, and periodic Kazhdan--Lusztig polynomials}, preprint~\href{https://arxiv.org/abs/2511.18518}{arXiv:2511.18518}.

\bibitem{rw}
S.~Riche and G.~Williamson,
\emph{A simple character formula}
Ann. H. Lebesgue \textbf{4} (2021), 503--535.

\bibitem{scholze}
P.~Scholze, \emph{Six-functors formalisms}, notes available at~\url{https://people.mpim-bonn.mpg.de/scholze/SixFunctors.pdf}.

\bibitem{soergel}
W.~Soergel, {\em Kazhdan--Lusztig polynomials and a combinatoric for tilting
  modules}, Represent. Theory {\bf 1} (1997), 83--114.
  



\bibitem{zhu}
X.~Zhu, \emph{Tame categorical local Langlands correspondence}, preprint~\href{https://arxiv.org/abs/2504.07482}{arXiv:2504.07482}.

\end{thebibliography}
\end{document}